\documentclass[reqno,12pt,letterpaper]{amsart}
\RequirePackage{amsmath,amssymb,amsthm,graphicx,mathrsfs,url}
\RequirePackage[usenames,dvipsnames]{xcolor}
\RequirePackage[colorlinks=true,linkcolor=ForestGreen,citecolor=MidnightBlue]{hyperref}
\RequirePackage{amsxtra}
\usepackage{enumitem}
\usepackage{cases}
\usepackage{matlab-prettifier}

\usepackage{afterpage}
\usepackage{placeins}

\numberwithin{equation}{section}

\usepackage{placeins}
\usepackage{mathtools}
\usepackage{tikz-cd}
\usepackage{dsfont}
\usepackage{subcaption}

\newcommand{\set}[1]{\left\{#1\right\}}
\newcommand{\n}[1]{\left\|#1\right\|}

\newtheorem{lemma}{Lemma}[section]
\newtheorem{proposition}[lemma]{Proposition}
\newtheorem{theorem}{Theorem}

\theoremstyle{definition}
\newtheorem{definition}[lemma]{Definition}
\theoremstyle{remark}
\newtheorem{remark}[lemma]{Remark}
\newtheorem{hypothesis}{Hypothesis}

\DeclareMathOperator{\ext}{ext}

\DeclareMathOperator{\rest}{rest}
\DeclareMathOperator{\re}{Re}
\DeclareMathOperator{\im}{Im}

\DeclareMathOperator{\sign}{sign}
\DeclareMathOperator{\pv}{p.v.}
\DeclareMathOperator{\WF}{WF}
\DeclareMathOperator{\DN}{DN}
\DeclareMathOperator{\D}{D}

\let\oldtocsection=\tocsection
\let\oldtocsubsection=\tocsubsection
\renewcommand{\tocsection}[2]{\hspace{0em}\oldtocsection{#1}{#2}}
\renewcommand{\tocsubsection}[2]{\hspace{1em}\oldtocsubsection{#1}{#2}}

\title{Orr--Sommerfeld equation and complex deformation}
\author{Malo J\'ez\'equel}
\email{malo.jezequel@math.cnrs.fr}
\address{CNRS, Univ. Brest, UMR6205, Laboratoire de Math{\'e}matiques de Bretagne Atlantique, France}

\author{Jian Wang}
\email{wangjian@ihes.fr}
\address{Institut des Hautes {\'E}tudes Scientifiques, Bures-sur-Yvette, France}

\begin{document}

\thispagestyle{empty}

\begin{abstract}
    For shear flows in a 2D channel, we define resonances near regular values of the shear profile for the Rayleigh equation under an analyticity assumption. This is done via complex deformation of the interval on which the Rayleigh equation is considered. We show such resonances are inviscid limits of the eigenvalues of the corresponding Orr--Sommerfeld equation. 
\end{abstract}

\maketitle
{
\hypersetup{linkcolor=NavyBlue}
\tableofcontents
}
\thispagestyle{empty}

\newpage

\section{Introduction}

\captionsetup[subfigure]{labelformat=parens, labelsep=space}
\renewcommand{\thesubfigure}{\Alph{subfigure}}

\begin{figure}[b]
   \centering
   \begin{subfigure}{0.45\textwidth}
       \centering
       \includegraphics[width=\textwidth]{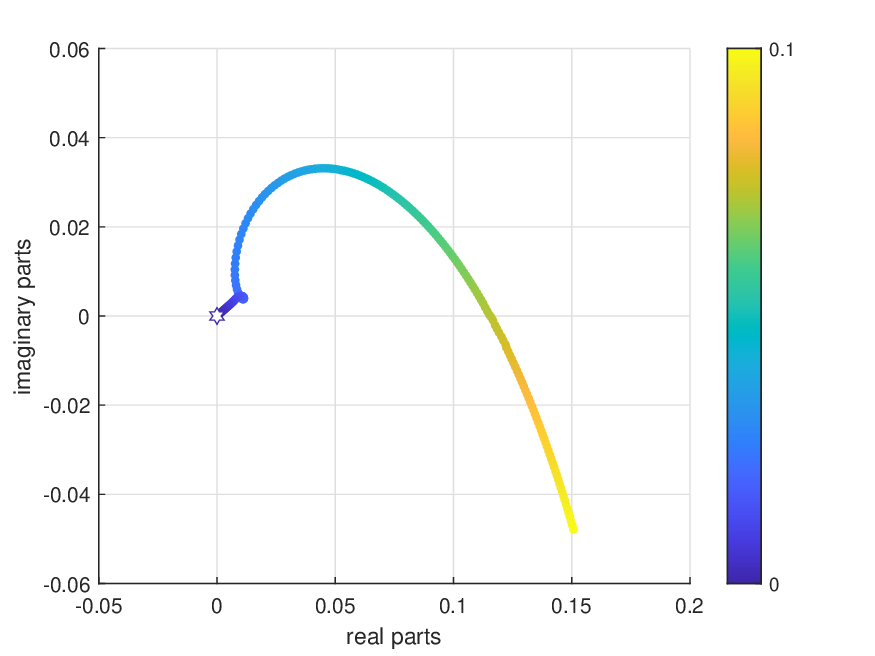}
       \caption{Viscous perturbations of a resonance}
       \label{fig:cos_unstable}
   \end{subfigure}
   \begin{subfigure}{0.45\textwidth}
       \centering
       \includegraphics[width=\textwidth]{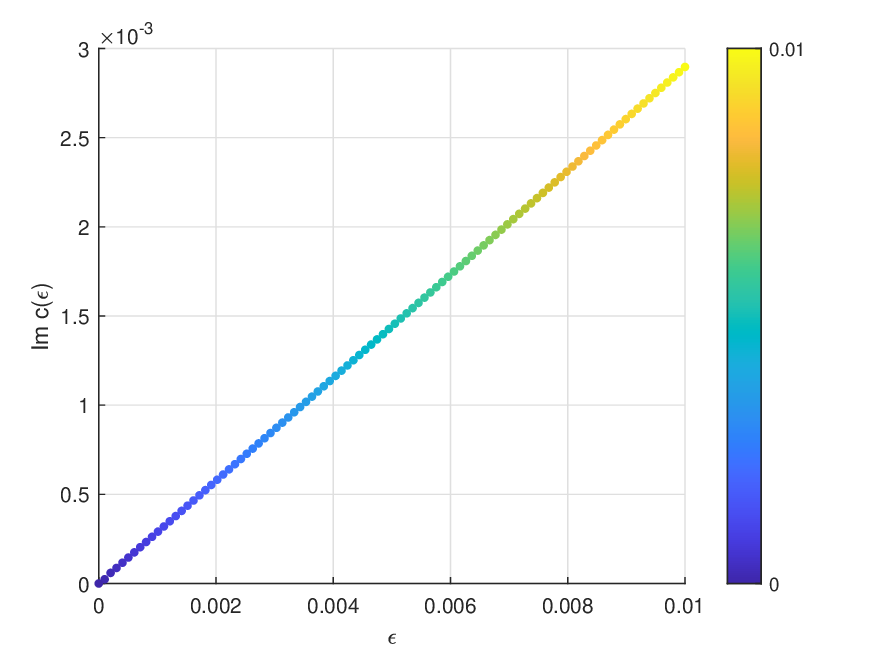}
       \caption{Imaginary parts of $c(\epsilon)$}
       \label{fig:cos_unstable_im}
   \end{subfigure}
   \caption{Shear profile $U(x)=\cos(0.7\pi x)$, $x\in [-1,1]$ with $\alpha=\frac{\sqrt{6}\pi}{5}$. This shear flow has a simple resonance $c_0=0$ for this choice of $\alpha$ (see Theorem \ref{theorem:limit} for the notion of resonance and \S \ref{subsection:example_segment} for a discussion of this example). (A) Numerical computation of the viscous perturbations $c(\epsilon)$ of the resonance (snowflake) for $\epsilon:=R^{-\frac12}\in [0, 0.1]$ predicted by Theorem~\ref{theorem:limit}. For small $\epsilon$, numerics suggest that the resonance becomes an unstable eigenvalue (which is confirmed in \S \ref{subsection:example_segment}). 
   (B) Imaginary parts of $c(\epsilon)$, for $\epsilon\in [0, 0.01]$.
   Both figures are colored according to $\epsilon$. The method used for computation is described in Appendix \ref{section:matlab}.}
    \label{fig:cos}
\end{figure}

Let $a < b $ be real numbers and $U$ be a $C^\infty$ function from $[a,b]$ to $\mathbb{R}$. Let $\alpha > 0$. For $c \in \mathbb{C}$ and $R > 0$, let us consider the Orr--Sommerfeld equation
\begin{equation}\label{eq:orr_sommerfeld}
   \frac{i}{\alpha R} (\partial_x^2 - \alpha^2)^2 \psi + (U- c )(\partial_x^2 - \alpha^2) \psi - U'' \psi = 0
\end{equation}
on $[a,b]$. This is an equation of order $4$, and we will consequently supplement it with boundary conditions
\begin{equation*}
    \psi(a) = \psi(b) = \partial_x \psi(a) = \partial_x \psi(b) = 0.
\end{equation*}
This equation appears in fluid mechanics when studying the stability of the shear profile $(0, U(x))$ as a stationary solution of Navier--Stokes equation (maybe including an external force) in the channel $[a,b] \times \mathbb{R}$ with Reynolds number $R$. Indeed, linearizing Navier--Stokes equation near the shear profile, taking a Fourier transform in time and in the second spatial variable and finally rewriting the equation in term of the stream function of the flow, we end up with \eqref{eq:orr_sommerfeld}, see for instance \cite[\S 3.1.2]{schmid_henningson_book}. It is then natural to ask for which value of $c$ the equation \eqref{eq:orr_sommerfeld} admits a non-trivial solution $\psi$. The sign of the imaginary part of such a $c$ is of particular interest, as a solution with $\im c > 0$ corresponds to an unstable (i.e. exponentially growing in time) solution of the linearized Navier--Stokes equation.

The main result of this paper (Theorems \ref{theorem:limit} and \ref{theorem:resonances} below) is a description when $R$ goes to $+ \infty$ of the set of $c$ for which \eqref{eq:orr_sommerfeld} has a non-trivial solution, in the neighbourhood of certain real parameters. To do so, we are naturally led to consider Rayleigh equation on $[a,b]$:
\begin{equation}\label{eq:rayleigh}
    (U- c)(\partial_x^2 - \alpha^2) \psi - U'' \psi = 0
\end{equation}
supplemented with the boundary condition $\psi(a) = \psi(b) = 0$. We are mostly interested in the situation in which $c$ is closed to a point $c_0$ that belongs to the range of $U$. Notice that at $c = c_0$ the equation \eqref{eq:rayleigh} is not elliptic, which makes the study of this equation considerably more complicated. 

We will bypass this difficulty using a method based on complex deformation that is greatly inspired by the complex scaling method in scattering theory (see \cite[\S 2.7, \S 4.5, \S 4.7]{dyatlov_zworski_book} and references therein). The main idea is very simple: we will replace the segment $[a,b]$ by a curve obtained as a small perturbation of $[a,b]$ in the complex plane on which the equation \eqref{eq:rayleigh} becomes elliptic. Naturally, to do so we need $U$ to be analytic near the places where we want to deform $[a,b]$. The existence of such a deformation of $[a,b]$ is then guaranteed when $c_0$ is a regular value of $U$ (and not a boundary value of $U$, since we want to let the extremity of $[a,b]$ fixed). Let us point out that the complex deformation must be chosen carefully in order to ensure the convergence of the solutions to Orr--Sommerfeld equation to solutions of Rayleigh equation when $R$ goes to $+ \infty$.

We need then to understand how the Orr--Sommerfeld equation degenerates into the Rayleigh equation on our deformation of $[a,b]$. Even if Rayleigh equation is elliptic on this deformation, this is not a simple question as Orr--Sommerfeld equation is of order $4$ while Rayleigh equation is only of order $2$. The difficulty mostly appears when dealing with the extra boundary condition that is required for Orr--Sommerfeld equation. We tackle this issue by using a 1D version of the Vishik--Lyusternik method \cite{vishik_lyusternik_OG}.

\subsection{Main results}

Since we want to study the asymptotic $R \to +\infty$, we will introduce the new variable $\epsilon = {R}^{-\frac12}$. Hence, for $c \in \mathbb{C}$ and $\epsilon > 0$, we introduce the operator
\begin{equation*}
P_{c,\epsilon} = i\alpha^{-1} \epsilon^2(\partial_x^2 - \alpha^2)^2 + (U - c) (\partial_x^2 - \alpha^2) - U''.
\end{equation*}
For $\epsilon > 0$, let $\Sigma_{\epsilon}$ be the set of points $c\in \mathbb C$ such that 
$$P_{c,\epsilon} : H^4(a,b) \cap H^2_0(a,b) \to L^2(a,b)$$ 
is not invertible.

\begin{remark}\label{remark:sigmaepsilon_as_spectrum}
We will see (Proposition \ref{proposition:elliptic_eigenvalues}) that $\Sigma_{\epsilon}$ is a discrete subset of $\mathbb{C}$. Moreover, one may define a notion of multiplicity for the elements of $\Sigma_{\epsilon}$. Consider the inverse of the operator $\partial_x^2 - \alpha^2 : H^2(a,b) \cap H_0^1(a,b) \to L^2(a,b)$. One may check that $(\partial_x^2 - \alpha^2)^{-1}$ is given by the formula:
\begin{equation*}
\begin{split}
    (\partial_x^2 - \alpha^2)^{-1} f(x) = & - \frac{\sinh(\alpha (x-a))}{\alpha \sinh( \alpha(b-a))} \int_x^b \sinh(\alpha(b - t)) f(t) \mathrm{d}t \\ & - \frac{\sinh(\alpha(b-x))}{\alpha \sinh( \alpha(b-a))} \int_a^x \sinh(\alpha (t-a)) f(t)\mathrm{d}t.
\end{split}
\end{equation*}
Introduce then for $c \in \mathbb{C}$ and $\epsilon > 0$ the operator
\begin{equation*}
Q_{c,\epsilon} = P_{c,\epsilon} (\partial_x^2 - \alpha^2)^{-1} = i\alpha^{-1} \epsilon^2 (\partial_x^2 - \alpha^2) + (U-c) - U'' (\partial_x^2 - \alpha^2)^{-1} = Q_{0,\epsilon} - c.
\end{equation*}
Then $\Sigma_{\epsilon}$ is the spectrum of the unbounded operator $Q_{0,\epsilon}$ on $L^2(a,b)$ with domain $\set{ u \in H^2(a,b) : ((\partial_x^2 - \alpha^2)^{-1} u)'(a) = ((\partial_x^2 - \alpha^2)^{-1} u)'(b) = 0}$. Since $Q_{0,\epsilon}$ is elliptic, the spectrum $\Sigma_{\epsilon}$ is made of isolated eigenvalues with finite (algebraic) multiplicities. We define in this way a notion of multiplicity for the elements of $\Sigma_{\epsilon}$. An equivalent definition is given in \S \ref{subsection:basic_properties}.
\end{remark}

We will relate the asymptotic of $\Sigma_{\epsilon}$ as $\epsilon$ goes to $0$ to the operator 
\begin{equation*}
P_{c,0} = (U- c) (\partial_x^2 - \alpha^2) - U''.
\end{equation*}
We will also consider the associated operator
\begin{equation*}
Q_{c,0} = U- c - U'' (\partial_x^2 - \alpha^2)^{-1}.
\end{equation*}

\begin{hypothesis}\label{hypothesis}
We will study $\Sigma_{\epsilon}$ near a point $c_0 \in \mathbb{R}$ under the following hypotheses:
\begin{itemize}
\item $c_0$ is a regular value of $U$, i.e. if $x \in [a,b]$ is such that $U(x) = c_0$ then $U'(x) \neq 0$,
\item $c_0$ is not a boundary value of $U$, i.e. $c_0 \notin \set{U(a), U(b)}$,
\item $U$ is analytic on a neighbourhood of $U^{-1}(\set{c_0})$.
\end{itemize}
\end{hypothesis}

Our first result is then the following:

\begin{theorem}\label{theorem:limit}
Let $c_0 \in \mathbb{R}$. Assume that $c_0$ satisfies hypothesis \ref{hypothesis}. Then there is $\delta > 0$ and a discrete subset $\mathcal{R}$ of $(c_0 - \delta,c_0 + \delta) + i (- \delta, + \infty)$ such that for every $c \in (c_0 - \delta,c_0 + \delta) + i (- \delta, + \infty)$:
\begin{enumerate}[label=(\roman*)]
\item if $c$ does not belong to $\mathcal{R}$ then there are $\epsilon_0,\nu > 0$ such that for every $\epsilon \in (0,\epsilon_0)$ there is no element of $\Sigma_\epsilon$ in the disk $\mathbb{D}(c,\nu)$ of center $c$ and radius $\nu$;\label{item:no_resonance}
\item if $c \in \mathcal{R}$, then there is an integer\footnote{We will say that $d$ is the multiplicity of $c$ as an element of $\mathcal{R}$.} $d \geq 1$ and $\epsilon_0,\nu > 0$ such that for every $\epsilon \in (0,\epsilon_0)$ there are precisely $d$ elements $c_1(\epsilon),\dots,c_d(\epsilon)$ of $\Sigma_{\epsilon}$, counted with multiplicities, in $\mathbb{D}(c,\nu)$. Moreover, for $j = 1,\dots,d$ we have $c_j(\epsilon) \underset{\epsilon \to 0}{\to} c$. In addition, for every symmetric polynomial $f \in \mathbb{C}[X_1,\dots,X_d]$ the map $\epsilon \mapsto f(c_1(\epsilon),\dots,c_d(\epsilon))$ has a Taylor expansion at $\epsilon = 0$.\label{item:expansion_resonance}
\end{enumerate}
\end{theorem}

The definition of $\mathcal R$ is given in \S \ref{subsection:definition_resonances}.
By analogy with \cite{dyatlov_zworski_viscosity,galkowski_zworski_viscosity}, we will sometimes refer to the elements of $\mathcal{R}$ as ``resonances'' for Rayleigh equation. A similar theorem holds for the Orr--Sommerfeld equation on the circle $x\in \mathbb R/L \mathbb Z$, $L>0$. The situation is actually slightly simpler, since only even powers of $\epsilon$ appear in the asymptotic expansions for the symmetric functions of $c_1(\epsilon),\dots,c_d(\epsilon)$. See \S \ref{subsection:circle_case} for details.

\begin{remark}
If $g$ is a function from an interval $(0,\epsilon_0)$ with $\epsilon_0 > 0$ to a Banach space $\mathcal{B}$, we say that $g$ has a Taylor expansion at $\epsilon = 0$, if there is a sequence $(g_k)_{k \geq 0}$ of elements of $\mathcal{B}$ such that for every integer $N \geq 0$ we have
\begin{equation*}
g(\epsilon) \underset{\epsilon \to 0}{=} \sum_{k = 0}^N \epsilon^k g_k + \mathcal{O}(\epsilon^{N+1}).
\end{equation*}
\end{remark}

\begin{remark}
The regularity property for $c_1(\epsilon),\dots,c_d(\epsilon)$ at $\epsilon = 0$ in Theorem \ref{theorem:limit} just describes the standard way a multiple eigenvalue (eventually) splits into several eigenvalues. Equivalently, we can state that the polynomials
\begin{equation*}
\prod_{j = 1}^d (X - c_j(\epsilon))
\end{equation*}
has a Taylor expansion at $\epsilon = 0$. Notice that if $d = 1$, then we have a single element $c_1(\epsilon)$ of $\Sigma_\epsilon$, with an asymptotic expansion
\begin{equation*}
    c_1(\epsilon) \underset{\epsilon \to 0}{\simeq} \sum_{k \geq 0} a_k \epsilon^k \ \text{ with } \ a_0 = c_1.
\end{equation*}
A formula for the coefficient $a_1$ is given in Proposition \ref{proposition:first_order_perturbation}.
\end{remark}

\begin{remark}\label{remark:rates}
One could be surprised by the fact that this is an asymptotic expansion in power of $\epsilon$ that appears in Theorem \ref{theorem:limit}, as, reintroducing the Reynolds number $R$, it corresponds to powers of $R^{-1/2}$ (instead of the original $R^{-1}$). This is due to the use of boundary layers in the proof of Theorem \ref{theorem:limit}. These are very specific solutions to the equation $P_{c,\epsilon} \psi = 0$, that behave near the boundary points of $[a,b]$ like $x \mapsto e^{ - \frac{x}{\epsilon}}$ near $0^+$. These boundary layers are central in the Vishik--Lyusternik method that we use to prove Theorem \ref{theorem:limit}, see \S \ref{subsection:LVmethod}.
In the circle case however, since there is no boundary condition to deal with, we do not need to work with boundary layers and, consequently, we get an asymptotic expansion in even powers of $\epsilon$ in Theorem \ref{theorem:limit} in that case, see \S \ref{subsection:circle_case}. See also Figure \ref{fig:sin3x}.
\end{remark}

Our second result allows to identify resonances with non-negative imaginary parts. To do so, we will use the following definition. 

\begin{definition}\label{definition:resonant_states}
For $c \in \mathbb{R}$, we introduce a set of distributions $\Omega(c)$ on $(a,b)$. We say that a distribution $\psi \in \mathcal{D}'(a,b)$ belongs to $\Omega(c)$ if
\begin{itemize}
\item the restriction of $\psi$ to $(a,b) \setminus U^{-1}(\set{c})$ is $C^\infty$ and extends as a $C^\infty$ function to $[a,b] \setminus U^{-1}(\set{c})$ with $\psi(a) = \psi(b) = 0$;
\item the wave front set of $\psi$ is contained in $$\{ (x,\xi) \in (a,b) \times \mathbb{R} : U(x) = c, U'(x) \xi < 0\};$$
\item there is an integer $n \geq 0$ such that $Q_{c,0}^n P_{c,0} \psi = 0$.
\end{itemize}
\end{definition}

With this definition, we can state:
\begin{theorem}\label{theorem:resonances}
Under the assumption of Theorem \ref{theorem:limit}, and up to making $\delta$ smaller:
\begin{enumerate}[label=(\roman*)]
\item if $c \in (c_0 - \delta, c_0 + \delta) + i (0, + \infty)$ then $c \in \mathcal{R}$ if and only if $c$ is an eigenvalue of $Q_{0,0}$ on $L^2$ (i.e. if $P_{c,0} : H^2(a,b) \cap H_0^1(a,b) \to L^2(a,b)$ is not invertible) and its multiplicities as an element of $\mathcal{R}$ and as an eigenvalue of $Q_{0,0}$ coincide.\label{item:resonances_easy}
\item if $c \in (c_0 - \delta,c_0 + \delta)$, then $\Omega(c)$ is finite dimensional and the multiplicity of $c$ as an element of $\mathcal{R}$ is the dimension\footnote{In particular, $c \notin \mathcal{R}$ if and only if $\Omega(c) = \set{0}$.} of $\Omega(c)$.\label{item:resonances_real}
\end{enumerate}
\end{theorem}

\begin{remark}
Let us comment on the definition of $\Omega(c)$. The point is to select certain solutions to Rayleigh equation that might not be smooth on $U^{-1}(\set{c})$ but still satisfies some regularity property, expressed in terms of wave front set \cite[Definition 8.1.2]{hormander_book1}. A more explicit description of the behaviour of the elements of $\Omega(c)$ near $U^{-1}(\set{c})$ is given in \S \ref{subsection:description_resonant_states}. We also explain in \S \ref{subsection:generalized_eigenvalue} that relation between the space $\Omega(c)$ and the notion of generalized embedded eigenvalue of Wei, Zhang and Zhao \cite{generalized_eigenvalue_segment}. By analogy with \cite{dyatlov_zworski_viscosity, galkowski_zworski_viscosity}, we will sometimes call the elements of $\Omega(c)$ ``resonant states''.

Finally, we ask for $Q_{c,0}^n P_{c,0} \psi = 0$ for some $n \geq 0$ (which makes sense using the integral formula for $(\partial_x^2 - \alpha^2)^{-1}$) instead of just $P_{c,0} \psi = 0$. We do so in order to take into account the potential presence of a Jordan block at $c$ for the operator $Q_{0,0}$ (on some space that will appear in the proof of Theorem \ref{theorem:limit}).
\end{remark}

\begin{remark}
As mentioned above, it is particularly interesting to know if there are elements in $\Sigma_\epsilon$ with positive imaginary parts. Indeed, such an element of $\Sigma_\epsilon$ yields a solution to the linearized Navier--Stokes equation that grows exponentially fast. The most interesting case is consequently when there is no element of $\mathcal{R}$ with positive imaginary part, but there is a real $c \in \mathcal{R}$. One can then wonder whether the resonance $c$ is going to be deformed into an element of $\Sigma_\epsilon$ with positive or non-positive imaginary part. This is determined by the first term with non-zero imaginary part in the asymptotic expansion from Theorem \ref{theorem:limit} (assuming for instance that $c$ is simple, and that there is such a term). In Figures \ref{fig:cos} and \ref{fig:cos_07pi_045}, we report numerically computed viscous perturbations of resonances for the flow $U(x)=\cos(0.7\pi x)$, $x\in [-1,1]$. In both cases, resonances with non-positive imaginary parts seem to be deformed into unstable eigenvalues in $\Sigma_{\epsilon}$ for some small $\epsilon$. The numerical method we used for these computations are described in Appendix \ref{section:matlab}.
\end{remark}

\begin{figure}[t]
   \centering
   \begin{subfigure}{0.45\textwidth}
       \centering
       \includegraphics[width=\textwidth]{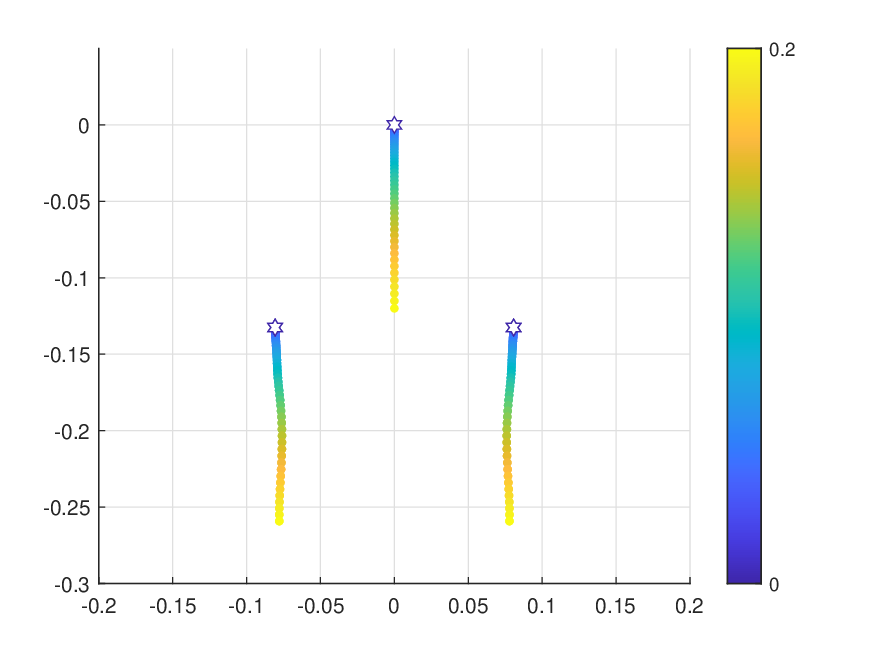}
       \caption{Viscous perturbations of resonances}
       \label{fig:sinusoid_circle}
   \end{subfigure}
   \begin{subfigure}{0.45\textwidth}
       \centering
       \includegraphics[width=\textwidth]{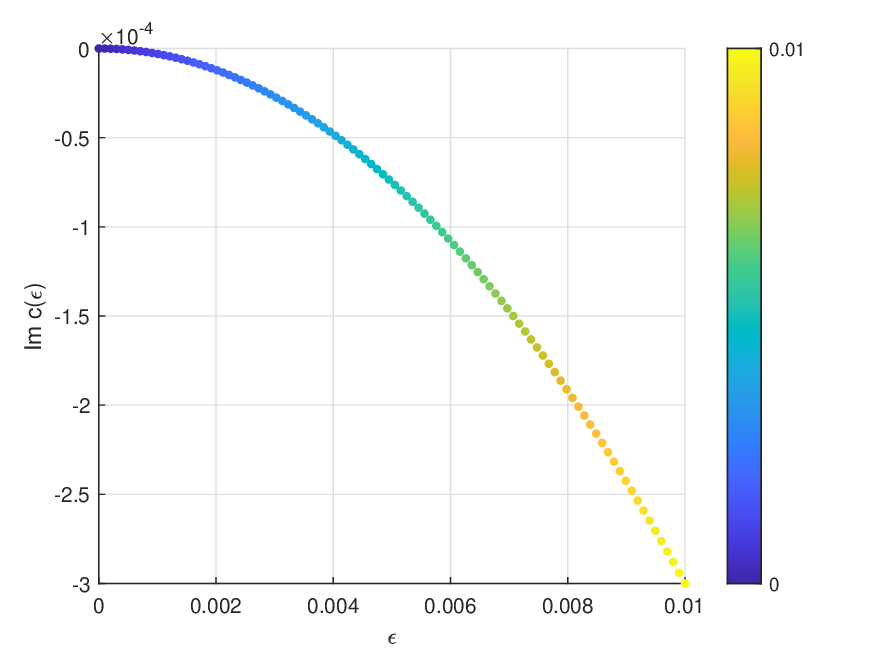}
       \caption{Imaginary parts of $c(\epsilon)$}
       \label{fig:sinusoid_circle_im}
   \end{subfigure}
   \caption{Shear profile $U(x)=\sin(3x)$, $x\in \mathbb R/2\pi\mathbb Z$, $\alpha=3$. For this flow and this choice of $\alpha$, $c_0=0$ is a simple resonance (see \S \ref{subsection:example_circle}). There are also two numerically computed resonances near $0$. (A) Numerical computation of the viscous perturbations of the resonances (snowflakes) with $\epsilon\in [0,0.2]$.
   (B) Imaginary parts of $c(\epsilon)$, the viscous perturbation of $0$, for $\epsilon\in [0, 0.01]$. Both figures are colored according to $\epsilon$.}
   \label{fig:sin3x}
\end{figure}

\subsection{Context}

There is a vast literature dedicated to the study of linear stability of shear flows, going back at least to the work of Kelvin, Reynolds, Rayleigh, Orr, Sommerfeld and Heisenberg. One may look at \cite{orr_part1,orr_part2,heisenberg} and references therein for a glimpse of the discussions on this topic back then. A textbook presentation of this subject may be found in \cite{lin_hydrodynamic_stability,drazin_reid,schmid_henningson_book}. The topic is also discussed in the survey \cite{Yud_03}, including some open problems.

Several authors investigated the way solutions to Orr--Sommerfeld equation behave as the Reynolds number goes to $+ \infty$ (one may refer for instance to the exposition in the book of Lin \cite[Chapter 8]{lin_hydrodynamic_stability}). This problem was a motivation for the development of the linear turning point theory, see the book of Wasow \cite[I.1.3]{linear_turning_point}. The behavior of the solutions to Orr--Sommerfeld equation in a complex neighbourhood of a turning point (a point in $U^{-1}(\set{c})$ in our notation) is often discussed (see for instance \cite[Chapter 8]{lin_hydrodynamic_stability}, this is also, in a more general context, a central topic in \cite{linear_turning_point}). It is a well-known fact that if $c$ belongs to the range of $U$ and $x_0 \in(a,b)$ is such that $U(x_0) = c$ and $U'(x_0) \neq 0, U''(x_0) \neq 0$ then \eqref{eq:rayleigh} admits a solution near $x_0$ with a singularity of the kind $(x - x_0) \log(x -x_0)$ (see e.g. \cite[p.20]{schmid_henningson_book}). One can then wonder which branch of the logarithm should be used in this solution. If we see \eqref{eq:rayleigh} as a limit of the same equation for $c$ with small positive or negative imaginary parts, then we end up with two different determinations of the logarithm. It is explained in \cite[Chapter 8]{lin_hydrodynamic_stability} that one should consider the limit of $c$ with positive imaginary part in order to retrieve solutions that are limits of solutions of the Orr--Sommerfeld equation. This is why $c$'s with positive imaginary parts play a specific role in Theorems \ref{theorem:limit} and \ref{theorem:resonances}. Notice indeed that in Theorem \ref{theorem:resonances}, the resonances with positive imaginary parts are just the $c$'s in the upper half-plane for which there is a smooth solution to \eqref{eq:rayleigh} (with Dirichlet boundary condition), while the same result a priori does not hold for $c$ with negative imaginary part. Let us point out that the wave front set condition in Definition \ref{definition:resonant_states} has precisely the effect of imposing a specific determination of the logarithm in the solution of \eqref{eq:rayleigh} that are not smooth (in a way that is coherent with the analysis in \cite[Chapter 8]{lin_hydrodynamic_stability}). Equivalently, it imposes on which side of the real-axis the elements of $\Omega(c)$ continues analytically near each turning point (see \S \ref{subsection:description_resonant_states} for details).

There are references in fluid mechanics literature with approaches that are related to the complex deformation method that we use here. In \cite{rosencrans_sattinger_spectrum}, Rosencrans and Sattinger consider the solution $\psi_c$ to \eqref{eq:rayleigh} with $c \notin U([a,b])$ such that $\psi_c(a) = 0$ and $\psi_c'(a) = 1$. Solving Rayleigh equation on a deformation of the segment, they prove that, under some conditions on $U$, for every $x \in [a,b]$ the map $c \mapsto \psi_c(x)$ may be continued analytically across all intervals in the range of $U$ that do not contain $U(a), U(x)$ or a singular value of $U$. Notice that this continuation may be performed both from above and below the real axis but that they only consider the continuation from above for the reasons exposed in \cite[Chapter 8]{lin_hydrodynamic_stability} and recalled in the previous paragraph. They work similarly with the solution $\phi_c$ to \eqref{eq:rayleigh} (still for $c$ outside the range of $U$) such that $\phi_c(b) = 0$ and $\phi_c'(b) = 1$. Hence, they find that the Wronskian determinant of $\psi_c$ and $\phi_c$ may be continued across intervals in the range of $U$ that avoids $U(a),U(b)$ and singular values of $U$. We will explain in \S \ref{subsection:jost_functions} how the resonances (the set $\mathcal{R}$) from Theorem \ref{theorem:limit} may be retrieved from the analytic continuation of the Wronskian determinant across the range of $U$ from above. The analytic continuation of this Wronskian determinant is also used by Stepin \cite{stepin_rayleigh} in order to study the spectrum of Rayleigh equation.

The idea of integrating Rayleigh equation on a complex path whose choice is imposed by the will to get solutions that are inviscid limits of solutions of Orr--Sommerfeld equation (which is basically what we call the complex deformation method here) is also present in the physics literature (see for instance the paper by Tatsumi, Gotoh and Ayukawa \cite{tatsumi_gotoh_ayukawa_1966} or, discussing a related problem, \cite{gotoh_i_1968, gotoh_nakata_1969, gotoh_numata_1969, nakata_1970}).

In order to understand how the Neumann boundary condition ``disappears'' as the Reynolds number goes to $+\infty$, we apply a one-dimensional version of Vishik--Lyusternik method \cite{vishik_lyusternik_OG}. This method is based on the introduction of ``boundary layer'': solutions to \eqref{eq:orr_sommerfeld} that are concentrated near the boundary points of $[a,b]$ (in a region of size $\epsilon = R^{-1/2}$). More precisely, they behave near the boundary points like $x \mapsto e^{-x/\epsilon}$ near $0^+$. The term ``boundary layer'' is borrowed from physics terminology \cite{boundary_layer_book}. It was introduced by Prandtl in order to study Navier--Stokes equation in the low viscosity regime.

By analogy with \cite{dyatlov_zworski_viscosity,galkowski_zworski_viscosity}, we call the (generalized) eigenvectors, associated to the resonances from Theorem~\ref{theorem:limit}, ``resonant states''. For a real parameter $c$ in \eqref{eq:rayleigh} that satisfies hypothesis~\ref{hypothesis}, the resonant states are just the element of $\Omega(c)$, see Theorem~\ref{theorem:resonances}. In this context, the resonant states are exactly the ``eigenvectors'' that appear in the notion of ``embedding eigenvalue'' used in \cite{generalized_eigenvalue_segment} by Wei, Zhang and Zhao to study linear inviscid damping (see \cite[Definition 5.1]{generalized_eigenvalue_segment}). This notion has been adapted to Rayleigh equation on the circle by Beckie, Chen and Jia in \cite[Definition 1.1]{generalized_eigenvalue_circle} under the name ``generalized embedded eigenvalue''. The coincidence between these notions from \cite{generalized_eigenvalue_segment,generalized_eigenvalue_circle} and the real resonances in Theorem \ref{theorem:limit} is proved in \S \ref{subsection:generalized_eigenvalue}.

Our paper has also been motivated by the works of Grenier, Guo and Nguyen \cite{grenier_guo_nguyen_2016_duke, grenier_guo_ngyuyen_2016_aim}, where they showed unstable modes appear in the inviscid limit for characteristic boundary layer flows or symmetric analytic shear flows in a finite channel. We remark that the unstable modes constructed in \cite{grenier_guo_nguyen_2016_duke, grenier_guo_ngyuyen_2016_aim} are out of the scope of Theorem \ref{theorem:limit}, as hypothesis \ref{hypothesis} contains the assumption that $c_0$ is not a boundary value of $U$. Notice also that in \cite{grenier_guo_nguyen_2016_duke, grenier_guo_ngyuyen_2016_aim}, the parameter $\alpha$ depends on $R$, but this is something that we could probably include in our analysis. The optimality of the $(\alpha,R)$ region for which \cite{grenier_guo_ngyuyen_2016_aim} constructs unstable modes is discussed by Almog and Helffer in \cite{almog2022stability}.

As already mentioned, our approach is very similar to the method of complex scaling from scattering theory \cite[\S 2.7, \S 4.5, \S 4.7]{dyatlov_zworski_book}. Consequently, our result may be considered similar to \cite[Theorem 1]{zworski_scattering_viscosity}. One may also relate Theorem \ref{theorem:limit} to the viscosity limit statements in \cite{dyatlov_zworski_viscosity,galkowski_zworski_viscosity}. The approach in these papers is very similar to ours except that they work with more complicated spaces than those obtained by complex deformation (except in certain exceptional cases, see \cite[Appendix B]{galkowski_zworski_viscosity}). However, they do not have to deal with boundary conditions.

In \cite{quantized_riemann_surfaces}, Galtsev and Shafarevich study the spectrum of the operator $- h^2 \partial_x^2 + i \cos (x)$ on the circle as the parameter $h$ goes to $0$. This reference only deals with a very specific case, but they are able to describe the asymptotic of the spectrum in a larger part of the complex plane than we do. A similar study is carried out by Shkalikov and Tumanov in \cite{shkalikov_tumanov} for the problem
$- \epsilon y''(z) + P(z,\lambda) y(z) = 0$
on a segment $[a,b]$ with $y(a)= y(b) = 0$. Here, $\epsilon> 0$ goes to $0$ and $P(z,\lambda)$ is a polynomial in $z$ with coefficients that are analytic in $\lambda$. 
This problem is thought as a simpler version of Orr--Sommerfeld equation. Both references \cite{quantized_riemann_surfaces} and \cite{shkalikov_tumanov} rely on the study of the Stokes lines associated to the corresponding equation.

The complex deformation method can likely be applied to other spectral problems in fluid mechanics (see \cite[Appendix B]{galkowski_zworski_viscosity} for another recent application). According to \cite[(9.69)]{Vallis_2017}, the linear equation describing the baroclinic instability is 
\begin{equation} \begin{cases}\label{baroclinic}
    \left( (U-c)(\partial_y^2+\partial_zF\partial_z-k^2) + Q_y \right)\psi=0, \ & (y,z)\in (-1, 1)\times (0,H), \\
    (U-c)\partial_z\psi-\partial_z U \psi=0, \ & (y,z)\in (-1,1)\times \{0,H\}, \\
    \psi=0, \ & (y,z)\in \{-1,1\}\times (0,H).
\end{cases} \end{equation}
Here $U, F, Q_y$ are given functions depending on $(y,z)$ and $k, H$ are positive real numbers. In particular, $U$ is the background flow and $F>0$. We refer to \cite[\S 9]{Vallis_2017} and \cite[\S 7]{pedlosky_1987} for the physical background. A similar 2D equation is used to describe the secondary instability of G\"ortler vortices, see \cite[(2.11)--(2.12)]{Hall_Horseman_1991} and \cite[(2)--(5)]{yu_liu_1991} --- the later takes viscosity into account.
From a PDE point of view, \eqref{baroclinic} lose ellipticity near $U^{-1}(\{c_0\})$ for a given $c_0$ in the range of $U$. If such $c_0$ is not a boundary value of $U$ and $\nabla U|_{U^{-1}(\{c_0\})}\neq 0$, then one can design a complex deformation that makes the operator in \eqref{baroclinic} elliptic. This should allow one to study the Fredholm property of \eqref{baroclinic} and consequently define resonances for baroclinic flows. However, extra efforts are needed to deal with the boundary conditions, and one needs to consider a related viscous problem in order to establish which complex deformation can produce relevant resonances.

\subsection{Structure of the paper}

We begin with a presentation of the complex deformation method, which is the core of the proof of Theorems \ref{theorem:limit} and \ref{theorem:resonances}, in \S \ref{section:complex_scaling}. We explain then in \S \ref{section:analysis_hyperfunctions} how we are going to apply it in the context of Orr--Sommerfeld and Rayleigh equations. Some technical estimates (uniform Fredholm estimates for $P_{c,\epsilon}$ as $\epsilon$ goes to $0$) are proved in \S \ref{section:technical_estimates}. We apply then the Vishik--Lyusternik method in \S \ref{section:definition_resonances} in order to prove Theorem~\ref{theorem:limit}. In \S \ref{section:description_resonant_states}, we describe the elements of $\Omega(c)$ (from Definition \ref{definition:resonant_states}) and prove Theorem~\ref{theorem:resonances}. We relate the resonances from Theorem \ref{theorem:limit} with other notions that can be found in the literature in \S \ref{section:alternative_characterizations}. In \S \ref{section:further_results}, we explain how to adapt Theorems \ref{theorem:limit} and \ref{theorem:resonances} to study Orr--Sommerfeld equation on the circle and how to study the variation of resonances when the parameter $\alpha$ and $U$ are changed. Finally, we give some examples in \S \ref{section:examples}.

This paper also contains two appendices. In Appendix \ref{section:FBI_transform}, we give a description of the spaces defined using complex deformation in terms of a real-analytic FBI transform, in the spirit of \cite{FBI_bonthonneau_jezequel,galkowski_zworski_viscosity,galkowski_zworski_hypoellipticity,bonthonneau_guillarmou_jezequel, asymptotically_hyperbolic_jezequel}. Matlab codes used for numerical simulations illustrating this paper are given in Appendix \ref{section:matlab}.

\vspace{5pt}
\noindent
{\bf Acknowledgements.} The authors would like to thank Franck Sueur for his interest in this project and his comments on the manuscript. We would like to thank Jeremy L. Marzuola and Xinyu Zhao for useful suggestions on numerical methods.
We also thank Maciej Zworski and Semyon Dyatlov for helpful discussions.  
The first author benefits from the support of the French government “Investissements d’Avenir” program integrated to France 2030, bearing the following reference ANR-11-LABX-0020-01.

\section{Complex deformation}\label{section:complex_scaling}

In this section, we explain the main idea of the method of complex deformation that we will use in the proof of Theorems \ref{theorem:limit} and \ref{theorem:resonances}. This approach is greatly inspired by the complex scaling method in scattering theory: the interested reader may refer to \cite[\S 2.7, \S 4.5, \S 4.7]{dyatlov_zworski_book} and references therein for an introduction to complex scaling. One may also refer to \cite[Appendix B]{galkowski_zworski_viscosity} for the implementation of a similar method on the torus.

In \S \ref{subsection:generalities}, we list some generalities of the complex deformation method, working on the circle. We explain then in \S \ref{subsection:complex_scaling_segment} how we will rely on the circle case to work on the segment. Finally, in \S \ref{subsection:inverse_laplace}, we prove that certain operators are invertible on spaces defined using complex deformation (a result that will be useful later).

\subsection{Generalities}\label{subsection:generalities}

It is convenient to work on a closed manifold, we will consequently embed $[a,b]$ isometrically in the circle $\mathbb{S}^1 = \mathbb{R}/L \mathbb{Z}$, where $L$ is a number larger than $b-a$, and start by explaining how to work on $\mathbb{S}^1$. Let us fix a compact subset $K$ of $\mathbb{S}^1$ and consider a $C^\infty$ function $m$ on $\mathbb{S}^1$, supported in $K$. Define the curve
\begin{equation}\label{eq:complex_perturbation}
M = \set{ x + i m(x) : x \in \mathbb{S}^1} \subseteq \mathbb{C}/ L \mathbb{Z}.
\end{equation}
We think of $M$ as a small perturbation of the circle $\mathbb{S}^1$. The basic idea in our approach is to replace the original manifold we are working on (in our case, the circle) by a complex deformation of it (here, $M$). In order to justify this approach, let us explain how distributions on $M$ may be seen as some kind of ``generalized distributions'' (similar to hyperfunctions) on the circle. This point of view does not help proving Theorems \ref{theorem:limit} and \ref{theorem:resonances}, and one can ignore it while reading the proof of Theorems \ref{theorem:limit} and \ref{theorem:resonances}. However, we think that it is still interesting to discuss this point of view here, as it is a way to explain why it is reasonable to use the complex deformation method. Moreover, this point of view is also useful in Appendix \ref{section:FBI_transform}.

For $r > 0$, we denote by $\widetilde{\mathcal{B}}_{K,r}$ the space of $C^\infty$ functions $f$ on $\mathbb{S}^1$ such that the restriction of $f$ to $\set{x \in \mathbb{S}^1 : d(x,K) < r}$ is also the restriction of a bounded holomorphic function (that we will also denote by $f$) on $\set{ x \in \mathbb{C} / L \mathbb{Z} : d(x,K) < r}$\footnote{To lift any ambiguity, let us mention that we endow $\mathbb{C}/ L \mathbb{Z}$ with the distance $$d(x,y) = \inf_{\substack{\tilde{x},\tilde{y} \in \mathbb{C} \\ \mathfrak{p}(\tilde{x}) = x, \mathfrak{p}(\tilde{y}) = y} } |\tilde{x} - \tilde{y}|,$$ where $\mathfrak{p}$ denotes the canonical projection $\mathbb{C} \to \mathbb{C} / L \mathbb{Z}$.}. We endow $\widetilde{\mathcal{B}}_{K,r}$ with the locally convex topology defined by the semi-norms:
\begin{equation*}
    f \mapsto \sup_{x \in \mathbb{S}^1} |f^{(k)}(x)|, \ k \in \mathbb{N} \ \  \text{and} \ \ f \mapsto \sup_{\substack{x \in \mathbb{C}/ L \mathbb{Z} \\ d(x,K) < r}} |f(x)|.
\end{equation*}
We let then $\mathcal{B}_{K,r}$ be the closure of trigonometric polynomials in $\widetilde{\mathcal{B}}_{K,r}$ and $\mathcal{B}_{K,r}'$ be the topological dual of $\mathcal{B}_{K,r}$. The reason for introducing the space $\mathcal{B}_{K,r}'$ is the following: in order to prove Theorems \ref{theorem:limit} and \ref{theorem:resonances}, we need to work with objects that are not distributions. A natural approach would be to work with hyperfunctions (i.e. linear functionals on real-analytic functions). However, looking at the assumptions of Theorems \ref{theorem:limit} and \ref{theorem:resonances}, we see that we intend to work with functions that are \emph{a priori} not real-analytic everywhere, which forbids the use of hyperfunctions. To bypass this difficulty, we use the elements of $\mathcal{B}_{K,r}$ as test functions, defining the space $\mathcal{B}_{K,r}'$ whose elements may be thought as ``hyperfunctions on the circle that are distributions away from $K$''. The reason for which we restrict to the closure of trigonometric polynomials in the definition of $\mathcal{B}_{K,r}$ is because in this way the injection of $\mathcal{B}_{K,r}$ in $C^\infty(\mathbb{S}^1)$ defines by duality\footnote{We use the standard $1$-form on the circle to identify $C^\infty(\mathbb{S}^1)$ with the space of smooth densities on $\mathbb{S}^1$, hence identifying the dual of $C^\infty(\mathbb{S}^1)$ with $\mathcal{D}'(\mathbb{S}^1)$.} an injection of $\mathcal{D}'(\mathbb{S}^1)$ in $\mathcal{B}_{K,r}'$.

Let $r> 0$ and assume that $\n{m}_\infty < r$. We explain now how this assumption allows to identify the space of distributions $\mathcal{D}'(M)$ on $M$ with a subset of $\mathcal{B}_{K,r}'$. Notice that $M$ is a compact subset of 
\begin{equation*}
\mathcal{W}_r = \mathbb{S}^1 \cup \set{ x \in \mathbb{C} / L \mathbb{Z} : d(x,K) < r}.
\end{equation*}
Now, if $f$ is an element of $\mathcal{B}_{K,r}$, then it has an extension to $\mathcal{W}_r$, holomorphic in its interior, that we still denote by $f$. It follows from Cauchy's formula that the map $f \mapsto f_{|M}$ is bounded from $\mathcal{B}_{K,r}$ to $C^\infty(M)$. Hence, if $\psi \in \mathcal{D}'(M)$, we may define an element $\iota(\psi)$ of $\mathcal{B}_{K,r}'$ by
\begin{equation}\label{eq:identification_hyperfunctions}
\iota(\psi) (f) = \psi( (f\mathrm{d}z)_{|M}) \textup{ for } f \in \mathcal{B}_{K,r}.
\end{equation}
In this formula, $\mathrm{d}z$ is the holomorphic $1$-form on $\mathbb{C}/ L \mathbb{Z}$ that comes from the translation invariant $1$-form $\mathrm{d}z$ on $\mathbb{C}$. Using the orientation on $M$ given by the identification with $\mathbb{S}^1$ from \eqref{eq:complex_perturbation}, the $1$-form $(f\mathrm{d}z)_{|M}$ identifies with a density, so that the right hand side in \eqref{eq:identification_hyperfunctions} makes sense. In the following, we will use the notation
\begin{equation*}
\iota(\psi)(f) = \int_M \psi(z) f(z) \mathrm{d}z.
\end{equation*}

\begin{lemma}
Assume that $\n{m}_\infty < r$. Then, $\iota$ defines an injection from $\mathcal{D}'(M)$ to $\mathcal{B}_{K,r}'$. Moreover, if $\psi,f \in \mathcal{B}_{K,r}$, we have
\begin{equation}\label{eq:justification_identification}
\iota(\psi_{|M})(f) = \int_{\mathbb{S}^1} \psi(x) f(x) \mathrm{d}x.
\end{equation}
\end{lemma}

\begin{proof}
Start by noticing that, if $\psi,f \in \mathcal{B}_{K,r}$ then the definition \eqref{eq:identification_hyperfunctions} yields
\begin{equation*}
\iota(\psi_{|M})(f) = \int_M \psi(z) f(z) \mathrm{d}z,
\end{equation*}
where we recall that we identify $\psi$ and $f$ with their extensions to $\mathcal{W}_r$. Notice then that the homotopy $(t,x) \mapsto x + i t m(x)$ between $\mathbb{S}^1$ and $M$ only deform the circle within the interior of $\mathcal{W}_r$, where $\psi$ and $f$ are holomorphic. Hence, we may shift contour in the formula above to get \eqref{eq:justification_identification}.

Let us prove that $\iota$ is injective. Let $\psi \in \mathcal{D}'(M)$ be such that $\iota(\psi) = 0$. Let $\gamma$ be a smooth density on $M$ and let $f$ be the $C^\infty$ function on $M$ such that $\gamma = f \mathrm{d}z_{|M}$. The function $f$ may be approximated \cite[Theorem 3.5.4]{forstneric_book} in $C^\infty(M)$ by holomorphic functions on a neighbourhood of $M$. Identifying $\mathbb{C}/ L \mathbb{Z}$ with $\mathbb{C} \setminus \set{0}$ using the holomorphic map $z \mapsto e^{\frac{2i \pi}{L}z}$ and applying Runge Theorem, we may approximate each of these functions by trigonometric polynomials (uniformly on a neighbourhood of $M$). Hence, $f$ may be approximated by trigonometric polynomials in $C^\infty(M)$, and it follows that $\psi(\gamma) = 0$. Since $\gamma$ is arbitrary, we find that $\psi = 0$, and thus that $\iota$ is injective.
\end{proof}

We will sometimes drop the notation $\iota$ and identifies $\mathcal{D}'(M)$ with a subset of $\mathcal{B}_{K,r}'$. Now, let $P$ be a differential operator on $\mathbb{S}^1$ with coefficients in $\mathcal{B}_{K,r}$:
\begin{equation*}
P = \sum_{k = 0}^N a_k(x) \partial_x^k.
\end{equation*}
It follows from Cauchy's formula that, for any $r' \in (0,r)$, the operator $P$ maps $\mathcal{B}_{K,r}$ into $\mathcal{B}_{K,r'}$ continuously, and so does its formal adjoint $P^\top$. Hence, we may define the action of $P$ from $\mathcal{B}_{K,r'}'$ to $\mathcal{B}_{K,r}'$ by duality
\begin{equation*}
P \psi(f) = \psi (P^\top f) \textup{ for every } \psi \in \mathcal{B}_{K,r'}' \textup{ and } f \in \mathcal{B}_{K,r}.
\end{equation*}
The main point of the complex deformation method is that $P$ actually maps $\mathcal{D}'(M)$ into itself, with a particularly simple description in this case. Let us define a differential operator $P_M$ on $M$ in the following way. On the open set $\mathbb{S}^1 \setminus K \subseteq M$, the operator $P_M$ agrees with $P$. Notice that the coefficients of $P$ are holomorphic on the interior of $\mathcal{W}_r$. To $P$, we can consequently associate a holomorphic differential operator on the interior of $\mathcal{W}_r$:
\begin{equation*}
P_{\textup{hol}} = \sum_{k = 0}^N a_k(z) \partial_z^k.
\end{equation*}
Now, if $\psi \in C^\infty(M)$ is supported in $\set{ x + i m(x) : x \in \mathbb{S}^1, d(x,K) < r}$, let $\tilde{\psi}$ be a pseudo-analytic extension\footnote{The function $\tilde{\psi} : \mathbb{C}/ L \mathbb{Z} \to \mathbb{C}$ coincides with $\psi$ on $M$ and $\bar{\partial} \tilde{\psi}$ vanishes at infinite order on $M$. The use of pseudo-analytic extensions is a conveniency here, we could define $P_M$ in a merely algebraic way.} of $\psi$ and define $P_M \psi = (P_{\textup{hol}} \tilde{\psi})_{|M}$. The result does not depend on the choice of the real analytic extension $\tilde{\psi}$ and coincides with the previous definition on $\set{ x + i m(x) : x \in \mathbb{S}^1, d(x,K) < r} \cap \mathbb{S}^1 \setminus K$. Hence, we defined the differential operator $P_M$ on the whole $M$. Notice that if we parametrize $M$ by $\mathbb{S}^1$ using the map $x \mapsto x + i m(x)$, then $P_M$ is given in these coordinates by the formula
\begin{equation*}
\sum_{k = 0}^N a_k(x + i m(x)) ((1+ i m'(x))^{-1} \partial_x)^k.
\end{equation*}
The point of this discussion is that, if $\psi \in \mathcal{D}'(M)$ then the action of $P$ on $\psi$ (as an element of $\mathcal{B}_{K,r'}'$) coincides with the action of $P_M$ on $\psi$ (as an element of $\mathcal{D}'(M)$). In algebraic terms:

\begin{proposition}\label{proposition:justification}
For every $\psi \in \mathcal{D}'(M)$, we have
\begin{equation*}
P\iota(\psi) = \iota(P_M \psi).
\end{equation*}
\end{proposition}

\begin{proof}
Testing this relation against an element of $\mathcal{B}_{K,r}$, we see that we want to prove that
\begin{equation}\label{eq:actual_relation}
(P^{\top} f \mathrm{d}z)_{| M} = P_M^{\top} (f \mathrm{d}z)_{|M}
\end{equation}
for every $f \in \mathcal{B}_{K,r}$. Here, the adjoint $P_M^\top$ is defined using the pairing between smooth functions and densities on $M$. Start by noticing that we have
\begin{equation*}
P^{\top} f (z) = \sum_{k = 0}^N (-1)^k \partial_z^k(a_k(z) f(z)) \textup{ for } z \in \mathcal{W}_r.
\end{equation*}
This relation, \emph{a priori} only valid on $\mathbb{S}^1$, actually holds on $\mathcal{W}_r$ by the analytic continuation principle. The relation \eqref{eq:actual_relation} follows then by iterated integration by parts on $M$ (that can be performed in the coordinates $x \mapsto x + i m(x)$).
\end{proof}

Proposition \ref{proposition:justification} justifies to drop the notation $P_M$, as it ends up being just another name for the action of $P$ on $\mathcal{D}'(M)$, seen as a subset of $\mathcal{B}_{K,r'}'$ (for some $r'$ close to $r$). Another consequence of Proposition \ref{proposition:justification} is that, in order to study the operator $P$, it is legitimate to study its action on $\mathcal{D}'(M)$ (as defined above) instead of its action on $L^2(\mathbb{S}^1)$ for instance. If $M$ is well-chosen, the operator $P$ may be better behaved on $M$ than on $\mathbb{S}^1$. However, it would be pointless to do so if we could not relate the action of $P$ on $M$ and on $\mathbb{S}^1$ in a more concrete way than through Proposition \ref{proposition:justification}. When $P$ is elliptic, it is easily done, as explain in the following simple result. In this lemma, we use the notion of restriction of a holomorphic differential operator defined on a complex analytic manifold to a totally real\footnote{The totally real assumption is empty in dimension 1, but needed to generalize this definition to higher dimensions.} submanifold of maximal dimension. This restriction may be defined in this more general context as we did above in a particular case, using pseudo-analytic extensions.

\begin{lemma}\label{lemma:elliptic_extension}
Let $X$ be a connected complex analytic manifold of dimension $1$. Let $M$ be a connected real submanifold of dimension $1$ of $X$. Assume that the map\footnote{We write $\pi_1(M)$ and $\pi_1(X)$ for the fundamental groups respectively of $M$ and $X$.} $\pi_1(M) \to \pi_1(X)$ induced by the inclusion is surjective. Let $P$ be an elliptic holomorphic differential operator on $X$, and $\widetilde{P}$ the associated differential operator on $M$. Let $\psi$ be a $C^\infty$ function on $M$ and $f$ a holomorphic function on $X$. Assume that $\widetilde{P}\psi = f_{|M}$. Then $\psi$ has a holomorphic extension to $X$ that satisfies $P \psi = f$.
\end{lemma}

\begin{proof}
Let $d$ be the order of the differential operator $P$. For each $x \in X$, we let $E_x$ denote the space of $d-1$-jets of holomorphic functions at $x$ (that is the space of holomorphic functions on a neighbourhood of $x$ up to holomorphic functions vanishing at order $d$ at $x$). By the holomorphic version of the Cauchy--Lipschitz theorem, if $V$ is a simply connected open subset of $X$, $x \in V$ and $\mathfrak{j} \in E_x$, then there is a unique holomorphic function $v$ on $V$ such that $P v = f$ and whose $d-1$-jet at $x$ is $\mathfrak{j}$. Now, let us consider a continuous path $\gamma : [0,1] \to X$ joining two points $x$ and $y$. By compactness, we can find a sequence of times $0 = t_0 < t_1 < \dots <t_n = 1$ and a sequence of open disks $D_0,\dots D_{n-1}$ in $X$ such that for $k = 0,\dots,n-1$, the segment $\gamma([t_k,t_{k+1}])$ is contained in $D_k$. Consider a jet $\mathfrak{j} \in E_x$, and let $v_0$ satisfies $P v_0 = f$ on $D_0$ with jet $\mathfrak{j}$ at $x$. Use then the jet of $v_0$ at $\gamma(t_1)$ as an initial solution to define a solution $v_1$ to $Pv_1 = f$ on $D_1$. We keep going and end up with a holomorphic functions $v_{n-1}$ on $D_{n-1}$. We let $\mathfrak{P}_\gamma(\mathfrak{j}) \in E_y$ denote the $d-1$-jet of $v_{n-1}$ at $y$. It is standard, using uniqueness in the Cauchy--Lipschitz theorem, that $\mathfrak{P}_\gamma(\mathfrak{j})$ does not depend on the choice of times $t_0,\dots,t_n$ and of disks $D_0,\dots,D_{n-1}$. Actually, one may check that $\mathfrak{P}_\gamma(\mathfrak{j})$ only depends on $\gamma$ through its homotopy class.

Let us fix a point $x_0 \in M$. Notice that there is a unique element $\mathfrak{j}_0 \in E_{x_0}$ whose restriction to $M$ coincides with the jet of $\psi$ at $x_0$ as a $C^\infty$ function (because $M$ is totally real). For each $x \in M$, let $\gamma_x$ be a path from $x_0$ to $x$ and set $\psi(x) = \mathfrak{P}(\mathfrak{j}_0)(x)$. In order to conclude, we only need to prove that our value of $\psi(x)$ does not depend on the choice of the path $\gamma_x$. Indeed, it is then clear that the resulting function $\psi$ is a holomorphic solution to $P\psi = f$ (because it is locally given by such solutions) and that it coincides with $\psi$ on $M$ (because if $x \in M$, we can choose the path $\gamma_x$ to be in $M$).

Let us consequently consider a point $x \in M$ and two paths $\gamma_x$ and $\tilde{\gamma}_x$ going from $x_0$ to $x$. Let $\mathfrak{j} = \mathfrak{P}_{\gamma_x}(\mathfrak{j}_0)$ and $\tilde{\mathfrak{j}} = \mathfrak{P}_{\tilde{\gamma}_x}(\mathfrak{j}_0)$. Notice that $\tilde{\gamma}_x $ is homotopically equivalent to $\gamma_x \gamma_x^{-1} \tilde{\gamma}_x $. But $\gamma_x^{-1} \tilde{\gamma}_x$ is a loop based at $x_0 \in M$. Hence, by assumption it is homotopy equivalent to a loop $c$ in $M$. Hence, we have $\tilde{\mathfrak{j}} = \mathfrak{P}_{\gamma_x} \mathfrak{P}_c \mathfrak{j_0}$. By uniqueness in Cauchy--Lipschitz theorem we find that $\mathfrak{P}_c (\mathfrak{j}_0) = \mathfrak{j}_0$ and thus $\tilde{\mathfrak{j}} = \mathfrak{P}_{\gamma_x}(\mathfrak{j}_0) = \mathfrak{j}$. In particular, $\mathfrak{j}(x) = \tilde{\mathfrak{j}}(x)$, which ends the proof of the lemma.
\end{proof}

\subsection{Complex deformation of a segment}\label{subsection:complex_scaling_segment}

Recall that our goal is to study an operator defined on the segment $[a,b]$. To do so we embedded $[a,b]$ isometrically in a circle $\mathbb{S}^1 = \mathbb{R}/ L \mathbb{Z}$. Let us introduce, with the notations from \S \ref{subsection:generalities}, the set
\begin{equation*}
M_{a,b} = \set{x + i m(x) : x \in (a,b)}.
\end{equation*}
We will think of $\mathcal{D}'(M_{a,b})$ as a set of generalized distributions on $(a,b)$. Notice that $M_{a,b}$ is a domain with boundary in $M$ and we can consequently define the Sobolev spaces $H^k(M_{a,b}), k \in \mathbb{R}$ on $M_{a,b}$ as the space of restrictions of elements of $H^k(M)$ on $M_{a,b}$. If $P$ is a differential operator with coefficients in $\mathcal{B}_{K,r}$ (with $\n{m}_\infty < r$), then we will define the action of $P$ on $\mathcal{D}'(M_{a,b})$ as the action of the operator $P_M$ defined in \S \ref{subsection:generalities}.

When working on a segment, we will always assume that $K \subseteq [a,b]$, which implies that the point $a$ and $b$ are the boundary points of $M_{a,b}$ in $M$. Let us explain how, if the stronger condition $K \subseteq (a,b)$ holds and $r$ is strictly less than the distance between $K$ and $\set{a,b}$, we may identify the elements of $\mathcal{D}'(M_{a,b})$ with generalized distributions on $(a,b)$. Let $\mathcal{B}_{K,r}^c(a,b)$ be the space\footnote{The topology on this space is defined by taking the inductive limit (of locally convex topological vector spaces) over compact subset of $(a,b)$.} of elements of $\mathcal{B}_{K,r}$ that are supported in $(a,b)$. We can then use the formula \eqref{eq:identification_hyperfunctions} to define an injection of $\mathcal{D}'(M_{a,b})$ inside $(\mathcal{B}_{K,r}^c(a,b))'$ that makes the following diagram commutative (where the vertical arrows are the restriction maps): 
\begin{equation*}
\begin{tikzcd}
\mathcal{D}'(M) \arrow[r] \arrow[d] & \mathcal{B}_{K,r}' \arrow[d] \\
\mathcal{D}'(M_{a,b}) \arrow[r] &  (\mathcal{B}_{K,r}^c(a,b))'
\end{tikzcd}
\end{equation*}
Moreover, the action of differential operators on $\mathcal{D}'(M_{a,b})$ may be described as in the circle case.

Notice that we could have defined $M_{a,b}$ directly as a deformation of $(a,b)$ in $\mathbb{C}$, without embedding $[a,b]$ in the circle and introducing the deformation $M$ first. However, it will be useful in the following to work first on a closed manifold, hence the introduction of the circle.

Finally, let us mention that for $k > 1/2$, we will use the notation
\begin{equation*}
H^k_{\D}(M_{a,b}) = \set{ u \in H^k(M_{a,b}) : u(a) = u(b) = 0}
\end{equation*}
to denote the space of elements of $H^k_D(M_{a,b})$ that satisfy Dirichlet boundary condition. Similarly, when $k > 5/2$, we let
\begin{equation*}
H^k_{\DN}(M_{a,b}) = \set{ u \in H^k(M_{a,b}): u(a) = u(b) = \partial_x u(a) = \partial_x u(b) = 0}
\end{equation*}
be the space of elements of $H^k(M_{a,b})$ that satisfy both Dirichlet and Neumann boundary conditions.

\subsection{Inverting \texorpdfstring{$\partial_x^2-\alpha^2$}{d2-alpha2} on the deformed segment}\label{subsection:inverse_laplace}

Let $\alpha > 0$. We want now to invert the operator $\partial_x^2 - \alpha^2$ on the spaces defined through complex deformation. Here, we are using the notations from \S\S \ref{subsection:generalities}--\ref{subsection:complex_scaling_segment}. The action of $\partial_x^2 - \alpha^2$ on $\mathcal{D}'(M)$ and $\mathcal{D}'(M_{a,b})$ is also defined in \S\S \ref{subsection:generalities}--\ref{subsection:complex_scaling_segment}.

We start with a result on the circle:

\begin{lemma}\label{lemma:inverse_laplace_circle}
Let $\alpha > 0$ and $k \in \mathbb{R}$. The operator $\partial_x^2 - \alpha^2 : H^{k+2}(M) \to H^k(M)$ is invertible.
\end{lemma}

\begin{proof}
Let us (artificially) introduce a small semiclassical parameter $h > 0$ and consider the semiclassical differential operator $A = h^2 (\partial_x^2 - \alpha^2) -1$ on $M$. If we use the parametrization $x \mapsto x + i m(x)$ of $M$ by $\mathbb{S}^1$, then the principal symbol of $A$ reads $(x,\xi) \mapsto - (1 + i m'(x))^2 \xi^2 - 1$, which is elliptic (because $(1+ i m'(x))^2$ stays away from $(- \infty,0]$ while $x$ goes over $\mathbb{S}^1$). Hence, for $h$ small enough, the operator $A : H^{k+2}(M) \to H^k(M)$ is invertible. Hence, $\partial_x^2 - \alpha^2= 1 + h^2 A$ is Fredholm of index zero. Thus, we only need to prove that $\partial_x^2 - \alpha^2$ is injective on $H^{k+2}(M)$.

Let $\psi \in H^{k+2}(M)$ be such that $(\partial_x^2 - \alpha^2)\psi = 0$. It follows from elliptic regularity that $\psi \in C^\infty(M)$. Applying Lemma \ref{lemma:elliptic_extension}, we find that $\psi$ is the restriction to $M$ of a holomorphic function, still denoted by $\psi$, on $\mathbb{C} / L \mathbb{Z}$. The extension also satisfies $(\partial_x^2 - \alpha^2) \psi = 0$. Hence, the restriction of $\psi$ to $\mathbb{S}^1$ must be zero, and by analytic continuation principle we find that $\psi = 0$.
\end{proof}

We have an analogue statement on the segment:

\begin{lemma}\label{lemma:inverse_laplace_segment}
Let $\alpha  > 0$ and $k > - 1/2$. The operator $\partial_x^2 - \alpha^2 : H^{k+2}_{\D}(M_{a,b}) \to H^{k}(M_{a,b})$ is invertible.
\end{lemma}

\begin{proof}
Let us prove surjectivity first. Let $f \in H^{k}(M_{a,b})$. Let $f_0 \in H^k(M)$ be an extension of $f$ to $M$. Let then $\psi \in H^{k+2}(M_{a,b})$ be the restriction of $(\partial_x^2 - \alpha^2)^{-1} f_0$ to $M_{a,b}$. The function $\psi$ satisfies $(\partial_x^2 - \alpha^2) \psi = 0$, but we may not have $\psi(a) = \psi(b) = 0$. This is fixed by adding a linear combination of $z \mapsto e^{\alpha z}$ and $z \mapsto e^{ - \alpha z}$ to $\psi$. Indeed, we can find a linear combination of these functions with any values at $a$ and $b$ since the matrix
\begin{equation*}
\begin{pmatrix}
e^{\alpha a} & e^{- \alpha a} \\ e^{\alpha b} & e^{ - \alpha b}
\end{pmatrix}
\end{equation*}
is invertible.

Let us now prove injectivity. Let $\psi \in H^{k + 2}_{\D}(M_{a,b})$ be such that $(\partial_x^2 - \alpha^2) \psi = 0$. By elliptic regularity, $\psi \in C^\infty(\overline{M}_{a,b})$. Let then $\phi$ be the solution to $(\partial_x^2 - \alpha^2) \phi = 0$ on $\mathbb{C}$ such that $\phi(a) = 0$ and $\partial_x \phi(a) = \partial_x \psi(a)$. The function $\phi$ exists and is entire (it is just a linear combination of $z \mapsto e^{\alpha z}$ and $z \mapsto e^{ - \alpha z}$). By uniqueness in Cauchy--Lipschitz theorem, we find that the restriction of $\phi$ to $M_{a,b}$ is just $\psi$, and thus $\phi(b) = 0$, which imposes that $\phi = 0$, and thus $\psi = 0$.
\end{proof}

A similar reasoning yields:

\begin{lemma}\label{lemma:inverse_laplace_square}
Let $\alpha > 0$. Then for every $k > -5/2$, the operator $(\partial_x^2 - \alpha^2)^2 : H_{\DN}^{k+4}(M_{a,b}) \to H^{k}(M_{a,b})$ is invertible. 
\end{lemma}

\section{Deformation of the circle}\label{section:analysis_hyperfunctions}

We are now going to use tools from \S \ref{section:complex_scaling} to study Orr--Sommerfeld equation and progress toward the proof of Theorems \ref{theorem:limit} and \ref{theorem:resonances}. Consequently, we will use the notations from the introduction and fix the profile $U$ and the real number $\alpha > 0$. We let also $c_0 \in \mathbb{R}$ satisfy hypothesis \ref{hypothesis}. The parameters $U, \alpha$ and $c_0$ are globally fixed from now to the end of \S \ref{section:alternative_characterizations}.

In this section, we explain how to specify the setting from \S \ref{section:complex_scaling} in order to prove Theorems \ref{theorem:limit} and \ref{theorem:resonances}. The crucial point is the definition of the deformation $M$ of the circle, done in \S \ref{section:escape_function}. We state then some basic properties of the operators $P_{c,\epsilon}$'s acting on the spaces defined by complex deformation in \S \ref{subsection:basic_properties}.

\subsection{Choice of the perturbation}\label{section:escape_function}

\begin{figure}[b]
   \centering
   \includegraphics[width=0.7\textwidth]{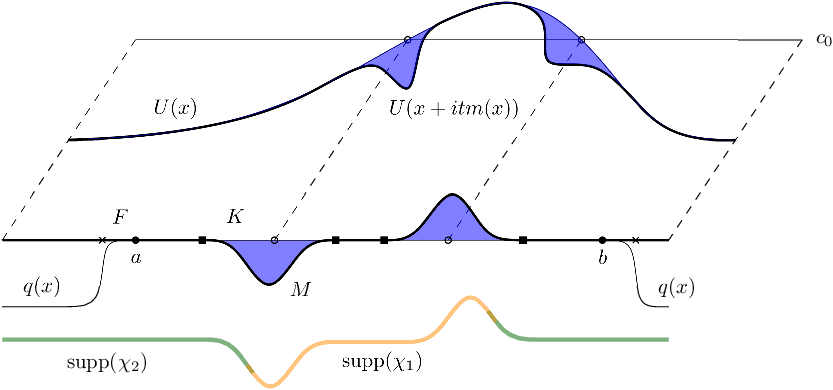}
   \caption{Illustration of the choice of perturbations. Blue shadows represent complex values. Functions $\chi_{1}$ and $\chi_2$ are used in the proof of Lemma \ref{lemma:invertibility_full_circle}. }
   \label{fig:contour}
\end{figure}

We will use the setting from \S \ref{section:complex_scaling}. We need consequently to fix the values of the parameters $K$ (a compact subset of the circle), $r$ (a positive number) and $m$ (a $C^\infty$ function supported in $K$).

Let us extend $U$ as a $C^\infty$ function on the circle. We let $K$ be any compact subset of $(a,b)$ such that $(U^{-1}(\set{c_0}) \cap [a,b]) \subseteq K$ and $U$ is analytic on a neighbourhood of $K$. We let then $r$ be small enough so that $U \in \mathcal{B}_{K,r}$ and $d(K,\set{a,b}) < r$. Let then $m_0 : \mathbb{S}^1 \to \mathbb{R}$ be any $C^\infty$ function supported in $K \setminus (U')^{-1}( \set{0})$ and such that $m_0(x) U'(x) \leq 0$ for every $x \in K$ and $m_0(x) U'(x) < 0$ for every $x \in U^{-1}(\set{c_0}) \cap [a,b]$. Define $m = \tau m_0$ for some $\tau > 0$ to be fixed later (in Lemma \ref{lemma:non_zero_function}). We could work with slightly more general deformations (in particular, we could deform the segment up to the extremities $a$ and $b$ as long as they are fixed) but the choice that we make here will make some statements simpler (typically Proposition \ref{proposition:boundary_layers}).

Let then $F$ be a closed neighbourhood of $[a,b]$ such that $U$ does not take the value $c_0$ on $F \setminus [a,b]$. Choose a $C^\infty$ function $q$ from $\mathbb{S}^1$ to $(- \infty,0]$, supported in $\mathbb{S}^1 \setminus [a,b]$ and such that 
\begin{equation*}
    \sup_{x \in \mathbb{S}^1 \setminus F} q(x) < 0.
\end{equation*}

For $c \in \mathbb{C}$ and $\epsilon \geq 0$, we extend $P_{c,\epsilon}$ as an operator of the circle by the formula
\begin{equation}\label{eq:Pextension_circle}
P_{c,\epsilon} = i\alpha^{-1} \epsilon^2(\partial_x^2 - \alpha^2)^2 + (U - c + iq)(\partial_x^2 - \alpha^2) - U''.
\end{equation}
Notice that the order of $P_{c,\epsilon}$ is not the same when $\epsilon = 0$ and $\epsilon > 0$. Notice also that the coefficients of $P_{c,\epsilon}$ are in $\mathcal{B}_{K,r}$, and thus the operator $P_{c,\epsilon}$ acts on $\mathcal{D}'(M)$ and $\mathcal{D}'(M_{a,b})$, as explained in \S \ref{section:complex_scaling}. The issue that appears when studying the action of the $P_{c,\epsilon}$'s on the circle is that, if $c$ belongs to the range of $U+ iq$, then the operator $P_{c,0}$ is not elliptic. The point of introducing a complex deformation $M$ of $\mathbb{S}^1$ is precisely to fix this issue: we want to construct $M$ such that the range of $U + iq$ on $M$ does not contain $c_0$. This is the point of the following lemma.

\begin{lemma}\label{lemma:non_zero_function}
If $\tau$ is small enough then there is $\delta > 0$ such that the following holds:
\begin{enumerate}[label=(\roman*)]
\item for every $x \in \mathbb{S}^1$, the argument of $(1 - i m'(x))/(1 + m'(x)^2)$ is in $( - \pi/4,\pi/4)$;
\item for every $x \in \mathbb{S}^1,t \in [0,1]$ and $c \in (c_0 - \delta,c_0 + \delta)$, we have $\im(U(x + i t m(x)) + i q(x + i tm(x))) \leq 0$ and if $U(x + i t m(x)) + i q(x + i t m(x)) = c$ then $U(x) = c$ and $t = 0$.\label{item:non_zero_ii}
\end{enumerate}
\end{lemma}

\begin{proof}
The first point is clear. To prove the second, let us consider a point of the form $x + i s m_0(x)$ with $x \in \mathbb{S}^1$ and $s \in [0,\tau]$. If $x \notin F$ then $m_0(x) = 0$ and $\im (U(x + i s m_0(x)) + i q(x+ i s m_0(x))) = q(x) < 0$, which implies in particular that $U(x+ i s m_0(x)) + i q(x + i s m_0(x)) \neq c$ for $c$ real. Let us consequently assume that $x \in F$ and notice that $\im (i q(x + i s m_0(x))) \leq 0$, so that we only need to estimate $U(x + i s m_0(x)) - c$, for $c$ real and close to $c_0$. We have
\begin{equation*}
U(x + i s m_0(x)) = U(x) + i s m_0(x) U'(x) + \mathcal{O}(s^2 m_0(x)^2),
\end{equation*}
so that
\begin{equation*}
\im( U (x + i s m_0(x)) - c) = s m_0(x) U'(x) + \mathcal{O}(s^2 m_0(x)^2),
\end{equation*}
and
\begin{equation*}
\re( U(x + i s m_0(x)) - c) = U(x) - c + \mathcal{O}(s^2).
\end{equation*}
Notice that the function $m_0$ has be chosen so that $- \sign(m_0(x)) U'(x)$ is uniformly bounded below whenever $m_0(x) \neq 0$. Hence, provided $\tau$ is small enough, we have $\im (U(x + i s m_0(x)) - c) \leq 0$, and the inequality is strict unless $s = 0$ or $m_0(x) = 0$. However, if $m_0(x) = 0$, then it follows from the definition of $m_0$ that $x$ is outside of a fixed neighbourhood of $U^{-1}(\set{c_0})$, and thus that the real part of $U(x + i s m_0(x)) - c$ is non-zero if $\tau$ is small enough and $c$ is close to $c_0$.
\end{proof}

From now on, we fix the value of $\tau$ so that Lemma \ref{lemma:non_zero_function} holds.

\subsection{Basic properties}\label{subsection:basic_properties}

Before going into more technical proofs, we can already prove some properties of the operators $P_{c,\epsilon}$ in Sobolev spaces on $M_{a,b}$. The main point is that, when $\epsilon > 0$, the operators $P_{c,\epsilon}$ are elliptic all along the deformation from $[a,b]$ to $\overline{M}_{a,b}$ and thus have the same properties on $[a,b]$ and $\overline{M}_{a,b}$, see Proposition \ref{proposition:elliptic_eigenvalues} below. A crucial consequence of this fact is that we can work on $\overline{M}_{a,b}$ instead of $[a,b]$ to study the behaviour of $\Sigma_\epsilon$ as $\epsilon$ goes to $0$.

We use this opportunity to explicit the notion of multiplicity for elements of $\Sigma_\epsilon$. For $\epsilon > 0$ and $c \in \mathbb{C}$ let $\mathfrak{S}_{c,\epsilon}([a,b])$ be the smallest subspace of $C^\infty([a,b])$ with the property that if $\psi \in C^\infty([a,b])$ and $\phi \in \mathfrak{S}_{c,\epsilon}([a,b])$ are such that $P_{c,\epsilon} \psi = (\partial_x^2 - \alpha^2) \phi$ and $\psi(a) = \psi(b) = \partial_x \psi(a) = \partial_x \psi(b) = 0$ then $\psi \in \mathfrak{S}_{c,\epsilon}([a,b])$. We let $\mathfrak{S}_{c,\epsilon}(\overline{M}_{a,b})$ be the smallest subset of $C^\infty(\overline{M}_{a,b})$ with the same property\footnote{The spaces $\mathfrak{S}_{c,\epsilon}([a,b])$ and $\mathfrak{S}_{c,\epsilon}(\overline{M}_{a,b})$ maybe interpreted as generalized eigenspaces for the operator $(\partial_x^2 - \alpha^2)^{-1} P_{c,\epsilon}$, see the proof of Proposition \ref{proposition:elliptic_eigenvalues}.}. We define the multiplicity of $c$ as an element of $\Sigma_\epsilon$ as the dimension of $\mathfrak{S}_{c,\epsilon}([a,b])$. This definition is motivated by the following result (and its proof), and coincides with the multiplicity of $c$ as an eigenvalue of $Q_{0,\epsilon}$.

\begin{proposition}\label{proposition:elliptic_eigenvalues}
Let $\epsilon > 0$ and $k > - 5/2$. The following holds:
\begin{enumerate}[label=(\roman*)]
\item the set $\Sigma_\epsilon$ is discrete;\label{item:sigma_discrete}
\item for every $c \in \mathbb{C}$, the operator $P_{c,\epsilon} : H^{k+4}_{\DN}(M_{a,b}) \to H^k(M_{a,b})$ is Fredholm of index zero, it is invertible if and only if $c \notin \Sigma_\epsilon$;\label{item:fredholm_complex}
\item for every $c \in \mathbb{C}$ the spaces $\mathfrak{S}_{c,\epsilon}([a,b])$ and $\mathfrak{S}_{c,\epsilon}(\overline{M}_{a,b})$ are finite dimensional, they have the same dimension, and this dimension is non-zero if and only if $c \in \Sigma_\epsilon$.\label{item:multiplicity_non_spectral}
\end{enumerate}
\end{proposition}

\begin{proof}
The operator $P_{c,\epsilon}$ is elliptic on $[a,b]$ for every $c \in \mathbb{C}$. Hence, $c \mapsto P_{c,\epsilon}$ defines a holomorphic family of Fredholm operators from $H_{\DN}^4((a,b))$ to $L^2([a,b])$. Consequently, \ref{item:sigma_discrete} is a consequence of Fredholm analytic theory \cite[Appendix C]{dyatlov_zworski_book} if we can prove that there is $c \in \mathbb{C}$ such that $P_{c,\epsilon} : H_{\DN}^4(a,b) \to L^2([a,b])$ is invertible. From the ellipticity of $P_{c,\epsilon}$, we find that it is equivalent to prove that it is invertible as an operator from $H^2_0(a,b)$ to $H^{-2}(a,b)$. To do so, we will prove that the quadratic form $\psi \mapsto \langle P_{c,\epsilon} \psi , \psi \rangle_{L^2(a,b)}$ is coercive on $H^2_0(a,b)$: the invertibility of $P_{c,\epsilon}$ is then a standard consequence of Lax--Milgram Theorem. To prove coercivity, we use the explicit formula (obtained by integrating by parts) valid for $\psi \in H_0^2(a,b)$:
\begin{equation*}
\begin{split}
\langle P_{c,\epsilon} \psi , \psi \rangle_{L^2(a,b)} & = i\alpha^{-1} \epsilon^2 \n{(\partial_x^2 - \alpha^2)\psi}_{L^2}^2 + c \n{\partial_x \psi}_{L^2}^2 + c \alpha^2 \n{\psi}_{L^2}^2 \\ & \qquad - \int_a^b (U'' + \alpha^2 U) |\psi|^2 \mathrm{d}x - \int_a^b U |\partial_x \psi|^2 \mathrm{d}x - \int_a^b U' \partial_x \psi \bar{\psi}\mathrm{d}x.
\end{split}
\end{equation*}
Hence, there is a constant $C > 0$, depending on $\epsilon$, such that for $\im c > 0$, we have
\begin{equation*}
\im \langle P_{c,\epsilon} \psi , \psi \rangle_{L^2(a,b)} \geq C^{-1} \n{\psi}_{H^2}^2 + C^{-1} \im c \n{\psi}_{H^1}^2 - C \n{\psi}_{H^1}^2. 
\end{equation*}
Hence, if $\im c$ is large enough, we get $\im \langle P_{c,\epsilon} \psi , \psi \rangle_{L^2(a,b)} \geq C^{-1} \n{\psi}_{H^2}^2$ proving the coercivity of $\psi \mapsto \langle P_{c,\epsilon} \psi , \psi \rangle_{L^2(a,b)}$, and thus the invertibility of $P_{c,\epsilon}$. This ends the proof of \ref{item:sigma_discrete}.

Let us move to the proof of \ref{item:fredholm_complex}. Let $c \in \mathbb{C}$. Notice that $P_{c,\epsilon}$ coincides with the operator $i\alpha^{-1} \epsilon^2 (\partial_x^2 - \alpha^2)^2$ up to a compact operator $H_{\DN}^{k+4}(M_{a,b}) \to H^k(M_{a,b})$. Since $i\alpha^{-1} \epsilon^2 (\partial_x^2 - \alpha^2)^2 : H_{\DN}^{k+4}(M_{a,b}) \to H^k(M_{a,b})$ is invertible according to Lemma \ref{lemma:inverse_laplace_square}, we find that $P_{c,\epsilon} : H_{\DN}^{k+4}(M_{a,b}) \to H^k(M_{a,b})$ is Fredholm of index zero. In particular, it is invertible if and only if it is injective. Moreover, by elliptic regularity, any element in the kernel of $P_{c,\epsilon}$ in $H^{k+4}_{\DN}(M_{a,b})$ is actually in $C^\infty(\overline{M}_{a,b})$. Assume consequently that there is a non-zero $\psi \in C^\infty(\overline{M}_{a,b})$ such that $P_{c,\epsilon} \psi = 0$. Define the set 
\begin{equation*}
\mathcal{W}_r^{a,b} = [a,b] \cup \set{x \in \mathbb{C}/ L \mathbb{Z} : d(x,K) < r},
\end{equation*}
and recall that the coefficients of $P_{c,\epsilon}$ have extensions to $\mathcal{W}_r^{a,b}$, holomorphic in the interior of $\mathcal{W}_r^{a,b}$. Applying Lemma \ref{lemma:elliptic_extension} to each connected component of the interior of $\mathcal{W}_r^{a,b}$ the function $\psi$ has an extension to $\mathcal{W}_r^{a,b}$ which is holomorphic in the interior of $\mathcal{W}_r^{a,b}$. Still denoting by $\psi$ this extension, it follows from the analytic continuation principle that the restriction of $\psi$ to $[a,b]$ is non-zero and in the kernel of $P_{c,\epsilon}$. Moreover, it satisfies both Dirichlet and Neumann boundary conditions. Hence, we get $c \in \Sigma_\epsilon$. We just proved that if $P_{c,\epsilon} : H^{k+4}_{\DN}(M_{a,b}) \to H^k(M_{a,b})$ then $c \in \Sigma_\epsilon$. The reciprocal implication is obtained similarly (we start with an element of the kernel of $P_{c,\epsilon}$ on $[a,b]$ and then extend it up to $M_{a,b}$ using Lemma \ref{lemma:elliptic_extension}).

It remains to prove \ref{item:multiplicity_non_spectral}. Let $c_1 \in \mathbb{C}$ and consider $\psi \in \mathfrak{S}_{c_1,\epsilon}(\overline{M}_{a,b})$. By assumption, there is an integer $n \geq 0$ and smooth functions $\psi_0,\dots, \psi_n$ on $\overline{M}_{a,b}$ satisfying Dirichlet and Neumann boundary conditions and with $P_{c_1,\epsilon} \psi_0 = 0, P_{c_1,\epsilon} \psi_{k} = (\partial_x^2 - \alpha^2) \psi_{k-1}$ for $k = 1,\dots,n$ and $\psi_n = \psi$. Using Lemma \ref{lemma:elliptic_extension} inductively, we find that $\psi_0,\dots, \psi_n$ all have extensions to $\mathcal{W}_r^{a,b}$, holomorphic in the interior of $\mathcal{W}_r^{a,b}$. Since these extensions satisfy the same equations, we find that the restriction of $u$ to $[a,b]$ is an element of $\mathfrak{S}_{c_1,\epsilon}([a,b])$. We can do the same in the other direction, and thus this procedure define a linear isomorphism between $\mathfrak{S}_{c_1,\epsilon}(\overline{M}_{a,b})$ and $\mathfrak{S}_{c_1,\epsilon}([a,b])$, in particular they have the same dimension. Recall that $c_1 \in \Sigma_\epsilon$ if and only if $P_{c_1,\epsilon} : H^{k+4}_{\DN}(M_{a,b}) \to H^k(M_{a,b})$ is injective, and thus $\mathfrak{S}_{c_1,\epsilon}(\overline{M}_{a,b})$ is non-trivial if and only if $c_1 \notin \Sigma_\epsilon$.

It follows from \ref{item:fredholm_complex} and Fredholm analytic theory \cite[Appendix C]{dyatlov_zworski_book} if that $c \mapsto P_{c,\epsilon}^{-1}$ is a meromorphic family of operators from $H^k(M_{a,b})$ to $H^{k+4}_{\DN}(M_{a,b})$, with residues of finite rank. Let $A$ be the residue of $c \mapsto P_{c,\epsilon}^{-1}$ at $c = c_1$. Notice that we have $(c - (\partial_x^2 - \alpha^2)^{-1} P_{0,\epsilon})^{-1} = - P_{c,\epsilon}^{-1}(\partial_x^2 - \alpha^2)$, where $(\partial_x^2 - \alpha^2)^{-1} P_{0,\epsilon}$ is seen as an operator from $H^{k+4}_{\DN}(M_{a,b})$ to $H^{k+2}_{\D}(M_{a,b})$ (or as a closed operator on $H^{k+2}_{\D}(M_{a,b})$ with domain $H^{k+4}_{\DN}(M_{a,b})$). Hence, $-A(\partial_x^2 - \alpha^2)$ is the residue of the resolvent of $(\partial_x^2 - \alpha^2)^{-1} P_{0,\epsilon}$ at $c = c_1$. It is then standard \cite[III.6.5]{kato_perturbation_theory} that $-A(\partial_x^2 - \alpha^2)$ is a spectral projector for $(\partial_x^2 - \alpha^2)^{-1} P_{0,\epsilon}$. Now, if $\psi \in \mathfrak{S}_{c_1,\epsilon}(\overline{M}_{a,b})$, then $\psi$ is a generalized eigenvector for $(\partial_x^2 - \alpha^2)^{-1}P_{0,\epsilon}$ associated to the eigenvalue $c_1$. Hence, $\psi$ belongs to the range of $A$ (actually, we have $\psi = - A(\partial_x^2 - \alpha^2) \psi$, notice that $\psi$ satisfies Dirichlet boundary condition). Consequently, we find that $\mathfrak{S}_{c_1,\epsilon}(\overline{M}_{a,b})$ is contained in the range of $A$, and is thus finite dimensional.
\end{proof}

\begin{remark}\label{remark:nonspectral_spectral_theory}
Let $\epsilon > 0$ and $c_1 \in \mathbb {C}$. In the notations of the proof of Proposition \ref{proposition:elliptic_eigenvalues}, we proved that $\mathfrak{S}_{c_1,\epsilon}(\overline{M}_{a,b})$ is contained in the range of the operator $A$. Actually, since the range of $A$ is finite dimensional and a generalized eigenspace for $(\partial_x^2 - \alpha^2)^{-1}P_{0,\epsilon}$, we find that $\mathfrak{S}_{c_1,\epsilon}(\overline{M}_{a,b})$ is exactly the range of $A$.

Notice also that, using the relation $(c - Q_{0,\epsilon})^{-1} = - (\partial_x^2 - \alpha^2) P_{c,\epsilon}^{-1}$, we find that $ - (\partial_x^2 - \alpha^2) A$ is the spectral projector associated to the eigenvalue $c_1$ for the operator $Q_{0,\epsilon} = P_{0,\epsilon} (\partial_x^2 - \alpha ^2)^{-1}$ (with the boundary condition defined in Remark \ref{remark:sigmaepsilon_as_spectrum}). Hence, we retrieve that the multiplicity of $c_1$ as an element of $\Sigma_\epsilon$ is the multiplicity of $c_1$ as an eigenvalue of $Q_{0,\epsilon}$ on $\overline{M}_{a,b}$.
\end{remark}

\begin{remark}
It follows from Lemma \ref{lemma:non_zero_function} that there is $\delta > 0$ such that for every $c \in (c_0 - \delta,c_0 + \delta) + i (- \delta, + \infty)$ the function $U - c$ does not vanish on $\overline{M}_{a,b}$. Hence, the operator $P_{c,0}$ is elliptic on $\overline{M}_{a,b}$. Consequently, $c \mapsto P_{c,0}$ is a holomorphic family on $(c_0 - \delta, c_0 + \delta) + i (- \delta, + \infty)$ of Fredholm operators $H^{k+2}_{\D}(M_{a,b})\to H^k(M_{a,b})$ for every $k > -3/2$, and we can prove that there is a value of $c$ for which it is invertible. Hence, we get a meromorphic family of operators $c \mapsto P_{c,0}^{-1}$ on $(c_0 - \delta, c_0 + \delta) + i (- \delta, + \infty)$. The poles of this meromorphic family are the elements of the set $\mathcal{R}$ from Theorem \ref{theorem:limit}. However, we need to have more precise estimates before being able to prove Theorem~\ref{theorem:limit}: we need to prove in some sense that the $P_{c,\epsilon}$'s are ``uniformly Fredholm as $\epsilon$ goes to $0$''.
\end{remark}

\section{Technical estimates}\label{section:technical_estimates}

This section is dedicated to the proof of technical estimates required for the proof of Theorem \ref{theorem:limit}. We want to prove in some sense that there is $\delta > 0$ such that for $c \in (c_0 - \delta,c_0 + \delta) + i (- \delta, + \infty)$ the operator $P_{c,\epsilon}$ is ``uniformly Fredholm as $\epsilon$ goes to $0$''. This will be achieved by giving explicit bounds for the inverse of the operator $P_{c,\epsilon} + U''$, which is a compact perturbation of $P_{c,\epsilon}$, see Proposition \ref{proposition:invertible_compact_perturbation}. We start by working on the circle in \S \ref{subsection:modified_operator_circle} and then deduce a result on the segment in \S \ref{subsection:modified_operator_segment}.

\subsection{Inverting a modified operator on the circle.}\label{subsection:modified_operator_circle}

Recall that in this section we are using the notations that we introduced in the previous sections. We want to invert the operator $P_{c,\epsilon} + U''$ with uniform bounds as $\epsilon$ goes to $0$, where the extension of $P_{c,\epsilon}$ on the circle is defined by \eqref{eq:Pextension_circle}. It will be convenient to work instead with the operator $\widetilde{Q}_{c,\epsilon} = (P_{c,\epsilon} + U'')(\partial_x^2 - \alpha^2)^{-1}$, which is given explicitly by the formula
\begin{equation*}
   \widetilde{Q}_{c,\epsilon} = {i\alpha^{-1} \epsilon^2}(\partial_x^2 - \alpha^2) + U - c + i q
\end{equation*}
for $c \in  \mathbb{C}$ and $\epsilon \geq 0$. Once again, the coefficients of the $\widetilde{Q}_{c,\epsilon}$'s are in $\mathcal{B}_{K,r}$, and thus the operators $\widetilde{Q}_{c,\epsilon}$ act on $\mathcal{D}'(M)$, as explained in \S \ref{section:complex_scaling}. The main result of this subsection is the following invertibility bound:

\begin{lemma}\label{lemma:invertibility_full_circle}
There are $\epsilon_0,\delta > 0$ such that for every $k \in \mathbb{R}$ :
\begin{enumerate}[label=(\roman*)]
\item for every $c \in (c_0 - \delta, c_0 + \delta) + i (- \delta, + \infty)$, the operator $\widetilde{Q}_{c,0} : H^k(M) \to H^k(M)$ is invertible; \label{item:inverse_zero}
\item for every $c \in (c_0 - \delta, c_0 + \delta) + i (- \delta, + \infty)$ and $\epsilon \in (0,\epsilon_0]$ the operator $\widetilde{Q}_{c,\epsilon} : H^{k+2}(M) \to H^k(M)$ is invertible; \label{item:existence_inverse}
\item \label{item:bound_inverse} if $W$ is a compact subset of $(c_0 - \delta, c_0 + \delta) + i (- \delta, + \infty)$ then there is a constant $C > 0$ such that for every $c \in W, \epsilon \in (0,\epsilon_0]$ and $\psi \in H^{k+2}(M)$
\begin{equation*}
    \n{\psi}_{H^k(M)} + \epsilon^2 \n{\psi}_{H^{k+2}(M)} \leq C \n{\widetilde{Q}_{c,\epsilon} \psi}_{H^k(M)}.
\end{equation*}
\end{enumerate}
\end{lemma}

\begin{proof}
The operator $\widetilde{Q}_{c,0}$ is just the operator of multiplication by the smooth function $U - c + i q$. It follows from Lemma \ref{lemma:non_zero_function} that this function does not vanish on $M$ if $c \in (c_0 - \delta, c_0 + \delta) + i (- \delta, + \infty)$ with $\delta$ small enough, and \ref{item:inverse_zero} follows.

Let us fix any smooth volume $\mu$ on $M$, and use it to define the norms on Sobolev spaces on $M$. Let us choose a partition of unity on $M$ made of two functions $\chi_1$ and $\chi_2$ (see Figure \ref{fig:contour}) such that there is $\nu > 0$ for every $c \in (c_0 - \delta,c_0 + \delta) + i (- \delta, + \infty)$ and $x \in \mathbb{S}^1$ we have $\re (e^{ \frac{i\pi}{4}}(U(x) - c + i q(x)) \geq \nu ( 1 + |\im c|)$ if $x$ is in the support of $\chi_1$ and $\re (e^{ \frac{3i\pi}{4}}(U(x) - c + i q(x)) \geq \nu ( 1 + |\im c|)$ if $x$ is in the support of $\chi_2$. The existence of such a partition of unity is ensured by Lemma \ref{lemma:non_zero_function}, provided $\delta$ is small enough (this is the last time we need to eventually reduce the value of $\delta$ in this proof).

Let $B$ be an elliptic self-adjoint invertible pseudodifferential operator of order $k$ on $M$ (using that $M$ is diffeomorphic to a circle such an operator is easily constructed, e.g. as a Fourier multiplier). For $\psi \in H^{k+2}(M)$, let us compute
\begin{equation*}\begin{split}
\re \left( e^{\frac{i\pi}{4}} \langle \widetilde{Q}_{c,\epsilon} \chi_1 B\psi , \chi_1 B \psi \rangle_{L^2(M)} \right) = & \epsilon^2 \langle \re ( e^{i \frac{\pi}{4}} i\alpha^{-1} (\partial_x^2 - \alpha^2)) \chi_1 B \psi , \chi_1 B \psi \rangle_{L^2(M)} \\
& + \int_{M} \re( e^{i \frac{\pi}{4}} (U - c + i q)) |\chi_1 B \psi |^2 \mathrm{d}\mu.
\end{split}\end{equation*}
It follows from Lemma \ref{lemma:non_zero_function} that the principal symbol of the operator $\re ( e^{i \frac{\pi}{4}} i\alpha^{-1} (\partial_x^2 - \alpha^2))$ is non-negative\footnote{The real part of an operator is defined by $\re A = (A + A^*)/2$. We are using here the $L^2$ adjoint associated to the volume form $\mu$ on $M$.}. It follows from the sharp G{\aa}rding inequality \cite[Theorem 9.11]{zworski_book} that
\begin{equation*}
\left\langle \re \left( e^{i \frac{\pi}{4}} i\alpha^{-1} (\partial_x^2 - \alpha^2)\right) \chi_1 B \psi , \chi_1 B \psi \right\rangle _{L^2(M)} \geq - C \n{\chi_1 B \psi}_{H^{1}(M)} \n{\chi_1 B\psi}_{L^2(M)},
\end{equation*}
for some constant $C$ (that does not depend on $\epsilon$ or $c$). Thus, we have
\begin{equation*}\begin{split}
& \re \left( e^{\frac{i\pi}{4}} \langle \widetilde{Q}_{c,\epsilon} \chi_1 B\psi , \chi_1 B \psi \rangle_{L^2(M)} \right) \\
& \geq \nu (1 + |\im c|)\n{\chi_1 B \psi}_{L^2}^2 - C \epsilon^2 \n{\chi_1 B \psi}_{H^{1}(M)} \n{\chi_1 B\psi }_{L^2(M)}.
\end{split}\end{equation*}
Applying Cauchy--Schwartz inequality, we find that
\begin{equation*}
\n{\chi_1 B \psi }_{L^2(M)} \leq \frac{C}{1 + |\im c|} \left( \n{\widetilde{Q}_{c,\epsilon} \chi_1 B \psi}_{L^2(M)} + \epsilon^2 \n{\chi_1 B \psi }_{H^1(M)} \right),
\end{equation*}
for some constant $C > 0$ that does not depend on $c$ or $\epsilon$. We notice then that $\widetilde{Q}_{c,\epsilon} \chi_1 B \psi = \chi_1 B \widetilde{Q}_{c,\epsilon} \psi + [\widetilde{Q}_{0,\epsilon}, \chi_1 B] \psi$. Since $[\widetilde{Q}_{0,\epsilon}, \chi_1 B]$ is the sum of an operator of order $k-1$ (that does not depend on $\epsilon$) and $\epsilon^2$ times an operator of order $k +1$, we find, with some constant $C > 0$ that does not depend on $\epsilon$ or $c$, that
\begin{equation*}
\n{\chi_1 B \psi}_{L^2(M)} \leq \frac{C}{1 + |\im c|} \left( \n{\widetilde{Q}_{c,\epsilon} \psi}_{H^k(M)} + a_0\n{\psi}_{H^{k-1}(M)} + \epsilon^2 \n{\psi}_{H^{k+1}(M)} \right).
\end{equation*}
Here, $a_0 = 1$ if $k \neq 0$ and $a_0 = 0$ if $k = 0$ (indeed, in the case $k = 0$, we can take $B$ to be the identity, and $[\widetilde{Q}_{0,\epsilon},\chi_1]$ is just $\epsilon^2$ times an operator of order $1$). We prove in the same way an estimate for $\n{\chi_2 B \psi}_{L^2}$ and deduce that
\begin{equation}\label{eq:lower_derivatives}
\n{\psi}_{H^k(M)} \leq \frac{C}{1 + |\im c|} \left( \n{\widetilde{Q}_{c,\epsilon} \psi}_{H^k(M)} + a_0 \n{\psi}_{H^{k-1}(M)} + \epsilon^2 \n{\psi}_{H^{k+1}(M)} \right).
\end{equation}
Recall that the operator $\partial_x^2 - \alpha^2$ is invertible from $H^{k+2}(M)$ to $H^k(M)$, and thus
\begin{equation*}
\epsilon^2 \psi = - i \alpha (\partial_x^2 - \alpha^2)^{-1}\left(\widetilde{Q}_{c,\epsilon}\psi - (U - c + i q) \psi\right).
\end{equation*}
Hence, for some $C > 0$ that does not depend on $c$ or $\epsilon$, we have
\begin{equation}\label{eq:higher_derivatives}
\epsilon^2 \n{\psi}_{H^{k+2}(M)} \leq C \n{\widetilde{Q}_{c,\epsilon} \psi}_{H^k (M)} + C( 1 + |\im c|) \n{\psi}_{H^k(M)}.
\end{equation}

We are now ready to prove \ref{item:existence_inverse}. Notice that, for $\epsilon > 0$, the operator $\widetilde{Q}_{c,\epsilon}$ coincides with $- i\alpha^{-1} \epsilon^2 (\partial_x^2 - \alpha^2)$ up to an operator of order $0$. Since $\partial_x^2 - \alpha^2 : H^{k+2}(M) \to H^k(M)$ is invertible, we find that $\widetilde{Q}_{c,\epsilon} : H^{k+2}(M) \to H^k(M)$ is Fredholm of index zero. Hence, we only need to prove that it is injective in order to prove that it is invertible. Moreover, we only need to do it in the case $k = 0$ (by ellipticity, any element of the kernel would be in $C^\infty(M)$, and in particular in $H^2(M)$). So, let $\psi \in H^2(M)$ be such that $\widetilde{Q}_{c,\epsilon}\psi =0$. The estimates \eqref{eq:lower_derivatives} and \eqref{eq:higher_derivatives} become
\begin{equation*}
\n{\psi}_{L^2(M)} \leq \frac{C \epsilon^2}{1 + |\im c |}\n{\psi}_{H^{1}(M)} \textup{ and } \epsilon^2 \n{\psi}_{H^2(M)} \leq C (1 + |\im c |) \n{\psi}_{L^2(M)}.
\end{equation*}
Interpolating between these two estimates, we find that
\begin{equation*}
\epsilon \n{\psi}_{H^1(M)} \leq C \n{\psi}_{L^2(M)},
\end{equation*}
and thus, with a new constant $C > 0$, that does not depend on $c, \epsilon$ or $\psi$,
\begin{equation*}
\n{\psi}_{L^2(M)} \leq \frac{C \epsilon^2}{1 + |\im c |}\n{\psi}_{L^2(M)}.
\end{equation*}
Hence, if $\epsilon$ is small enough so that $C\epsilon^2/(1 + |\im c|) < 1$, we find that $\psi = 0$, proving \ref{item:existence_inverse}.

It remains to prove \ref{item:bound_inverse}. Let $W$ be a compact subset of $(c_0 - \delta,c_0 + \delta) + i (- \delta, + \infty)$. We will prove the result by contradiction. If \ref{item:bound_inverse} does not hold, then there are sequences $(\epsilon_n)_{n \geq 0}, (c_n)_{n \geq 0}$ and $(\psi_n)_{n \geq 0}$ of elements respectively of $(0,\epsilon_0], W$ and $H^{k+2}(M)$ such that $\n{\psi_n}_{H^k(M)} + \epsilon_n^2 \n{\psi_n}_{H^{k+2}(M)} = 1$ for every $n \geq 0$ but $\n{\widetilde{Q}_{c_n,\epsilon_n} \psi_n}_{H^k(M)}$ goes to $0$ as $n$ goes to $+ \infty$. Up to extracting, we may assume that $(\epsilon_n)_{n \geq 0}$ and $(c_n)_{n \geq 0}$ converge respectively to $\epsilon_\infty \in [0,\epsilon_0]$ and $c_\infty \in W$. 

Let us deal first with the case $\epsilon_\infty > 0$. It follows from $\n{\psi_n}_{H^k(M)} + \epsilon_n^2 \n{\psi_n}_{H^{k+2}(M)} = 1$ that $(\psi_n)_{n \geq 0}$ is uniformly bounded in $H^{k+2}(M)$. We can consequently extract a subsequence in order to assume that $(\psi_n)_{n \geq 0}$ converges in $H^{k+1}(M)$ to some distribution $\psi_\infty$. Notice that this limit satisfies $Q_{c_\infty, \epsilon_\infty} \psi_\infty = 0$, and thus $\psi_\infty = 0$ by \ref{item:existence_inverse}. It follows then from \eqref{eq:higher_derivatives} that $(\psi_n)_{n \geq 0}$ converges to $0$ in $H^{k+2}(M)$, which implies that $\n{\psi_n}_{H^k(M)} + \epsilon_n^2 \n{\psi_n}_{H^{k+2}(M)} = 1$ goes to $0$ as $n$ goes to $+ \infty$, a contradiction.

Let us deal now with the case $\epsilon_\infty = 0$. By an interpolation inequality, we get that, for some $C > 0$ that does not depend on $n$, we have 
$$\epsilon_n \n{\psi_n}_{H^{k+1}(M)} \leq C(\n{\psi}_{H^k(M)} + \epsilon_n^2 \n{\psi_n}_{H^{k+2}(M)}).$$
Plugging this inequality in \eqref{eq:lower_derivatives} and then applying \eqref{eq:higher_derivatives} to bound $\epsilon_n^2 \n{\psi_n}_{H^{k+2}(M)}$, we find that for $n$ large enough, we have, for some new $ C > 0$,
\begin{equation}\label{eq:better_low_derivatives}
\n{\psi_n}_{H^k} \leq C \left( \n{\widetilde{Q}_{c_n,\epsilon_n} \psi_n}_{H^k(M)} + \n{\psi_n}_{H^{k-1}(M)} \right),
\end{equation}
and plugging this estimate into \eqref{eq:higher_derivatives}, we also get
\begin{equation}\label{eq:better_high_derivatives}
\epsilon_n^2 \n{\psi_n}_{H^{k+2}(M)} \leq C \left( \n{\widetilde{Q}_{c_n,\epsilon_n} \psi_n}_{H^k(M)} + \n{\psi_n}_{H^{k-1}(M)} \right).
\end{equation}
Since $(\psi_n)_{n \geq 0}$ is uniformly bounded in $H^k(M)$, we can assume that it converges in $H^{k-1}(M)$ to a distribution $\psi_\infty \in \mathcal{D}'(M)$ that satisfies $Q_{c_\infty,0} \psi_\infty = 0$, and thus $\psi_\infty = 0$ by \ref{item:inverse_zero}. It follows then from \eqref{eq:better_low_derivatives} and \eqref{eq:better_high_derivatives} that $\n{\psi_n}_{H^k(M)} + \epsilon_n^2 \n{\psi_n}_{H^{k+2}(M)}$ goes to $0$ as $n$ goes to $+ \infty$, a contradiction.
\end{proof}

Let us translate Lemma \ref{lemma:invertibility_full_circle} into a statement about the operator $P_{c,\epsilon}$.

\begin{lemma}\label{lemma:invertibility_P_full_circle}
There are $\delta,\epsilon_0 > 0$ such that for every $k \in \mathbb{R}$ we have:
\begin{enumerate}[label=(\roman*)]
\item for every $c \in (c_0 - \delta, c_0 + \delta) + i (- \delta, + \infty)$ and $\epsilon \in (0,\epsilon_0)$ the operator $P_{c,\epsilon} + U'' : H^{k+4}(M) \to H^k(M)$ is invertible;
\item if $W$ is a compact subset of $(c_0 - \delta, c_0 + \delta) + i (- \delta, + \infty)$, there is a constant $C > 0$ such that for every $c \in W, \epsilon \in (0,\epsilon_0)$ and $\psi \in H^{k+4}(M)$
\begin{equation*}
    \n{\psi}_{H^{k+2}(M)} + \epsilon^2 \n{\psi}_{H^{k+4}(M)} \leq C \n{(P_{c,\epsilon} + U'') \psi}_{H^k(M)};
\end{equation*}
\item for every $c \in (c_0 - \delta, c_0 + \delta) + i (- \delta, + \infty)$, the operator $P_{c,0} + U'' : H^{k+2}(M) \to H^k(M)$ is invertible.
\end{enumerate}
\end{lemma}

\begin{proof}
Recall from Lemma \ref{lemma:inverse_laplace_circle} that the operator $\partial_x^2 - \alpha^2: H^{k+4}(M) \to H^{k+2}(M)$ is invertible. Consequently, the first three points follow directly from Lemma \ref{lemma:invertibility_full_circle} using the factorization $P_{c,\epsilon} + U'' = \widetilde{Q}_{c,\epsilon}(\partial_x^2 - \alpha^2)$.
\end{proof}

\subsection{Inverting a modified operator on \texorpdfstring{$[a,b]$}{[a,b]}}\label{subsection:modified_operator_segment}

We want now to go back to the problem on $[a,b]$. The main result of this section is the following invertibility estimate which is central in the proof of Theorem \ref{theorem:limit}. We use here the notations introduced in \S \ref{subsection:complex_scaling_segment} to apply the complex deformation method on a segment.

\begin{proposition}\label{proposition:invertible_compact_perturbation}
There is $\delta > 0$ such that:
\begin{enumerate}[label=(\roman*)]
\item for every $c \in (c_0 - \delta,c_0 + \delta) + i( - \delta, + \infty)$ and every $k > - 3/2$ the operator $P_{c,0} + U'' : H^{k+2}_{\D}(M_{a,b}) \to H^k(M_{a,b})$ is invertible; \label{item:Fredholm_inviscid}
\item for every compact subset $W$ of $(c_0 - \delta,c_0 + \delta) + i( - \delta, + \infty)$ there is $\epsilon_0> 0$ such that for every $k > - 5/2, c \in W$ and $\epsilon \in (0,\epsilon_0)$, the operator $P_{c,\epsilon} + U'' : H^{k+4}_{\DN}(M_{a,b}) \to H^k(M_{a,b})$ is invertible; \label{item:Fredholm_viscous}
\item for every compact subset $W$ of $(c_0 - \delta,c_0 + \delta) + i( - \delta, + \infty)$ there is $\epsilon_0> 0$ such that for every $k \in (-3/2, - 1/2)$ there is $C > 0$ such that for every $c \in W, \epsilon \in (0,\epsilon_0)$ and $\psi \in H^{k+4}_{\DN}(M_{a,b})$ we have \label{item:Fredholm_bound}
\begin{equation*}
\n{\psi}_{H^{k+2}(M_{a,b})} + \epsilon^2 \n{\psi}_{H^{k+4}(M_{a,b})} \leq C \n{(P_{c,\epsilon} + U'')\psi}_{H^k(M_{a,b})}.
\end{equation*}
\end{enumerate}
\end{proposition}

\begin{remark}
By ellipticity, the invertibility of $P_{c,0}+ U'' : H^{k+2}_{\mathrm{D}}(M_{a,b}) \to H^k(M_{a,b})$ and $P_{c,\epsilon} + U'' : H^{k+4}_{\DN}(M_{a,b}) \to H^k(M_{a,b})$ when $\epsilon > 0$ does not depend on the choice of $k$ (as long as it is large enough so that the boundary conditions make sense). However, if we want to get uniform invertibility estimates (as in \ref{item:Fredholm_bound}), we need to impose $k < - 1/2$. This is due to the presence of ``boundary layers'' in the inverse for $P_{c,\epsilon} + U''$ for $\epsilon > 0$ small. They are specific solutions to the equation $(P_{c,\epsilon} + U'')u = 0$, constructed in Proposition \ref{proposition:boundary_layers} below, that concentrate in a region of size $\epsilon$ near the extremities of $\overline{M}_{a,b}$. In order to get uniform estimates, we need to work in a space in which the boundary layers are not too large as $\epsilon$ goes to $0$, hence the restriction on $k$. Boundary layers will also appear in \S \ref{subsection:LVmethod} as they are central in the Vishik--Lyusternik method \cite{vishik_lyusternik_OG} that we use for the proof of Theorem \ref{theorem:limit}.
\end{remark}

Proposition \ref{proposition:invertible_compact_perturbation} requires some preparation. We start by stating a result that is useful in order to control boundary conditions.

\begin{lemma}\label{lemma:bound_Dirichlet}
Let $x_0 \in M$. Let $k \in (- \frac{3}{2}, \frac{1}{2})$. There is a constant $C > 0$ such that for every $\psi \in H^{k+2}(M)$ and $\epsilon \in (0,1)$ we have
\begin{equation}\label{eq:bound_Dirichlet}
    |\psi(x_0)| \leq C \epsilon^{k - \frac{1}{2}} \left(\n{\psi}_{H^k(M)} + \epsilon^2 \n{\psi}_{H^{k+2}(M)}\right).
\end{equation}
\end{lemma}

\begin{proof}
Identifies $M$ with a circle (by any smooth diffeomorphism) and let $(c_\ell(\psi))_{\ell \in \mathbb{Z}}$ be the Fourier coefficients of $\psi$ under this identification. We have then
\begin{equation*}
    |\psi(x_0)| \leq \sum_{\ell \in \mathbb{Z}} |c_\ell(\psi)| \leq C \left( \sum_{\ell \in \mathbb{Z}} \frac{1}{(1 + |\ell|^{k} + \epsilon^2 |\ell|^{k+2})^2} \right)^{\frac{1}{2}} \left(\n{\psi}_{H^k(M)} + \epsilon^2 \n{\psi}_{H^{k+2}(M)}\right).
\end{equation*}
Splitting the sum into $\ell$'s such that $|\ell| \leq \epsilon^{-1}$ and $|\ell| > \epsilon^{-1}$, we find that 
\begin{equation*}
    \sum_{\ell \in \mathbb{Z}} \frac{1}{(1 + |\ell|^{k} + \epsilon^2 |\ell|^{k+2})^2} \underset{\epsilon \to 0}{=} \mathcal{O}(\epsilon^{2 k - 1}),
\end{equation*}
which ends the proof of the lemma.
\end{proof}

\begin{remark}
Notice that in Lemma \ref{lemma:bound_Dirichlet}, one can take $k > 1/2$, in which case the factor $\epsilon^{k- \frac{1}{2}}$ in \eqref{eq:bound_Dirichlet} is replaced by $1$. Indeed, using the notation from the proof, the function $\psi$ belongs to $H^k$ uniformly in $\epsilon$. In the case $k = 1/2$, one may check that the factor $\epsilon^{k-\frac{1}{2}}$ becomes $ 1+ |\log \epsilon|$.
\end{remark}

In order to deal with the degeneracy of Neumann boundary condition as $\epsilon$ goes to $0$, we will construct approximate solutions to $P_{c,\epsilon}u = 0$ called ``boundary layers''. Their main feature is that they are concentrated near the boundary points of $\overline{M}_{a,b}$ and have a Neumann boundary value which is significantly larger than their Dirichlet boundary value. More precisely:

\begin{proposition}\label{proposition:boundary_layers}
Let $c_1 \in \mathbb{C}$ be such that $U(a) \neq \re c_1$ and $U(b) \neq \re c_1$. Then there is $\nu > 0$ such that for every $c \in \mathbb{D}(c_1,\nu)$ and $\epsilon \in (0,1)$ there are $C^\infty$ functions $u_{c,\epsilon}, v_{c,\epsilon} : \overline{M}_{a,b} \to \mathbb{C}$ such that the following holds:
\begin{enumerate}[label=(\roman*)]
\item for every $c \in \mathbb{D}(c_1,\nu)$ and $\epsilon \in (0,1)$, we have $u_{c,\epsilon}(a)= v_{c,\epsilon}(b) = 1, u_{c,\epsilon} \equiv 0$ on a neighbourhood of $b$ and $v_{c,\epsilon} \equiv 0$ on a neighbourhood of $a$;\label{item:bl1}
\item for every $c \in \mathbb{D}(c_1,\nu)$, we have $\epsilon u_{c,\epsilon}'(a) \underset{\epsilon \to 0}{=} \sqrt{ i \alpha  (U(a) - c)} + \mathcal{O}(\epsilon)$ and $\epsilon v_{c,\epsilon}'(b) \underset{\epsilon \to 0}{=} \sqrt{ i \alpha (U(b) - c)} + \mathcal{O}(\epsilon)$, where the square roots have respectively a negative and a positive real part, and the convergence holds uniformly in $c$;\label{item:bl2}
\item for every $k \geq 0$, there is a constant $C > 0$ such that for every $c \in \mathbb{D}(c_1,\nu)$ and $\epsilon \in (0,1)$ we have\label{item:bl3}
\begin{equation*}
\n{u_{c,\epsilon}}_{H^k(M_{a,b})} \leq C \epsilon^{\frac{1}{2}-k} \textup{ and } \n{v_{c,\epsilon}}_{H^k(M_{a,b})} \leq C \epsilon^{\frac{1}{2}-k} ;
\end{equation*}
\item for every $c \in \mathbb{D}(c_1,\nu)$, we have\label{item:bl4}
\begin{equation*}
    (P_{c,\epsilon} + U'')(u_{c,\epsilon}) \underset{\epsilon \to 0}{=} \mathcal{O}(\epsilon^\infty) \textup{ and } (P_{c,\epsilon} + U'')(v_{c,\epsilon}) \underset{\epsilon \to 0}{=} \mathcal{O}(\epsilon^\infty) ,
\end{equation*}
where the bound is in $C^\infty(\overline{M}_{a,b})$ and uniform in $c$;
\item for every compact subset of $L$ of $\overline{M}_{a,b}$ that does not contain the point $a$, all derivatives of $u_{c,\epsilon}$ are $\mathcal{O}(\epsilon^\infty)$ uniformly in $L$ and in $c \in \mathbb{D}(c_1,\nu)$;\label{item:bl5}
\item for every compact subset of $L$ of $\overline{M}_{a,b}$ that does not contain the point $b$, all derivatives of $v_{c,\epsilon}$ are $\mathcal{O}(\epsilon^\infty)$ uniformly on $L$ and in $c \in \mathbb{D}(c_1,\nu)$.\label{item:bl6}
\end{enumerate}
\end{proposition}

The main technical ingredient in the proof of Proposition \ref{proposition:boundary_layers} is the following construction.

\begin{lemma}\label{lemma:small_solution}
Let $c_1 \in \mathbb{C}$ be such that $U(a) \neq \re c_1$. Then there is $\nu > 0$ such that for every $c \in \mathbb{D}(c_1,\nu)$ there is a $C^\infty$ function $\phi_{c} : (a - \nu, a + \nu) \to \mathbb{C}$ with $\phi_{c}(a) = 0$ and $\phi_{c}'(a)= \sqrt{i \alpha  (U(a) - c)}$, where we use the square root with negative real part, and for each $\epsilon \in (0,1)$ there is a $C^\infty$ function $g_{c,\epsilon} : (a-\nu,a + \nu)$ such that $g_{c,\epsilon}(a) = 1$ and
\begin{equation*}
    (P_{c,\epsilon} + U'') \left( g_{c,\epsilon} e^{\frac{\phi_{c}}{\epsilon}} \right) = e^{\frac{\phi_{c}}{\epsilon}} \mathcal{O}(\epsilon^\infty).
\end{equation*}
Here, the $\mathcal{O}$ holds in $C^\infty$ on $(a - \nu, a + \nu)$, uniformly in $c \in \mathbb{D}(c_1,\nu)$ and $\epsilon \in (0,1)$. Similarly, the functions $\phi_{c}$ and $g_{c,\epsilon}$ are $C^\infty$ uniformly in $c \in \mathbb{D}(c_1,\nu)$ and $\epsilon \in (0,1)$.
\end{lemma}

\begin{proof}
Notice that $\epsilon^2(P_{c,\epsilon} + U'')$ is a semi-classical differential operator for the small parameter $\epsilon$, with principal symbol $p(x,\xi) = i\alpha^{-1} \xi^4 - (U(x)-c)\xi^2$. Let $\xi_0(c)$ denote the square root with positive imaginary part of $- i \alpha (U(a) - c)$ (it exists if $c$ is close enough to $c_1$ because $- i \alpha (U(a) - c)$ remains away from $\mathbb{R}$). We have $p(a,\xi_0(c)) = 0$ and $\frac{\partial p}{\partial \xi}(a,\xi_0(c)) = 2 i \alpha^{-1} \xi_0^3(c) \neq 0$. Consequently, the functions $g_{c,\epsilon}$ and $\phi_{c}$ may be obtained by a standard WKB construction, see for instance \cite[\S 4.1]{sjostrand_non_self_adjoint}. Notice that the derivative of $\phi_c$ at $a$ is $i \xi_0(c)$ (we do not use the standard notation for the WKB construction and include the factor $i$ directly in the phase $\phi_c$), that is the square root of $i \alpha (U(a) - c)$ with negative real part.
\end{proof}

\begin{proof}[Proof of Proposition \ref{proposition:boundary_layers}]
We construct only the function $u_{c,\epsilon}$, the construction of $v_{c,\epsilon}$ is symmetric. Let $g_{c,\epsilon}$ and $\phi_c$ be as in Lemma \ref{lemma:small_solution}. Up to making $\nu$ smaller, we may assume that $m \equiv 0$ on $(a - \nu, a + \nu)$ and that $\re \phi_c'$ is negative on $(a - \nu,a + \nu)$. Let then $\theta$ be a $C^\infty$ supported in $(a- \nu, a + \nu)$ and such that $\theta \equiv 1$ near $a$. Define then $u_{c,\epsilon}$ by 
\begin{equation*}
u_{c,\epsilon}(x) = \begin{cases} \theta(x) g_{c,\epsilon}(x) e^{\frac{\phi_c(x)}{\epsilon}}, & \textup{ if } x \in [a,a+ \nu), \\
 0, & \textup{ if } x \in \overline{M}_{a,b} \setminus [a, a+ \nu). \end{cases}
\end{equation*}
The first two points follow directly from the definition of $u_{c,\epsilon}$. To prove \ref{item:bl3} just notice that, for $k \in \mathbb{N}$, the size of the $k$th derivative of $u_{c,\epsilon}$ near $a$ is comparable to the size of $\epsilon^{-k} e^{- \frac{x}{\epsilon}}$ near $0$ in $\mathbb{R}^+$, which proves the result when $k$ is an integer. The general case follows by interpolation.

The point \ref{item:bl4} follows from Lemma \ref{lemma:small_solution} near $a$. Away from $a$, the function $u_{c,\epsilon}$ and its derivatives are $\mathcal{O}(\epsilon^\infty)$ (either because it is identically zero or using that $\re \phi_c$ is negative on the support of $\theta$ except at $a$). This last remark also proves \ref{item:bl5}.
\end{proof}

We are now ready to prove Proposition \ref{proposition:invertible_compact_perturbation}.

\begin{proof}[Proof of Proposition \ref{proposition:invertible_compact_perturbation}]
Notice that the operator $P_{c,0} + U''$ is just $(U - c)(\partial_x^2 - \alpha^2)$ We know from the proof of \ref{item:inverse_zero} of Lemma \ref{lemma:invertibility_full_circle} that there is $\delta > 0$ such that for every $c \in (c_0 - \delta, c_0 + \delta) + i (- \delta,+ \infty)$ the function $U- c + iq$ does not vanish on $M$, and thus $U -c $ does not vanish on $\overline{M}_{a,b}$. Hence, the multiplication operator $U-c : H^{k}(M_{a,b}) \to H^{k}(M_{a,b})$ is invertible. We know from Lemma \ref{lemma:inverse_laplace_segment} that $\partial_x^2 - \alpha^2 : H^{k+2}_{\D}(M_{a,b}) \to H^k(M_{a,b})$ is invertible as well. Hence, $P_{c,0} + U'':  H^{k+2}_{\D}(M_{a,b}) \to H^k(M_{a,b})$ is invertible, proving \ref{item:Fredholm_inviscid}.

Let us now move to the proof of \ref{item:Fredholm_viscous}. Fix $k > - 5/2$ and $c_1 \in (c_0 - \delta,c_0 + \delta) + i (- \delta,+ \infty)$, where $\delta$ is given by Lemma \ref{lemma:invertibility_P_full_circle}. We also assume that $\delta$ is small enough so that $U(a)$ and $U(b)$ do not belong to $(c_0 - \delta,c_0 + \delta)$. Let $\nu > 0$ be small enough so that $\overline{\mathbb{D}}(c_1,\nu) \subseteq (c_0 - \delta,c_0 + \delta) + i (- \delta,+ \infty)$ and Proposition \ref{proposition:boundary_layers} applies. We want to invert for $P_{c,\epsilon} + U''$ for $c \in \mathbb{D}(c_1,\nu)$ and $\epsilon > 0$ small enough. Notice that $P_{c,\epsilon} + U'' - i\alpha^{-1} \epsilon^2 (\partial_x^2 - \alpha^2)^2$ is an operator of order $2$, and thus compact as an operator from $H^{k+4}_{\DN}(M_{a,b})$ to $H^k(M_{a,b})$. It follows from Lemma \ref{lemma:inverse_laplace_square} that $P_{c,\epsilon} + U''$ is Fredholm of index zero. Consequently, we only need to prove that $P_{c,\epsilon} + U'': H^{k+4}_{\DN}(M_{a,b}) \to H^k(M_{a,b})$ is surjective. 

Denote by $\rest_{k+4}$ the restriction operator from $H^{k+4}(M)$ to $H^{k+4}(M_{a,b})$ and by $\ext_k : H^k(M_{a,b}) \to H^k(M)$ a bounded extension operator. Let $\epsilon_0 > 0$ be given by Lemma \ref{lemma:invertibility_P_full_circle}. For $c \in \mathbb{D}(c_1,\nu)$ and $\epsilon \in (0,\epsilon_0)$, introduce the operator $W_{c,\epsilon}^0 : f \mapsto \mathrm{rest}_{k+4} \circ (P_{c,\epsilon}+ U'')^{-1}\circ \ext_k (f)$, where the inverse for $P_{c,\epsilon} + U''$ on the circle is given by Lemma \ref{lemma:invertibility_P_full_circle}. It is bounded from $H^k(M_{a,b})$ to $H^{k+4}(M_{a,b})$, and it is a right inverse for $P_{c,\epsilon} + U''$, but it does not take value in $H^{k+4}_{\DN}(M_{a,b})$ a priori. To fix that, we will use the functions $u_{c,\epsilon}$ and $v_{c,\epsilon}$ from Proposition \ref{proposition:boundary_layers}, as well as the functions $g_a$ and $g_b$ defined by $g_a(x) = e^{-  \alpha(x-a)}$ and $g_b(x) = e^{- \alpha(b-x)}$ for $x \in M_{a,b}$. A second attempt as an inverse for $P_{c,\epsilon} + U'' $ is going to be
\begin{equation*}
    W_{c,\epsilon}^1 : f \mapsto W_{c,\epsilon}^0 f - p_a(f) g_a - p_b(f) g_b - r_a(f) u_{c,\epsilon} - r_b(f) v_{c,\epsilon},
\end{equation*}
where the coefficients $p_a(f),p_b(f),r_a(f),r_b(f)$ must satisfy
\begin{equation}\label{eq:all_booundary_conditions}
\begin{pmatrix}
A & B \\ C & D
\end{pmatrix} \begin{pmatrix}
p_a(f) \\ p_b(f) \\ r_a(f) \\ r_b(f)
\end{pmatrix} = \begin{pmatrix}
W_{c,\epsilon}^0 f(a) \\ W_{c,\epsilon}^0 f(b) \\ \epsilon (W_{c,\epsilon}^0 f)'(a) \\ \epsilon (W_{c,\epsilon}^0 f)'(b)
\end{pmatrix}.
\end{equation}
The blocks $A,B,C,D$ are given by
\begin{equation*}\begin{split}
    A = & \begin{pmatrix}
    g_a(a) & g_b(a) \\ g_a(b) & g_b(b)
    \end{pmatrix} = \begin{pmatrix}
    1 & e^{ - \alpha(b-a)} \\ e^{- \alpha(b-a)} & 1
    \end{pmatrix}, \\
    B = & \begin{pmatrix}
    u_{c,\epsilon}(a) & v_{c,\epsilon}(a) \\ u_{c,\epsilon}(b) & v_{c,\epsilon}(b)
    \end{pmatrix} = I, \\
    C = & \epsilon \begin{pmatrix}
    g_a'(a) & g_b'(a) \\ g_a'(b) & g_b'(a)
    \end{pmatrix} = \epsilon \begin{pmatrix}
     - \alpha & \alpha e^{- \alpha(b-a)} \\ - \alpha e^{ - \alpha(b-a)} & \alpha
    \end{pmatrix} = \mathcal{O}(\epsilon), \\
    D = & \epsilon \begin{pmatrix}
    u_{c,\epsilon}'(a) & v_{c,\epsilon}'(a) \\
    u_{c,\epsilon}'(b) & v_{c,\epsilon}'(b)
    \end{pmatrix} = \begin{pmatrix} \epsilon u'_{c,\epsilon}(a) & 0 \\ 0 & \epsilon v'_{c,\epsilon}(b) \end{pmatrix}.
\end{split}\end{equation*}

Notice that the matrix $A$ is invertible. The matrix $D$ is diagonal and it follows from Proposition \ref{proposition:boundary_layers} that, for $\epsilon$ small enough, the diagonal terms of $D$ are bounded away from zero. Hence, $D$ is invertible (with uniformly bounded inverse). Thus, the matrix in \eqref{eq:all_booundary_conditions} is invertible when $\epsilon$ is small enough, which allows to define $p_a(f),p_b(f),r_a(f),r_b(f)$ as bounded linear forms on $H^k(M_{a,b})$. Notice also that the inverse of the matrix in \eqref{eq:all_booundary_conditions} is uniformly bounded as $\epsilon$ goes to $0$. In addition, it follows from our assumption that $k > - \frac{5}{2}$ that the operator $f \mapsto (W_{c,\epsilon}^0 f(a), W_{c,\epsilon}^0 f(b), \epsilon(W_{c,\epsilon}^0 f)'(a), \epsilon(W_{c,\epsilon}^0 f)'(b))$ is well-defined and bounded from $H^k(M_{a,b})$ to $\mathbb{C}^4$. Using Lemma \ref{lemma:invertibility_P_full_circle}, we find that the operator norm of this operator is $\mathcal{O}(\epsilon^{-2})$. 

The operator $W_{c,\epsilon}^1$ we defined maps $H^k(M_{a,b})$ into $H_{\DN}^{k+4}(M_{a,b})$. Moreover, we have that
\begin{equation*}
    (P_{c,\epsilon} + U'')W_{c,\epsilon}^1 = I  - (P_{c,\epsilon} + U'') u_{c,\epsilon} \otimes r_a - (P_{c,\epsilon} + U'') v_{c,\epsilon} \otimes r_b.
\end{equation*}
It follows then from Proposition \ref{proposition:boundary_layers} and the estimate on $r_a$ and $r_b$ we just got that the right hand side is $I + \mathcal{O}(\epsilon^\infty)$ acting on $H^k(M_{a,b})$, which proves that $P_{c,\epsilon} + U''$ has a right inverse for $\epsilon$ small, and thus is invertible. With this argument, it seems that how small $\epsilon$ has to be depend on $k$. However, since $P_{c,\epsilon}$ is elliptic, it is simultaneously invertible or not for all values of $k > - 5/2$.

It remains to prove \ref{item:Fredholm_bound}, so we assume $k \in ( - 3/2,- 1/2)$. We look at the inverse operator for $P_{c,\epsilon} + U''$ constructed above. Notice that the inverse for $(P_{c,\epsilon} + U'')W_{c,\alpha,\epsilon}^1$ is uniformly bounded on $H^k(M_{a,b})$ as $\epsilon$ goes to $0$ (because it is obtained by Neumann series). Hence, we only need to prove that there is $C> 0$ (locally bounded in $C$ and uniform as $\epsilon$ goes to $0$) such that
\begin{equation*}
\n{W_{c,\epsilon}^1 f}_{H^{k+2}(M_{a,b})} + \epsilon^2 \n{W_{c,\epsilon}^1 f}_{H^{k+4}(M_{a,b})} \leq C \n{f}_{\mathcal{H}^k(M_{a,b})}
\end{equation*}
for every $f \in H^{k}(M_{a,b})$. It follows from Lemma \ref{lemma:invertibility_P_full_circle} that this estimate is satisfied if we replace $W_{c,\epsilon}^1 f$ by $W_{c,\epsilon}^0 f$. Hence, we can focus on $p_a(f) g_a + p_b(f) g_b + r_a(f) u_{c,\epsilon} + r_b(f) v_{c,\epsilon}$.

In \eqref{eq:all_booundary_conditions}, the blocks $A$ and $D$ are uniformly invertible, the block $B$ is uniformly bounded and $C$ is $\mathcal{O}(\epsilon)$. Thus, we must have
\begin{equation*}
\begin{pmatrix}
A & B \\ C & D
\end{pmatrix}^{-1} = \begin{pmatrix}
\mathcal{O}(1) & \mathcal{O}(1) \\ \mathcal{O}(\epsilon) & \mathcal{O}(1)
\end{pmatrix}.
\end{equation*} 
It follows that, for $f \in \mathcal{H}_{\Lambda}^{k}(a,b)$, we have
\begin{equation*}
\max(|p_a(f)|,| p_b(f)|) \leq C(|W_{c,\epsilon}^0 f(a)| + |W_{c,\epsilon}^0 f(b)| + \epsilon |(W_{c,\epsilon}^0 f)'(a)| + \epsilon |(W_{c,\epsilon}^0 f)'(b)|)
\end{equation*}
and
\begin{equation*}
\max(|r_a(f)|,|r_b(f)|) \leq C\epsilon(|W_{c,\epsilon}^0 f(a)| + |W_{c,\epsilon}^0 f(b)| + |(W_{c,\epsilon}^0 f)'(a)| + |(W_{c,\epsilon}^0 f)'(b)|)
\end{equation*}
for some constant $C > 0$ that does not depend on $\epsilon$, and is locally bounded in $c$. It follows then from Lemma \ref{lemma:invertibility_P_full_circle} and the fact that $k +2 > 1/2$ that, for some $C > 0$ that does not depend on $\epsilon$, we have
\begin{equation*}
\max(|W_{c,\epsilon}^0 f(a)|, |W_{c,\epsilon}^0 f(b)|) \leq C \n{f}_{H^k(M_{a,b})}
\end{equation*}
and, using Lemma \ref{lemma:bound_Dirichlet},
\begin{equation*}
\max(|(W_{c,\epsilon}^0 f)'(a)|, |(W_{c,\epsilon}^0 f)'(b)|) \leq C \epsilon^{k+ \frac{1}{2}} \n{f}_{H^k(M_{a,b})}.
\end{equation*}
Thus, we have (using $k > - 3/2$)
\begin{equation*}
\max(|p_a(f)|,| p_b(f)|) \leq C \n{f}_{H^k(M_{a,b})}
\end{equation*}
and (using $k < - 1/2$) 
\begin{equation*}
\max(|r_a(f)|,|r_b(f)|) \leq C \epsilon^{\frac{3}{2}+k} \n{f}_{H^k(M_{a,b})}.
\end{equation*}

Using the estimate from \ref{item:bl3} in Proposition \ref{proposition:boundary_layers}, we finally find that
\begin{equation*}
\n{p_a(f) g_a + p_b(f) g_b + r_a(f) u_{c,\epsilon} + r_b(f) u_{c,\epsilon}}_{H^{k+2}(M_{a,b})} \leq C \n{f}_{H^k(M_{a,b})}
\end{equation*}
and
\begin{equation*}
\n{p_a(f) g_a + p_b(f) g_b + r_a(f) u_{c,\epsilon} + r_b(f) u_{c,\epsilon}}_{H^{k+4}(M_{a,b})} \leq C \epsilon^{-2} \n{f}_{H^{k}(M_{a,b})},
\end{equation*}
which concludes the proof \ref{item:Fredholm_bound}.
\end{proof}

\section{Resonances as inviscid limits}\label{section:definition_resonances}

This section is dedicated to the proof of Theorem \ref{theorem:limit}. In \S \ref{subsection:definition_resonances}, we define the set $\mathcal{R}$ from Theorem \ref{theorem:limit}, and prove that if $c$ is in the range of application of Proposition \ref{proposition:invertible_compact_perturbation} but not in $\mathcal{R}$ then $P_{c,\epsilon}$ is invertible on $\overline{M}_{a,b}$ for $\epsilon$ small enough. In \S \ref{subsection:LVmethod}, we give an asymptotic description of the inverse $P_{c,\epsilon}^{-1}$ as $\epsilon$ goes to $0$. Finally, we prove Theorem \ref{theorem:limit} in \S \ref{subsection:perturbation_resonances}.

From now on, and until the end of \S \ref{section:alternative_characterizations}, we let $\delta > 0$ be as in Proposition \ref{proposition:invertible_compact_perturbation}. Up to making $\delta$ smaller, we also assume that $\delta$ is as in Lemma \ref{lemma:non_zero_function} and that the following holds
\begin{itemize}
\item item \ref{item:non_zero_ii} in Lemma \ref{lemma:non_zero_function} is satisfied;
\item the function $U-c + i q$ does not vanish on $M$ for every $c \in (c_0 - \delta, c_0 + \delta) + i (- \delta, + \infty)$, which is possible thanks to Lemma \ref{lemma:non_zero_function};
\item for every $c \in (c_0 - \delta, c_0 + \delta)$, the number $c$ is distinct from $U(a)$ and $U(b)$ and $c$ is a regular value of $U_{|[a,b]}$.
\end{itemize}

\subsection{Definition of resonances}\label{subsection:definition_resonances}

Let us define for any $k > - 3/2$ the set
\begin{equation*}\begin{split}
\mathcal{R} = \{ c \in (c_0 - \delta, c_0 + \delta) + & i (- \delta, + \infty) :  \\
&  P_{c,0} : H^{k+2}_{\D}(M_{a,b}) \to H^k(M_{a,b}) \textup{ is not invertible} \}.
\end{split}\end{equation*}
The ellipticity of $P_{c,0}$ on $\overline{M}_{a,b}$ implies that $\mathcal{R}$ does not depend on the choice of $k$ in its definition. This is the set $\mathcal{R}$ from Theorem \ref{theorem:limit}. In particular, once Theorem \ref{theorem:limit} will be proven, we will find that the set $\mathcal{R}$ does not depend on the choice of the deformation from \S \ref{section:escape_function}, up to the value of $\delta$. See Figure \ref{fig:tau}. We start by checking that $\mathcal{R}$ is discrete.

\begin{figure}[t]
   \centering
   \begin{subfigure}{0.45\textwidth}
       \centering
       \includegraphics[width=\textwidth]{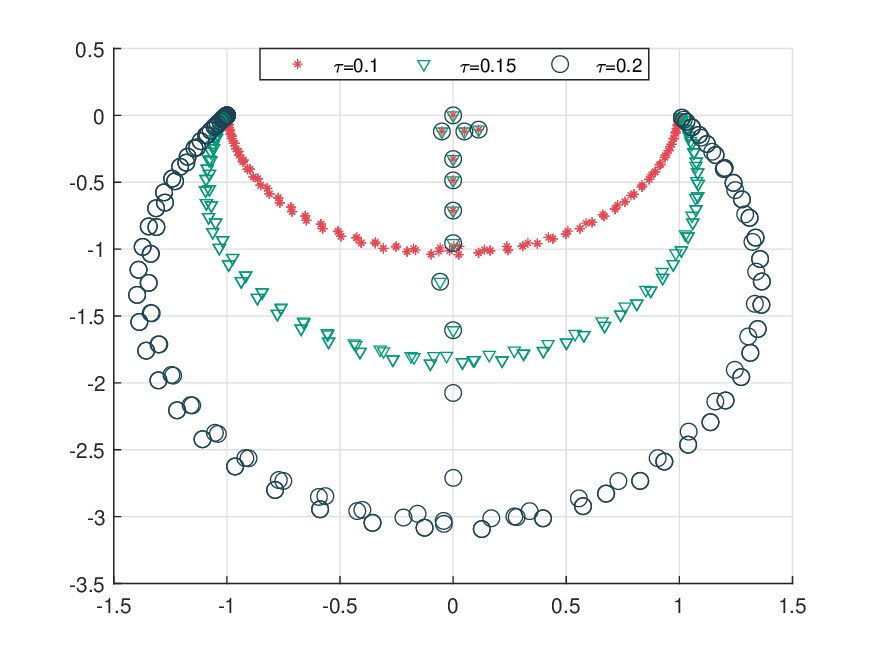}
       \caption{Resonances for a segment}
       \label{fig:sinusoid_segment_tau}
   \end{subfigure}
   \begin{subfigure}{0.45\textwidth}
       \centering
       \includegraphics[width=\textwidth]{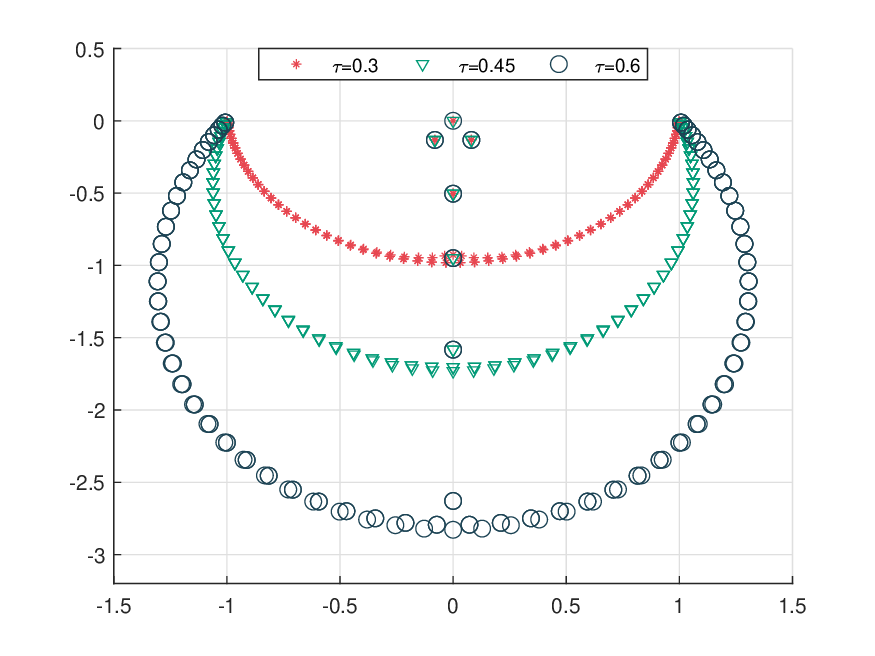}
       \caption{Resonances for a circle}
       \label{fig:sinusoid_circle_tau}
   \end{subfigure}
   \caption{Numerical computation of $\mathcal R$ with different choices of $\tau$ in the complex deformation (see \S \ref{section:escape_function} for the meaning of this parameter and Appendix \ref{section:matlab} for the numerical methods). Near $0$, the set $\mathcal R$ remains unchanged for different $\tau$. Lower curves in all colors are outside the scope of Theorem \ref{theorem:limit}. They correspond to values of the parameter $c$ for which Rayleigh equation is not elliptic on the spaces defined by complex deformation (this is the range of $U$ on the complex deformation). (A) $U(x)=\cos(3\pi x)$, $x\in [-1,1]$, $\alpha=\frac{\sqrt{35} \pi}{2}$. (B) $U(x)=\sin(3x)$, $x\in \mathbb R/2\pi\mathbb Z$, $\alpha=3$.
   }
   \label{fig:tau}
\end{figure}

\begin{lemma}\label{lemma:discrete_resonances}
The set $\mathcal{R}$ is discrete. For every $k > - 3/2$, the family of inverse operators $P_{c,0}^{-1} : H^k(M_{a,b}) \to H^{k+2}_{\D}(M_{a,b})$ is meromorphic in $c$. The coefficients of negative indexes in the Laurent expansion of $P_{c,0}^{-1}$ at a point of $\mathcal{R}$ have finite rank.
\end{lemma}

\begin{proof}
Choose any $k > - 3/2$. We know that $P_{c,0} + U'' : H^{k+2}_{\D}(M_{a,b}) \to H^k(M_{a,b})$ is invertible for every $c \in (c_0 - \delta, c_0 + \delta) + i (- \delta, + \infty)$, see Proposition \ref{proposition:invertible_compact_perturbation}. Since the multiplication operator $U'' : H^{k+2}_{\D}(M_{a,b}) \to H^k(M_{a,b})$ is compact, it follows that $c \mapsto P_{c,0}$ is a holomorphic family of Fredholm operator of index zero. The result is then a consequence of analytic Fredholm theory \cite[Appendix C]{dyatlov_zworski_book} if we can prove that there is $c \in (c_0 - \delta, c_0 + \delta) + i (-\delta, + \infty)$ such that $P_{c,0}$ is invertible. Using the factorization $P_{c,0} = Q_{c,0} (\partial_x^2 - \alpha^2)$ and Lemma \ref{lemma:inverse_laplace_segment}, we only need to prove that $Q_{c,0} = Q_{0,0} - c$ is invertible as an operator $H^k(M_{a,b}) \to H^{k}(M_{a,b})$ for some value of $c$. One only needs to take $c$ large enough since $Q_{0,0}$ is bounded on $H^k(M_{a,b})$.
\end{proof}

\begin{remark}\label{remark:multiplicity_resonances}
We will call the elements of $\mathcal{R}$ resonances. If $c_1 \in \mathcal{R}$, we can define resonant states associated to $c_1$: they are the elements of the range of the residue of $c \mapsto P_{c,\epsilon}^{-1}$ at $c = c_1$. Notice that the resonant states are elements of $C^\infty(\overline{M}_{a,b})$. Working as in \S \ref{subsection:basic_properties}, we find that the resonant states are exactly the element of $\mathfrak{S}_{c_1,0}(\overline{M}_{a,b})$ (this space is defined as in the case $\epsilon > 0$, dropping Neumann boundary condition). We will say more about them in \S \ref{section:description_resonant_states}. As in \S \ref{subsection:basic_properties}, one may check that the multiplicity of $c_1$ as an element of $\mathcal{R}$ is just the multiplicity of $c_1$ as an eigenvalue of $Q_{0,0}$.
\end{remark}

We will now prove that for $c \in (c_0 - \delta,c_0 + \delta) + i(- \delta, + \infty) \setminus \mathcal{R}$, the invertibility of the operator $P_{c,0}$ on $\overline{M}_{a,b}$ is preserved by the addition of the small but higher order operator $i\alpha^{-1} \epsilon^2(\partial_x^2 - \alpha^2)^2$. The main ingredient to deal with this singular perturbation is Proposition \ref{proposition:invertible_compact_perturbation} that allows us to reduce to the study of a family of bounded operators. Notice that Lemma \ref{lemma:invertibility_a_priori} with Proposition \ref{proposition:elliptic_eigenvalues} already proves \ref{item:no_resonance} from Theorem \ref{theorem:limit}.

\begin{lemma}\label{lemma:invertibility_a_priori}
Let $c_1 \in (c_0 - \delta, c_0 + \delta) + i (- \delta, + \infty)$ be such that $c_1 \notin \mathcal{R}$. Let $k \in (-3/2, - 1/2)$. There are $C,\epsilon_1 > 0$ such that for every $c \in \mathbb{D}(c_1,\epsilon_1)$ and $\epsilon \in [0,\epsilon_1)$ the operator $P_{c,\epsilon} : H^{k+4}_{\DN}(M_{a,b}) \to H^k(M_{a,b})$ (if $\epsilon > 0$) or $P_{c,0} : H^{k+2}_{\D}(M_{a,b}) \to H^k(M_{a,b})$ (if $\epsilon = 0$) is invertible. Moreover, for $f \in H^k(M_{a,b})$, we have
\begin{equation*}
\n{P_{c,\epsilon}^{-1} f}_{H^{k+2}(M_{a,b})} + \epsilon^2 \n{P_{c,\epsilon}^{-1} f}_{H^{k+4}(M_{a,b})} \leq C \n{f}_{H^{k}(M_{a,b})},
\end{equation*}
where it is understood that when $\epsilon = 0$ the term with a factor $\epsilon$ is set to $0$.
\end{lemma}

\begin{proof}
Let us start by writing
\begin{equation*}
    P_{c,\epsilon} = ( I - U'' (P_{c,\epsilon} + U'')^{-1})(P_{c,\epsilon} + U''),
\end{equation*}
where the second operator on the right maps $H^{k+4}_{\DN}(M_{a,b})$ (respectively $H^k_D(M_{a,b})$ when $\epsilon = 0$) into $H^{k}(M_{a,b})$ and the first one maps $H^{k}(M_{a,b})$ into itself. Since we know that $P_{c,\epsilon} + U''$ is invertible with this mapping property, we see that $P_{c,\epsilon}$ is invertible if and only if $ I - U'' (P_{c,\epsilon} + U'')^{-1} : H^{k}(M_{a,b}) \to H^{k}(M_{a,b})$ is invertible. In particular, we know that $I - U'' (P_{c_1,0} + U'')^{-1}$ is invertible. 

In this context, the resolvent identity becomes
\begin{equation}\label{eq:resolvent_identity}
\begin{split}
    & (P_{c_1,0} + U'')^{-1} - (P_{c,\epsilon} + U'')^{-1} \\
    & = i\alpha^{-1} \epsilon^2  (P_{c_1,0} + U'')^{-1}(\partial_x^2 - \alpha^2)^2 (P_{c,\epsilon} + U'')^{-1} \\ & \quad +  (c_1 - c)(P_{c_1,0} + U'')^{-1}(\partial_x^2 - \alpha^2) (P_{c,\epsilon} + U'')^{-1}.
\end{split}
\end{equation}
We start by estimating the first term in the right hand side. Let $k' \in (- 3/2,k)$. It follows from point \ref{item:Fredholm_bound} in Proposition \ref{proposition:invertible_compact_perturbation} and an interpolation estimate that the operator $\epsilon^{2 + k' - k} (\partial_x^2 - \alpha^2)^2 (P_{c,\epsilon} + U'')^{-1}$ is bounded uniformly in $\epsilon$ close to $0$ and $c$ close to $c_1$, as an operator from $H^{k}(M_{a,b})$ to $H^{k'}(M_{a,b})$. Since $(P_{c_1,0} + U'')^{-1}$ is bounded as an operator from $H^{k'}(M_{a,b})$  to $H^k(M_{a,b})$, we find that the first term in the right hand side of \eqref{eq:resolvent_identity} is of $\mathcal{O}(\epsilon^{k - k'})$ as an operator from $H^k(M_{a,b})$ to itself.

For the second term, we notice that $(\partial_x^2 - \alpha^2) (P_{c,\epsilon} + U'')^{-1}$ is bounded from $H^k(M_{a,b})$ to itself uniformly in $c$ and $\epsilon$, using point \ref{item:Fredholm_bound} in Proposition \ref{proposition:invertible_compact_perturbation} again. Since $(P_{c_1,0} + U'')^{-1}$ is bounded as an operator from $H^k(M_{a,b})$ to itself, we find that the second term is $\mathcal{O}(|c_1 - c|)$ as an operator from $H^k(M_{a,b})$ to itself.

Summing up, we find that $(P_{c_1,0} + U'')^{-1} - (P_{c,\epsilon} + U'')^{-1}$ is $\mathcal{O}(\epsilon^{k - k'} + |c_1 - c|)$ as an operator from $H^k(M_{a,b})$ to itself. Hence, if $\epsilon$ and $|c_1 - c|$ are small enough, we find that $I - U''(P_{c,\epsilon} + U'')^{-1}$ is invertible (by Neumann series), and thus that $P_{c,\epsilon} : H^{k+4}_{\DN}(M_{a,b}) \to H^k(M_{a,b})$ is invertible.

It remains to bound the operator norm of $P_{c,\epsilon}^{-1}$. To do so, we write
\begin{equation*}
    P_{c,\epsilon}^{-1} = (P_{c,\epsilon} + U'')^{-1} (I - U''(P_{c,\epsilon} + U'')^{-1})^{-1}.
\end{equation*}
The factor on the right is bounded uniformly, as an operator on $H^k(M_{a,b})$, when $\epsilon$ is small and $c$ near $c_1$ by the argument above. The factor on the left is dealt with using point \ref{item:Fredholm_bound} in Proposition \ref{proposition:invertible_compact_perturbation}.
\end{proof}

\subsection{Vishik--Lyusternik method}\label{subsection:LVmethod}

Let $c_1 \in (c_0 - \delta,c_0 + \delta) + i(- \delta,+ \infty) \setminus \mathcal{R}$. We want to get an approximation of $P_{c,\epsilon}^{-1}$ as $\epsilon$ goes to $0$ and for $c$ in a neighbourhood of $c_1$ (the existence of the inverse is guaranteed by Lemma \ref{lemma:invertibility_a_priori}). In order to deal with the additional boundary conditions that appear when $\epsilon > 0$, we apply the Vishik--Lyusternik method \cite{vishik_lyusternik_OG}. We start by constructing ``boundary layers'', i.e. quasimodes for $P_{c,\epsilon}$ supported near the boundary points of $\overline{M}_{a,b}$.

\begin{proposition}\label{proposition:boundary_layers_bis}
There is $\nu > 0$ such that for every $c \in \mathbb{D}(c_1,\nu)$ and $\epsilon \in (0,1)$ there are two $C^\infty$ functions $u_{c,\epsilon}^{\#}$ and $v^{\#}_{c,\epsilon}$ on $\overline{M}_{a,b}$ that satisfy points \ref{item:bl1}, \ref{item:bl2}, \ref{item:bl3}, \ref{item:bl5} and \ref{item:bl6} from Proposition \ref{proposition:boundary_layers}. The point \ref{item:bl4} from Proposition \ref{proposition:boundary_layers} is replaced by: for every $c \in \mathbb{D}(c_1,\nu)$, we have
\begin{equation*}
    P_{c,\epsilon}(u_{c,\epsilon}^{\#}) \underset{\epsilon \to 0}{=} \mathcal{O}(\epsilon^\infty) \textup{ and } P_{c,\epsilon}(v_{c,\epsilon}^{\#}) \underset{\epsilon \to 0}{=} \mathcal{O}(\epsilon^\infty) ,
\end{equation*}
where the bound is in $C^\infty(\overline{M}_{a,b})$ and uniform in $c$. For every $\epsilon \in (0,1)$, the functions $u_{c,\epsilon}^{\#}$ and $v^{\#}_{c,\epsilon}$ depend continuously on $c$ as elements of $C^\infty(\overline{M}_{a,b})$. Moreover, the quantities 
\begin{equation*}
\lambda_{c,\epsilon} \coloneqq  \frac{1}{\epsilon (u^{\#}_{c,\epsilon})'(a)} \textup{ and } \mu_{c,\epsilon} \coloneqq \frac{1}{\epsilon (v^{\#}_{c,\epsilon})'({b})}
\end{equation*}
have Taylor expansions at $\epsilon = 0$, uniformly in $c$.
\end{proposition}

The proof of Proposition \ref{proposition:boundary_layers_bis} is the same as Proposition \ref{proposition:boundary_layers}. The additional properties that we require for $u_{c,\epsilon}^{\#}$ and $v^{\#}_{c,\epsilon}$ follow readily from the proof\footnote{The Taylor expansions in $\epsilon$ for $\lambda_{c,\epsilon}$ and $\mu_{c,\epsilon}$ comes from the construction of $g_{c,\epsilon}$ in Lemma \ref{lemma:small_solution} as the sum of an asymptotic series of smooth functions. To see that the continuous dependence on $c$ goes through the summation of the asymptotic series, see \cite[Theorem 1.2.6]{hormander_book1} for instance.}.  Until the end of this subsection, we let $\nu > 0$ be small enough so that both Lemma \ref{lemma:invertibility_a_priori} and Proposition \ref{proposition:boundary_layers_bis} apply, and work with $c \in \mathbb{D}(c_1,\nu)$.

Now, if $f \in H^k(M_{a,b})$ with $k > 3/2$ and $\epsilon > 0$, let 
\begin{equation*}
    N_{c,\epsilon}f = \epsilon (\lambda_{c,\epsilon} f'(a) u^{\#}_{c,\epsilon} + \mu_{c,\epsilon} f'(b)  v^{\#}_{c,\epsilon}).
\end{equation*}
To approximate $P_{c,\epsilon}^{-1}$, we introduce a sequence of operators $(A_{c,\epsilon}^{(\ell)})_{\ell \geq 0}$ defined by the following inductive scheme:
\begin{equation}\label{eq:inductive_approximation}
    \begin{split}
        A_{c,\epsilon}^{(0)} & = P_{c,0}^{-1}, \\
        A_{c,\epsilon}^{(\ell+1)} & = P_{c,0}^{-1}\left( I - (P_{c,\epsilon} - P_{c,0})(I - N_{c,\epsilon}) A_{c,\epsilon}^{(\ell)}\right).
    \end{split}
\end{equation}
The idea behind this definition is the following: we use the same inductive scheme to invert the small perturbation $P_{c,\epsilon}$ of $P_{c,0}$ as in Neumann argument (or in the proof of the contraction mapping principle), but in addition we apply the operator $I - N_{c,\epsilon}$ at each step in order to satisfy (at least approximately) Neumann boundary condition. We still get a good approximation of $P_{c,\epsilon}^{-1}$ since the perturbation is made by elements of the kernel of $P_{c,\epsilon}$. We find by induction that for every $\ell \geq 0$ and $k > 1/2$ the operator $A_{c,\epsilon}^{(\ell)}$ is well-defined and bounded as an operator from $H^{k - 2 + 2 \ell}(M_{a,b})$ to $H^{k}(M_{a,b})$. Beware that we do not claim that this bound is uniform as $\epsilon$ goes to $0$. Indeed, due to the presence of boundary layers in the definition of $N_{c,\epsilon}$ the $A_{c,\epsilon}^{(\ell)}$'s have poor mapping properties as $\epsilon$ goes to $0$. In order to bypass this issue, we will give another description of the $A_{c,\epsilon}^{(\ell)}$'s.

To do so, let us introduce the operator $D$, define on $H^k(M_{a,b})$ for $k > 1/2$ by
\begin{equation*}
D(f)(x) = f(a) + \frac{x-a}{b-a}(f(b) - f(a)).
\end{equation*}
Introduce also the operator $E_{c,\epsilon} = \epsilon^{-1} D N_{c,\epsilon}$, which is explicitly given by the formula 
\begin{equation*}
E_{c,\epsilon} f (x) = f'(a) \lambda_{c,\epsilon} \frac{b-x}{b-a} + f'(b) \mu_{c,\epsilon} \frac{x-a}{b-a}.
\end{equation*}
The point of introducing the operators $D$ and $E_{c,\epsilon}$ is that
\begin{equation*}
P_{c,0}^{-1} P_{c,0} N_{c,\epsilon} = N_{c,\epsilon} - \epsilon (I - P_{c,0}^{-1} P_{c,0}) E_{c,\epsilon}.
\end{equation*}
This formula comes from the fact that if $f \in H^k(M_{a,b})$ with $k > 1/2$ then $(I-D)f \in H_{\D}^k(M_{a,b})$ and thus $P_{c,0}^{-1} P_{c,0} (I - D) f = (I- D) f$.

We can now rewrite the inductive scheme for the $A_{c,\epsilon}^{(\ell)}$'s as
\begin{equation*}
\begin{split}
A_{c,\epsilon}^{(\ell+1)}  = & P_{c,0}^{-1}\left( I - i\alpha^{-1} \epsilon^2 (\partial_x^2 - \alpha^2)^2 A_{c,\epsilon}^{(\ell)} \right) + \epsilon (I - P_{c,0}^{-1} P_{c,0}) E_{c,\epsilon} A_{c,\epsilon}^{(\ell)} \\ &   - N_{c,\epsilon} A_{c,\epsilon}^{(\ell)} + P_{c,0}^{-1} P_{c,\epsilon} N_{c,\epsilon} A_{c,\epsilon}^{(\ell)}.
\end{split}
\end{equation*}

This suggests to introduce a new sequence of operators $(B_{c,\epsilon}^{(\ell)})_{\ell \geq - 1}$ by
\begin{equation}\label{eq:formula_for_B}\begin{split}
    B_{c,\epsilon}^{(-1)} & =0, \ B_{c,\epsilon}^{(0)} = P_{c,0}^{-1}, \\
    B_{c,\epsilon}^{(\ell+1)} & = P_{c,0}^{-1}  + \left( \epsilon (I - P_{c,0}^{-1} P_{c,0}) E_{c,\epsilon} - i\alpha^{-1} \epsilon^2 P_{c,0}^{-1}(\partial_x^2 - \alpha^2)^2 \right) B_{c,\epsilon}^{(\ell)}
\end{split}\end{equation}
for $\ell \geq 0$. As above, we find that for every $\ell \geq 0$ and $k > 1/2$ the operator $B_{c,\epsilon}^{(\ell)}$ is well-defined and bounded as an operator from $H^{k - 2 + 2 \ell}(M_{a,b})$ to $H^{k}(M_{a,b})$.

We can now give an approximation for $P_{c,\epsilon}^{-1}$.

\begin{lemma}\label{lemma:description_approximation}
For every $\ell \geq 0$ and $k > \frac{1}{2}$, we have that:
\begin{enumerate}[label=(\roman*)]
\item the operator $B_{c,\epsilon}^{(\ell)}$ has a Taylor expansion at $\epsilon = 0$ as an operator from the space $H^{k - 2 + 2 \ell}(M_{a,b})$ to $H^{k}(M_{a,b})$; \label{item:expansion_partial}
\item the operator $A_{c,\epsilon}^{(\ell)} - B_{c,\epsilon}^{(\ell)} + N_{c,\epsilon} B_{c,\epsilon}^{(\ell - 1)}$ is of $\mathcal{O}(\epsilon^\infty)$ as an operator from the space $H^{k - 2 + 2 \ell}(M_{a,b})$ to $H^{k}(M_{a,b})$; \label{item:description_approximation}
\item if in addition $k < 3/2$, the operator $P_{c,\epsilon}^{-1} - A_{c,\epsilon}^{(\ell)}$ is of $\mathcal{O}(\epsilon^{\ell})$ as an operator from $H^{k + 2 \ell}(M_{a,b})$ to $H^{k}(M_{a,b})$. \label{item:quality_approximation}
\end{enumerate}
All these statements are uniform for $c \in \mathbb{D}(c_1,\nu)$.
\end{lemma}

\begin{proof}
The proof of \ref{item:expansion_partial} is just an induction on $\ell$ using the definition of $B_{c,\epsilon}^{(\ell)}$.

To prove \ref{item:description_approximation}, introduce the operator $C_{c,\epsilon}^{(\ell)} = A_{c,\epsilon}^{(\ell)} - B_{c,\epsilon}^{(\ell)} + N_{c,\epsilon} B_{c,\epsilon}^{(\ell - 1)}$, and notice that it satisfies $C_{c,\epsilon}^{(0)} = 0$ and (we use the fact that $N_{c,\epsilon}$ is a projector to prove this formula)
\begin{equation*}
\begin{split}
C_{c,\epsilon}^{(\ell + 1)} = & \left(\epsilon( I - P_{c,0}^{-1} P_{c,0}) E_{c,\epsilon} - i \alpha^{-1} \epsilon^2 P_{c,0}^{-1} (\partial_x^2 - \alpha^2)^2 - N_{c,\epsilon}\right) C_{c,\epsilon}^{(\ell)} \\ &  + P_{c,0}^{-1} P_{c,\epsilon} N_{c,\epsilon} (A_{c,\epsilon}^{(\ell)} + B_{c,\epsilon}^{(\ell - 1)}).
\end{split}
\end{equation*}
One can then use this relation to prove the result by induction. Indeed, the operator in front of $C_{c,\epsilon}^{(\ell)}$ grows at most like a power of $\epsilon^{-1}$ (as an operator from $H^{k+2}(M_{a,b})$ to $H^{k}(M_{a,b})$) and so do $A_{c,\epsilon}^{(\ell)}$ and $B_{c,\epsilon}^{(\ell - 1)}$ (as operators from $H^{k + 2 \ell}(M_{a,b})$ to $H^{k + 2}(M_{a,b})$), while $P_{c,\epsilon} N_{c,\epsilon}$ is a $\mathcal{O}(\epsilon^\infty)$ from $H^{k+2}(M_{a,b})$ to $C^\infty(M_{a,b})$ (because the range of $N_{c,\epsilon}$ is contained in the span of the boundary layers from Proposition \ref{proposition:boundary_layers_bis}).

It remains to prove \ref{item:quality_approximation}. Let us start by noticing that for $\ell \geq 0$, we have
\begin{equation*}
B_{c,\epsilon}^{(\ell + 1)} - B_{c,\epsilon}^{(\ell)} = \epsilon \left( ( I - P_{c,0}^{-1} P_{c,0}) E_{c,\epsilon} - i\alpha^{-1}\epsilon^2 P_{c,0}^{-1} (\partial_x^2 - \alpha^2)^2\right) (B_{c,\epsilon}^{(\ell)} - B_{c,\epsilon}^{(\ell - 1)}).
\end{equation*}
Hence, we find by induction that 
\begin{equation*}
\n{B_{c,\epsilon}^{(\ell +1)} - B_{c,\epsilon}^{(\ell)}}_{H^{k + 2 \ell}(M_{a,b}) \to H^k(M_{a,b})} \underset{\epsilon \to 0}{=} \mathcal{O}(\epsilon^{\ell +1})
\end{equation*}
for $k > \frac{1}{2}$ and every integer $\ell \geq -1$. If $f \in H^{k + 2 \ell - 1}(M_{a,b})$, then
\begin{equation*}
A_{c,\epsilon}^{(\ell)} f = (A_{c,\epsilon}^{(\ell)} - B_{c,\epsilon}^{(\ell)} + N_{c,\epsilon} B_{c,\epsilon}^{(\ell - 1)}) f + (B_{c,\epsilon}^{(\ell)} - B_{c,\epsilon}^{(\ell - 1)}) f + ( I - N_{c,\epsilon}) B_{c,\epsilon}^{(\ell - 1)} f.
\end{equation*}
The first two terms are respectively $\mathcal{O}(\epsilon^\infty \n{f}_{H^{k + 2 \ell - 1}(M_{a,b})})$ and $\mathcal{O}(\epsilon^{\ell} \n{f}_{H^{k + 2 \ell - 1}(M_{a,b})})$ in $H^{k+1}(M_{a,b})$, and the derivative of the last one vanishes at $a$ and at $b$ (due to the definition of the operator $N_{c,\epsilon}$). Hence, we have
\begin{equation}\label{eq:control_neumann}
|(A_{c,\epsilon}^{(\ell)} f)'(a)| + |(A_{c,\epsilon}^{(\ell)} f)'(b)| \leq C_{k,\ell} \epsilon^\ell \n{f}_{H^{k + 2 \ell - 1}(M_{a,b})}.
\end{equation}
Now, let us write
\begin{equation}\label{eq:inductive_estimate}
\begin{split}
    P_{c,\epsilon} A_{c,\epsilon}^{(\ell+1)} - I 
    = & (P_{c,\epsilon} - P_{c,0})P_{c,0}^{-1} \left( I - P_{c,\epsilon}A_{c,\epsilon}^{(\ell)}\right) + (P_{c,\epsilon} - P_{c,0})P_{c,0}^{-1} P_{c,\epsilon} N_{c,\epsilon} A_{c,\epsilon}^{(\ell)} \\
    & + P_{c,\epsilon} ( I - P_{c,0}^{-1} P_{c,0}) N_{c,\epsilon} A_{c,\epsilon}^{(\ell)} \\
    = & i\alpha^{-1} \epsilon^2 (\partial_x^2 - \alpha^2)^2 P_{c,0}^{-1} \left( I - P_{c,\epsilon}A_{c,\epsilon}^{(\ell)}\right) + i\alpha^{-1} \epsilon^2 (\partial_x^2 - \alpha^2)^2 P_{c,0}^{-1} P_{c,\epsilon} N_{c,\epsilon} A_{c,\epsilon}^{(\ell)} \\
    & + \epsilon P_{c,\epsilon} ( I - P_{c,0}^{-1} P_{c,0}) E_{c,\epsilon} A_{c,\epsilon}^{(\ell)}.
\end{split}
\end{equation}
Here, we used the fact that $(I - P_{c,0}^{-1} P_{c,0}) A_{c,\epsilon}^{(\ell)} = 0$, since $A_{c,\epsilon}^{(\ell)}$ takes value in the range of $P_{c,0}^{-1}$. Hence, we find that (the constant $C$ may depend on $k$ and $\ell$)
\begin{equation*}
\n{P_{c,\epsilon} A_{c,\epsilon}^{(\ell+1)} - I}_{H^{k + 2 \ell}(M_{a,b}) \to H^{k-4}(M_{a,b})} \leq C \epsilon^2 \n{P_{c,\epsilon} A_{c,\epsilon}^{(\ell)} - I}_{H^{k + 2 \ell}(M_{a,b}) \to H^{k-2}(M_{a,b})} + C \epsilon^{\ell + 1}.
\end{equation*}
Here, we used the fact that $u_{c,\epsilon}^{\#}$ and $v_{c,\epsilon}^{\#}$ are almost in the kernel of $P_{c,\epsilon}$ to bound the second term in the last line of \eqref{eq:inductive_estimate} and we used \eqref{eq:control_neumann} to bound the last term. By induction, we find that for every $k > 1/2$ and $\ell \geq 0$, we have
\begin{equation}\label{eq:quality_approximation}
\n{P_{c,\epsilon} A_{c,\epsilon}^{(\ell)} - I}_{H^{k + 2 \ell - 2}(M_{a,b}) \to H^{k-4}(M_{a,b})} \underset{\epsilon \to 0}{=} \mathcal{O}(\epsilon^{\ell}).
\end{equation} 

Now, choose two $C^\infty$ functions $h_a$ and $h_b$ on $M_{a,b}$, supported close to $a$ and $b$ respectively, and such that $h_a(a) = h_b(b) = 0$ and $h_a'(a) = h_b'(b) = 1$. Define the operator $\widetilde{A}_{c,\epsilon}^{(\ell)}$ by
\begin{equation*}
\widetilde{A}_{c,\epsilon}^{(\ell)} f = A_{c,\epsilon}^{(\ell)} f - (A_{c,\epsilon}^{(\ell)} f)'(a) h_a - (A_{c,\epsilon}^{(\ell)} f)'(b) h_b.
\end{equation*}
It follows from \eqref{eq:control_neumann} that
\begin{equation}\label{eq:approximate_approximation}
\n{A_{c,\epsilon}^{(\ell)} - \widetilde{A}_{c,\epsilon}^{(\ell)}}_{H^{k + 2 \ell - 1}(M_{a,b}) \to H^m(M_{a,b})} \underset{\epsilon \to 0}{=} \mathcal{O}(\epsilon^\ell)
\end{equation}
for $k > 1/2$ and $m \in \mathbb{R}$. Recalling \eqref{eq:quality_approximation}, we find that
\begin{equation*}
\n{P_{c,\epsilon} \widetilde{A}_{c,\epsilon}^{(\ell)} - I}_{H^{k + 2 \ell - 1}(M_{a,b}) \to H^{k - 3}(M_{a,b})} \underset{\epsilon \to 0}{=} \mathcal{O}(\epsilon^{\ell}),
\end{equation*}
for $k > 1/2$. The advantage of $\widetilde{A}_{c,\epsilon}^{(\ell)}$ compared to $A_{c,\epsilon}^{(\ell)}$ is that it maps $H^{k + 2 \ell - 1}(M_{a,b})$ inside $H_{\DN}^{k + 1}(M_{a,b})$. Hence, recalling Lemma \ref{lemma:invertibility_a_priori} we find that, for $k \in (1/2, 3/2)$ and $\ell \geq 0$, 
\begin{equation*}
\n{\widetilde{A}_{c,\epsilon}^{(\ell)} - P_{c,\epsilon}^{-1}}_{H^{k + 2 \ell}(M_{a,b}) \to H^k(M_{a,b})} \underset{\epsilon \to 0}{=} \mathcal{O}(\epsilon^\ell),
\end{equation*}
and the result follows from \eqref{eq:approximate_approximation}.
\end{proof}

\subsection{Perturbation theory for resonances}\label{subsection:perturbation_resonances}

We have now at our disposal all the technical pieces required to prove Theorem \ref{theorem:limit}. Let $c_1 \in (c_0 - \delta,c_0 + \delta) +  i(- \delta, + \infty)$.  We do not assume anymore that $c_1 \notin \mathcal{R}$. The goal of this section is to describe $\Sigma_\epsilon$ near $c_1$ as $\epsilon$ goes to $0$. According to Lemma \ref{lemma:discrete_resonances}, there is $\nu > 0$ such that $\overline{\mathbb{D}}(c_1,\nu) \subseteq (c_0 - \delta, c_0  + \delta) + i (- \delta, + \infty)$ and $\overline{\mathbb{D}}(c_1,\nu) \setminus \set{c_1} \cap \mathcal{R} = \emptyset$. Thanks to Lemma \ref{lemma:invertibility_a_priori}, we may define for $\epsilon \geq 0$ small the finite rank operator
\begin{equation*}
    \Pi_{c_1,\epsilon} = - \frac{1}{2 i \pi} \int_{\partial \mathbb{D}(c_1,\nu)} (\partial_x^2 - \alpha^2) P_{c,\epsilon}^{-1} \mathrm{d} c.
\end{equation*}
Notice that this operator maps $H^{k}(M_{a,b})$ to itself for every $k > -1/2$, and $H^k(M_{a,b})$ to $H^{k+2}(M_{a,b})$ for every $k > -3/2$ when $\epsilon > 0$. Moreover, since it has finite rank, its range on $H^k(M_{a,b})$ does not depend on $k > -1/2$ (and it is also its range when acting on $C^\infty(\overline{M}_{a,b})$). Using the relation $P_{c,\epsilon} = Q_{c,\epsilon} (\partial_x^2 - \alpha^2)$, we find that
\begin{equation*}
        \Pi_{c_1,\epsilon} = \frac{1}{2 i \pi} \int_{\partial \mathbb{D}(c_1,\nu)} (c - Q_{0,\epsilon})^{-1} \mathrm{d} c.
\end{equation*}
It follows from \cite[III.6.5]{kato_perturbation_theory} that $\Pi_{c_1,\epsilon}$ is the spectral projector for $Q_{0,\epsilon}$ (with the domain defined in Remark \ref{remark:sigmaepsilon_as_spectrum}) associated to the spectrum within $\mathbb{D}(c_1,\nu)$, see also Remark~\ref{remark:nonspectral_spectral_theory}.

We start by studying the continuity of $\Pi_{c_1,\epsilon}$ at $\epsilon = 0$. We do not expect $\Pi_{c_1,\epsilon}$ to be smooth at $\epsilon = 0$ due to the presence of boundary layers in the asymptotic expansion for $P_{c,\epsilon}^{-1}$ as $\epsilon$ goes to $0$.

\begin{lemma}\label{lemma:regularity_projector}
There is $\epsilon_1 > 0$ such that for every $\epsilon \in (0,\epsilon_1)$ and $k,k' \in (-3/2, + \infty)$ the operator $\Pi_{c_1,\epsilon}$ is bounded from $H^{k}(M_{a,b})$ to $H^{k'}(M_{a,b})$. If $k,k' \in (-3/2,-1/2)$ then $\epsilon \mapsto \Pi_{c_1,\epsilon}$ is continuous at $\epsilon = 0$ where $\Pi_{c_1,\epsilon}$ is seen as an operator from $H^{k}(M_{a,b})$ to $H^{k'}(M_{a,b})$.
\end{lemma}

\begin{proof}
Recall from the proof of Lemma \ref{lemma:invertibility_a_priori} the operator
\begin{equation*}
R_{c,\epsilon} \coloneqq ( I - U'' (P_{c,\alpha,\epsilon} + U'')^{-1})^{-1}
\end{equation*}
which is well-defined for $c \notin \Sigma_\epsilon$ (if $\epsilon > 0$) or $c \notin \mathcal{R}$ (if $\epsilon = 0$). For every $\ell \geq 0$, we have
\begin{equation*}
R_{c,\epsilon} = \sum_{p = 0}^{\ell-1} (U'' (P_{c,\epsilon} + U'')^{-1})^{p} + (U'' (P_{c,\epsilon} + U'')^{-1})^{\ell} R_{c,\epsilon}.
\end{equation*}
It follows from Proposition \ref{proposition:invertible_compact_perturbation} that the sum in the right hand side is holomorphic in $c$ on a neighbourhood of $\overline{\mathbb{D}}(c_1,\nu)$, and thus we have
\begin{equation}\label{eq:spectral_projector_gain_derivative}
\Pi_{c_1,\epsilon} = -\frac{1}{2 i \pi} \int_{\partial \mathbb{D}(c_1,\nu)} (\partial_x^2 - \alpha^2) (P_{c,\epsilon} + U'')^{-1} (U'' (P_{c,\epsilon} + U'')^{-1})^{\ell} R_{c,\epsilon} \mathrm{d}c.
\end{equation}
Since the operator $(P_{c,\epsilon} + U'')^{-1}$ maps $H^k(M_{a,b})$ to $H^{k+4}(M_{a,b})$ for every $k > - 3/2$, we find that $\Pi_{c_1,\epsilon}$ maps $H^{k}(M_{a,b})$ to $H^{k + 2 + 4 \ell}(M_{a,b})$ for every $k > - 3/2$ and $\ell \in \mathbb{N}$. This proves the first part of the lemma (beware that we do not claim any uniformity in $\epsilon$ in this first part).

Let us move to the second part of the lemma. Working as in the proof of Lemma~\ref{lemma:invertibility_a_priori}, we find that for every $k \in (- 3/2, - 1/2)$ and $s < k +2$, the map $\epsilon \mapsto (P_{c,\epsilon} + U'')^{-1}$ is continuous at $\epsilon = 0$, when $(P_{c,\epsilon} + U'')^{-1}$ is seen as an operator from $H^k(M_{a,b})$ to $H^{s}(M_{a,b})$. It implies in particular that $\epsilon \mapsto R_{c,\epsilon}$ is continuous at $\epsilon=  0$ as an operator on $H^{k}(M_{a,b})$. Using these continuity properties and \eqref{eq:spectral_projector_gain_derivative} with $\ell = 1$, we get the second part of the lemma.
\end{proof}

A consequence of Lemma \ref{lemma:regularity_projector} is that $\Pi_{c_1,\epsilon}$ is continuous at $\epsilon = 0$ as an endomorphism of $H^{-1}(M_{a,b})$, which implies that the ranks of $\Pi_{c_1,\epsilon}$ is constant near $\epsilon = 0$, see \cite[Problem 3.21 p.156]{kato_perturbation_theory}. Let $N$ denote the rank of $\Pi_{c_1,\epsilon}$ for $\epsilon$ near $0$, and let $f_1,\dots,f_N$ be a basis for the range of $\Pi_{c_1,0}$. Notice that Lemma \ref{lemma:regularity_projector} implies that $f_1,\dots,f_N$ belong to $C^\infty(\overline{M}_{a,b})$. Let us now choose $\chi_1,\dots, \chi_N$ in $C^\infty_c(M_{a,b})$ such that the matrix $(\langle f_i, \chi_j \rangle_{L^2(M_{a,b})})_{1 \leq i,j \leq N}$ is invertible. Let us focus on the fact that we choose the $\chi_j$'s supported away from $a$ and $b$. This is what we will use to get rid of the lack of smoothness of $\Pi_{c_1,\epsilon}$ at $\epsilon = 0$. Indeed, this lack of smoothness is produced by the boundary layers, that are $\mathcal{O}(\epsilon^\infty)$ away from the boundary points of $\overline{M}_{a,b}$. 

For $\epsilon > 0$ small, it follows from Lemma \ref{lemma:regularity_projector} that $\Pi_{c_1,\epsilon} f_1, \dots, \Pi_{c_1,\epsilon} f_N$ is a basis for the range of $\Pi_{c_1,\epsilon}$. Moreover, the range of $\Pi_{c_1,\epsilon}$ is a sum of generalized eigenspaces for $Q_{0,\epsilon}$, and in particular it is stable under the action of this operator. Let $\mathcal{Q}(\epsilon)$ denote the matrix of $Q_{0,\epsilon}$ acting on the range of $\Pi_{c_1,\epsilon}$ in the basis $\Pi_{c_1,\epsilon} f_1, \dots, \Pi_{c_1,\epsilon} f_N$. Namely, writing $\mathcal{Q}(\epsilon) = (q_{i,j}(\epsilon))_{1 \leq i,j \leq N}$, we have for $ i = 1,\dots,N$
\begin{equation*}
    Q_{0,\epsilon} \Pi_{c_1,\epsilon} f_i = \sum_{j = 1}^N q_{i,j}(\epsilon) \Pi_{c_1,\epsilon} f_j.
\end{equation*}
Introduce also the matrices 
\begin{equation*}
    M(\epsilon) = (\langle \Pi_{c_1,\epsilon} f_i, \chi_j \rangle)_{1 \leq i ,j \leq N} \textup{ and } L(\epsilon) = (\langle Q_{0,\epsilon} \Pi_{c_1,\epsilon} f_i, \chi_j \rangle)_{1 \leq i,j \leq N}.
\end{equation*}
It follows from Lemma \ref{lemma:regularity_projector} that $M(\epsilon)$ is invertible for $\epsilon$ small enough, and we have $\mathcal{Q}(\epsilon) = L(\epsilon) M(\epsilon)^{-1}$. Using this formula, we show that:

\begin{lemma}\label{lemma:existence_expansion}
The matrix $\mathcal{Q}(\epsilon)$ has a Taylor expansion at $\epsilon = 0$.
\end{lemma}

\begin{proof}
We only need to prove that $M(\epsilon)$ and $L(\epsilon)$ have Taylor expansions at $\epsilon = 0$. We will prove for instance that $M(\epsilon)$ has a Taylor expansion at $\epsilon = 0$. The case of $L(\epsilon)$ is similar. 

Let us fix $i$ and $j$ in $\set{1,\dots,N}$ and let $\tilde{\chi} \in C_c^\infty(M_{a,b})$ be such that $\chi \equiv 1$ on the support of $\chi_j$. Notice that 
\begin{equation*}
\langle \Pi_{c_1,\epsilon} f_i, \chi_j \rangle = \langle \chi_j \Pi_{c_1,\epsilon} f_i, \tilde{\chi} \rangle. 
\end{equation*}
Let us write, using the operators\footnote{These operators are a priori only defined locally in $c$, but we can cover $\partial \mathbb{D}(c_1,\epsilon)$ by a finite number of disks on which these operators are defined, and eventually use a partition of unity (in $c$) to glue the operators defined on different disks in order to work with operators that depend continuously on $c$.} from Lemma \ref{lemma:description_approximation}, that for $\ell \in \mathbb{N}$,
\begin{equation}\label{eq:expansion_projector}
\begin{split}
- \chi_j \Pi_{c_1,\epsilon} f_i 
= & \frac{1}{2 i \pi} \int_{\partial \mathbb{D}(c_1,\nu)} \chi_j (\partial_x^2 - \alpha^2) (P_{c,\epsilon}^{-1} - A_{c,\epsilon}^{(\ell)}) f_i \mathrm{d} c \\ 
& + \frac{1}{2 i \pi} \int_{\partial \mathbb{D}(c_1,\nu)} \chi_j (\partial_x^2 - \alpha^2) (A_{c,\epsilon}^{(\ell)} - B_{c,\epsilon}^{(\ell)} + N_{c,\epsilon} B_{c,\epsilon}^{(\ell - 1)}) f_i \mathrm{d} c \\ 
&  + \frac{1}{2 i \pi} \int_{\partial \mathbb{D}(c_1,\nu)} \chi_j (\partial_x^2 - \alpha^2) B_{c,\epsilon}^{(\ell)} f_i \mathrm{d} c \\ 
& + \frac{1}{2 i \pi} \int_{\partial \mathbb{D}(c_1,\nu)} \chi_j (\partial_x^2 - \alpha^2) N_{c,\epsilon} B_{c,\epsilon}^{(\ell - 1)} f_i \mathrm{d} c.
\end{split}
\end{equation}
Recalling that $f_i \in C^\infty(\overline{M}_{a,b})$, it follows from Lemma \ref{lemma:description_approximation} that the first two terms in the right hand side are $\mathcal{O}(\epsilon^\ell)$ in $H^{-1}(M_{a,b})$, and that the third term admits a Taylor expansion at $\epsilon = 0$. To deal with the last term, notice that the range of $N_{c,\epsilon}$ is made of $C^\infty$ functions that decay exponentially fast as $\epsilon$ goes to $0$ away from $a$ and $b$. Since $\chi_j$ is compactly supported in $M_{a,b}$, we find that this last term is actually a $\mathcal{O}(\epsilon^\infty)$ in $C^\infty(\overline{M}_{a,b})$. Since $\tilde{\chi}$ belongs to $C_c^\infty(M_{a,b})$, we find that $\langle \Pi_{c_1,\epsilon} f_i, \chi_j \rangle $ has a Taylor expansion at $\epsilon = 0$ up to order $\ell-1$, which ends the proof of the lemma since $\ell$ is arbitrary.

When dealing with $L(\epsilon)$, the only difference is the presence of the operator $Q_{0,\epsilon}$. In order to deal with it, one just needs to notice that 
\begin{equation*}
    - \chi_j Q_{0,\epsilon} \Pi_{c_1,\epsilon} f_i = \frac{1}{2 i \pi} \int_{\partial \mathbb{D}(c_1,\nu)} \chi_j Q_{0,\epsilon} (\partial_x^2 - \alpha^2) P_{c,\epsilon}^{-1} f_i \mathrm{d}c = \frac{1}{2 i \pi} \int_{\partial \mathbb{D}(c_1,\nu)} \chi_j P_{0,\epsilon} P_{c,\epsilon}^{-1} f_i \mathrm{d}c.
\end{equation*}
Here, we used the fact that functions in the range of $P_{c,\epsilon}^{-1}$ satisfy Dirichlet boundary condition. One can then perform the same decomposition as in \eqref{eq:expansion_projector} and studies each term as above. Indeed, since $P_{0,\epsilon}$ is a differential operator (with coefficients that are polynomials in $\epsilon$), it preserves the localization property of the boundary layers that appear in the definition of $N_{c,\epsilon}$.
\end{proof}

There is little left to conclude the proof of Theorem \ref{theorem:limit}.

\begin{proof}[Proof of Theorem \ref{theorem:limit}]
The operator $\Pi_{c_1,\epsilon}$ is the spectral projector of $Q_{0,\epsilon}$ associated to the spectrum within $\mathbb{D}(c_1,\nu)$. With the point of view from \S \ref{subsection:basic_properties}, one can also find that $(\partial_x^2 - \alpha^2)^{-1} \Pi_{c_1,\epsilon}(\partial_x^2 - \alpha^2)$ is a projector with range $\oplus_{c \in \mathbb{D}(c_1,\nu)} \mathfrak{S}_{c,\epsilon}(\overline{M}_{a,b})$, where there are only finitely many non-trivial terms in the sum. Hence, the spectrum of the matrix $\mathcal{Q}(\epsilon)$ is just the spectrum of $Q_{0,\epsilon}$ within $\mathbb{D}(c_1,\nu)$. Here, we mean the spectrum of $Q_{0,\epsilon}$ on $M_{a,b}$ (with the ``boundary'' condition $\partial_x (\partial_x^2 - \alpha^2)^{-1} u(a) = \partial_x (\partial_x^2 - \alpha^2)^{-1}u(b) = 0$ in the case $\epsilon > 0$). It follows then from Proposition \ref{proposition:elliptic_eigenvalues} (see also Remark \ref{remark:nonspectral_spectral_theory}) that, for $\epsilon > 0$ small, the spectrum of $\mathcal{Q}(\epsilon)$ is just $\Sigma_\epsilon \cap \mathbb{D}(c_1,\nu)$, multiplicities taken into account. At $\epsilon = 0$, the spectrum of $\mathcal{Q}(0)$ is just $c_1$, with multiplicity the multiplicity of $c_1$ as an element of $\mathcal{R}$, see Remark \ref{remark:multiplicity_resonances}.

It follows then from Lemma \ref{lemma:existence_expansion} that the polynomial
\begin{equation*}
\mathcal{P}(\epsilon) = \det( X - \mathcal{Q}(\epsilon)) = \prod_{c \in \Sigma_\epsilon \cap \mathbb{D}(c_1,\nu)} (X - c)
\end{equation*}
has a Taylor expansion at $\epsilon = 0$. In this formula, the elements of $\Sigma_\epsilon$ are obviously repeated according to multiplicity. Since $\mathcal{P}(0) = (X - c_1)^m$ where $m$ is the multiplicity of $c_1$ as an element of $\mathcal{R}$ (in particular $\mathcal{P}(\epsilon) = 1$ for $\epsilon$ small if $m =0$), Theorem \ref{theorem:limit} follows. 
\end{proof}

\subsection{First order perturbation of a simple resonance}

The goal of this section is to prove the following formula for the first order perturbation of a simple resonance.

\begin{proposition}\label{proposition:first_order_perturbation}
Assume that $c_1 \in \mathcal{R}$ is simple and for $\epsilon > 0$ small, let $c(\epsilon)$ denote the element of $\Sigma_\epsilon$ close to $c_1$ given by Theorem \ref{theorem:limit}. Let $\dot{c}(0)$ denote the first order perturbation of $c_1$, i.e. we have $c(\epsilon) \underset{\epsilon \to 0}{=} c_1+ \dot{c}(0) \epsilon + \mathcal{O}(\epsilon^2)$. Then,
\begin{equation*}
    \dot{c}(0) = \frac{\lambda_{c_1,0} (\partial_x \psi_0(a))^2 - \mu_{c_1,0} (\partial_x \psi_0(b))^2}{\int_{M_{a,b}} \psi_0 \frac{(\partial_x^2 - \alpha^2) \psi_0}{U - c_1}\mathrm{d}x}
\end{equation*}
where 
\begin{itemize}
    \item $\psi_0$ is any non-zero resonant state associated to the resonance $c_1$ (it implies that the integral in the denominator is non-zero);
    \item $\lambda_{c_1,0}$ and $\mu_{c_1,0}$ are the limits of $\lambda_{c_1,\epsilon}$ and $\mu_{c_1,\epsilon}$ as $\epsilon$ goes to $0$, that is $\lambda_{c_1,0} = \frac{1}{\sqrt{i \alpha(U(a) - c_1)}}$ and $\mu_{c_1,0} = \frac{1}{\sqrt{i \alpha (U(b) - c_1)}}$, where the square roots have respectively a negative and a positive real part.
\end{itemize}
\end{proposition}

From now on and until the end of the subsection, we assume that $c_1 \in \mathcal{R}$ is simple. In order to prove Proposition \ref{proposition:first_order_perturbation}, let us choose a (non-zero) resonant state $\psi_0$ associated to $c_1$, and define for $\epsilon  > 0$ small
\begin{equation}\label{eq:definition_psi_epsilon}
    \psi_\epsilon = (\partial_x^2 - \alpha^2)^{-1}\Pi_{c_1,\epsilon} (\partial_x^2 - \alpha^2) \psi_0 = - \frac{1}{2i \pi} \int_{\partial \mathbb{D}(c_1,\nu)} P_{c,\epsilon}^{-1}(\partial_x^2 - \alpha^2)\psi_0 \mathrm{d}c.
\end{equation}
Here, we are using the spectral projector for $(\partial_x^2 - \alpha^2) P_{0,\epsilon}$ instead of $\Pi_{c_1,\epsilon}$ because we want $\psi_\epsilon$ to converge to $\psi_0$ as $\epsilon$ goes to $0$. We start by giving a first order approximation of $\psi_\epsilon$.

\begin{lemma}\label{lemma:perturbation_eigenvector}
There is $\dot{\psi}_0 \in C^\infty(\overline{M}_{a,b})$ such that, for every compact subset $L$ of $M_{a,b}$, the function $\psi_\epsilon - \psi_0 - \epsilon \dot{\psi}_0$ and all its derivatives are $\mathcal{O}(\epsilon^2)$ on $L$ as $\epsilon$ goes to $0$. Moreover, we have
\begin{equation}\label{eq:perturbation_eigenvector}
    P_{c_1,0} \dot{\psi}_0 = \Pi_{c_1,0} P_{c_1,0} E_{c_1,0} \psi_0.
\end{equation}
\end{lemma}

\begin{remark}
In \eqref{eq:perturbation_eigenvector}, the operator $E_{c_1,0}$ is obtained by replacing $\lambda_{c_1,\epsilon}$ and $\mu_{c_1,\epsilon}$ in the definition of $E_{c_1,\epsilon}$ by their limit as $\epsilon$ goes to $0$. Notice that the image of a smooth function $f$ by the operator $\Pi_{c_1,0} P_{c_1,0}$ only depends on the values of $f$ at $a$ and $b$, since if $f(a) = f(b) = 0$, then $f$ satisfies Dirichlet boundary condition and thus $\Pi_{c_1,0} P_{c_1,0} f = \Pi_{c_1,0} Q_{c_1,0} (\partial_x^2 - \alpha^2) f = 0$.
\end{remark}

\begin{proof}[Proof of Lemma \ref{lemma:perturbation_eigenvector}]
Let $L$ be a compact subset of $M_{a,b}$. In order to get an approximation for $\psi_\epsilon$ on $L$, we work as in the proof of Lemma \ref{lemma:existence_expansion}: we use the approximation for $P_{c,\epsilon}^{-1}$ given in Lemma \ref{lemma:description_approximation} in the formula \eqref{eq:definition_psi_epsilon}. Then, we notice that all the terms that include boundary layers (and their derivatives) are $\mathcal{O}(\epsilon^\infty)$ on $L$, and can consequently be neglected.

Hence, we find that there is a sequence $(f_j)_{j \geq 0}$ of elements of $C^\infty(\overline{M}_{a,b})$ such that for every $\ell \geq 0$ the function $\psi_\epsilon - \sum_{j = 0}^\ell \epsilon^j f_j$ is a $\mathcal{O}(\epsilon^{\ell + 1})$ in $H^1$ on a neighbourhood of $L$. Here, the approximation a priori only holds in $H^1$ due to the restriction in Lemma \ref{lemma:description_approximation}. However, since $\psi_\epsilon - \sum_{j = 0}^\ell \epsilon^j f_j$ is uniformly smooth as $\epsilon$ goes to $0$, we can get an approximation in a smaller Sobolev space if we accept to reduce a little the quality of the approximation. But, since we have an approximation at any polynomial order, we get this way that $\psi_\epsilon - f_0 - \epsilon f_1$ and all its derivatives are $\mathcal{O}(\epsilon^2)$ on $L$ as $\epsilon$ goes to $0$.

Let us compute $f_0$. Recalling Lemma \ref{lemma:description_approximation} and the formula for the $B_{c,\epsilon}^{(\ell)}$'s given in \eqref{eq:formula_for_B}, we find that
\begin{equation*}
    f_0 = (\partial_x^2 - \alpha^2)^{-1} \Pi_{c_1,0} (\partial_x^2 - \alpha^2) \psi_0 = \psi_0.
\end{equation*}
Here, we used the fact that $(\partial_x^2 - \alpha^2)^{-1} \Pi_{c_1,0} (\partial_x^2 - \alpha^2)$ is a spectral projector for $(\partial_x^2 - \alpha^2)^{-1} P_{0,0}$ and that $\psi_0$ is an eigenvector for this operator. Let us set $\dot{\psi}_0 = f_1$, we find as for $f_0$ that
\begin{equation*}
    \dot{\psi}_0 = - \frac{1}{2 i \pi} \int_{\partial \mathbb{D}(c_1,\nu)} (I - P_{c,0}^{-1} P_{c,0}) E_{c,0} P_{c,0}^{-1}(\partial_x^2 - \alpha^2) \psi_0 \mathrm{d}c,
\end{equation*}
where $E_{c,0}$ is obtained by replacing $\lambda_{c,\epsilon}$ and $\mu_{c,\epsilon}$ in the definition of $E_{c,\epsilon}$ by their limits as $\epsilon$ goes to $0$.

In order to compute $\dot{\psi}_0$, let us recall that we assume that $c_1$ is a simple resonance, which implies that $c \mapsto P_{c,0}^{-1}$ has a simple pole at $c = c_1$. Hence, we can write
\begin{equation}\label{eq:laurent_series}
    P_{c,0}^{-1} \underset{c \to c_1}{=} - \frac{(\partial_x^2 - \alpha^2)^{-1}\Pi_{c_1,0}}{c - c_1} + H + \mathcal{O}(|c- c_1|),
\end{equation}
where $H$ maps $C^\infty(\overline{M}_{a,b})$ into itself. With this notation, it follows that
\begin{equation*}
\begin{split}
    \dot{\psi}_0 & = E_{c_1,0} (\partial_x^2 - \alpha^2)^{-1}\Pi_{c_1,0}(\partial_x^2 - \alpha^2) \psi_0 \\ & \qquad - (\partial_x^2 - \alpha^2)^{-1} \Pi_{c_1,0} P_{c_1,0} E_{c_1,0}H (\partial_x^2 - \alpha^2) \psi_0 \\ & \qquad - H P_{c_1,0} E_{c_1,0} (\partial_x^2 - \alpha^2)^{-1} \Pi_{c_1,0} (\partial_x^2 - \alpha^2) \psi_0 \\ & \qquad + (\partial_x^2 - \alpha^2)^{-1} \Pi_{c_1,0} \frac{\partial}{\partial c} \left( P_{c,0} E_{c,0} \right)_{|c = c_1} (\partial_x^2 - \alpha^2)^{-1} \Pi_{c_1,0} (\partial_x^2 - \alpha^2) \psi_0.
\end{split}
\end{equation*}
It follows then that 
\begin{equation*}
    P_{c_1,0} \dot{\psi}_0 = ( I - P_{c_1,0}H)P_{c_1,0}E_{c_1,0} \psi_0.
\end{equation*}
Plugging \eqref{eq:laurent_series} in $P_{c,0} P_{c,0}^{-1} = I$, we find that $I - P_{c_1,0} H = \Pi_{c_1,0}$, so that
\begin{equation*}
    P_{c_1,0} \dot{\psi}_0 = \Pi_{c_1,0} P_{c_1,0} E_{c_1,0} \psi_0.
\end{equation*}
This completes the proof.
\end{proof}

In order to apply Lemma \ref{lemma:perturbation_eigenvector}, let us identify $\Pi_{c_1,0}$.

\begin{lemma}\label{lemma:formula_projector}
The integral
\begin{equation}\label{eq:non_zero_integral}
    \int_{M_{a,b}} (\partial_x^2 - \alpha^2)\psi_0 \frac{\psi_0}{U-c_1}\mathrm{d}x
\end{equation}
is non-zero and for every $f \in C^\infty(\overline{M}_{a,b})$ we have
\begin{equation*}
    \Pi_{c_1,0} f = \frac{\int_{M_{a,b}} f \frac{\psi_0}{U - c_1}\mathrm{d}x}{\int_{M_{a,b}} \psi_0 \frac{(\partial_x^2 - \alpha^2) \psi_0}{U - c_1}\mathrm{d}x} (\partial_x^2 - \alpha^2) \psi_0.
\end{equation*}
\end{lemma}

\begin{proof}
The proof that the integral is non-zero is inspired from  \cite[Lemma 1]{stepin_rayleigh}. Let us consider the function $\phi \in C^\infty(\overline{M}_{a,b})$ on $\overline{M}_{a,b}$ that solves the Cauchy problem
\begin{equation*}
    \begin{cases} P_{c_1,0} \phi = (\partial_x^2 - \alpha^2) \psi_0, \\ \phi(a) = 0, \ \partial_x \phi(a) = 1. \end{cases}
\end{equation*}
Since $c_1$ is a simple eigenvalue, we must have $\phi(b) \neq 0$. Moreover, integrating by parts, we find that
\begin{equation*}
\begin{split}
     \int_{M_{a,b}} (\partial_x^2 - \alpha^2) \psi_0 \frac{\psi_0}{U-c_1}\mathrm{d}x & = \int_{M_{a,b}} P_{c_1,0} \phi \frac{\psi_0}{U-c_1}\mathrm{d}x \\ & = - \phi(b) \partial_x \psi_0(b) + \int_{M_{a,b}} \phi \frac{P_{c_1,0} \psi_0}{U - c_1} \mathrm{d}x = - \phi(b) \partial_x \psi_0(b).
\end{split}
\end{equation*}
Since $\psi_0(b) = 0$ and $\psi_0$ is non-zero, we must have $\partial_x \psi_0(b) \neq 0$, which proves that the integral \eqref{eq:non_zero_integral} is non-zero.

Since $\Pi_{c_1,0}$ is a projector on the one-dimensional eigenspace of $Q_{0,0}$ associated to the eigenvalue $c_1$, we know that there is a continuous linear functional $\ell$ on $C^\infty(\overline{M}_{a,b})$ such that 
\begin{equation*}
    \Pi_{c_1,0} f = \ell(f) (\partial_x^2 - \alpha^2) \psi_0
\end{equation*}
for every $f \in C^\infty(\overline{M}_{a,b})$. Now, if $f, g\in C^\infty(\overline{M}_{a,b})$ and $c \in \mathbb{D}(c_1,\nu) \setminus \set{c_1}$, then $P_{c,0}^{-1}f$ and $P_{c,0}^{-1}g$ satisfy Dirichlet boundary condition, and thus an integration by part as above give
\begin{equation*}
    \int_{M_{a,b}} P_{c,0}^{-1} f \frac{g}{U - c}\mathrm{d}x = \int_{M_{a,b}} P_{c,0}^{-1} f \frac{P_{c,0} P_{c,0}^{-1}g}{U - c}\mathrm{d}x = \int_{M_{a,b}}f \frac{P_{c,0}^{-1} g }{U - c}\mathrm{d}x.
\end{equation*}
Identifying the residue at $c = c_1$ of both sides of this equality, we get
\begin{equation*}
    \int_{M_{a,b}} (\partial_x^2 - \alpha^2)^{-1} \Pi_{c_1,0} f \frac{g}{U - c_1}\mathrm{d}x = \int_{M_{a,b}} f \frac{(\partial_x^2 - \alpha^2)^{-1} \Pi_{c_1,0} g}{U - c_1}\mathrm{d}x.
\end{equation*}
Taking $g = (\partial_x^2 - \alpha^2) \psi_0$ in this equality, we get
\begin{equation*}
    \ell(f) \int_{M_{a,b}} \psi_0 \frac{(\partial_x^2 - \alpha^2) \psi_0}{U - c_1}\mathrm{d}x = \int_{M_{a,b}} f \frac{\psi_0}{U - c_1}\mathrm{d}x,
\end{equation*}
and the result follows.
\end{proof}

We are now ready to prove Proposition \ref{proposition:first_order_perturbation}.

\begin{proof}[Proof of Proposition \ref{proposition:first_order_perturbation}]
For $\epsilon > 0$ small, we have $P_{c(\epsilon),\epsilon} \psi_\epsilon = 0$. Using Lemma \ref{lemma:perturbation_eigenvector}, we find that, in $M_{a,b}$ away from $a$ and $b$, we have
\begin{equation*}
    P_{c(\epsilon),\epsilon}\psi_\epsilon \underset{\epsilon \to 0}{=} \epsilon \left( P_{c_1,0} \dot{\psi}_0 - \dot{c}(0) (\partial_x^2 - \alpha^2) \psi_0 \right) + \mathcal{O}_{C^\infty}(\epsilon^2).
\end{equation*}
Hence, we find that $P_{c_1,0} \dot{\psi}_0 = \dot{c}(0) (\partial_x^2 - \alpha^2) \psi_0$ in $M_{a,b}$. Comparing with Lemmas \ref{lemma:perturbation_eigenvector} and \ref{lemma:formula_projector}, we find that
\begin{equation*}
    \dot{c}(0) = \frac{\int_{M_{a,b}} P_{c_1,0} E_{c_1,0} \psi_0 \frac{\psi_0}{U - c_1}\mathrm{d}x}{\int_{M_{a,b}} \psi_0 \frac{(\partial_x^2 - \alpha^2) \psi_0}{U - c_1}\mathrm{d}x}.
\end{equation*}
Since $E_{c_1,0} \psi_0(a) = \lambda_{c_1,0} \partial_x \psi_0(a)$ and $E_{c_1,0} \psi_0(b) = \mu_{c_1,0} \partial_x \psi_0(b)$, the result follows integrating by parts.
\end{proof}

\section{Description of resonant states}\label{section:description_resonant_states}

In this section, we explain how the resonant states, that have been defined as smooth functions on $\overline{M}_{a,b}$ in \S \ref{subsection:definition_resonances}, can be interpreted as distributions on $(a,b)$ when the associated parameter $c$ has a non-negative imaginary part, proving Theorem \ref{theorem:resonances}.

For $c \in (c_0 - \delta,c_0 + \delta)$, we start by studying the possible singularity of solutions to $P_{c,0} \psi = 0$ near $U^{-1}(\set{c})$ that extend to some complex region in \S \ref{subsection:ODE_lemma}. Then, we use this knowledge in \S \ref{subsection:description_resonant_states} to relate the elements of $\Omega(c)$ with the resonant states defined in \S \ref{subsection:definition_resonances}.

\subsection{A singular ODE}\label{subsection:ODE_lemma}

In order to understand the singularities of resonant states, we will study in this subsection operators of the form $P = z \partial_z^2 - w(z)$ near $0$ in the complex plane (notice that, dividing by $z$, we get a Fuchsian ODE). We start with a standard fact, see for instance \cite[p.20]{schmid_henningson_book}.

\begin{lemma}\label{lemma:fundamental_solutions}
Let $w$ be a holomorphic function defined on a neighbourhood of zero and such that $w(0) \neq 0$. Let $P$ denote the differential operator $P = z \partial_z^2 - w(z)$. There are two holomorphic functions $\varphi$ and $\psi$ on a neighbourhood of zero such that $\varphi(0) = 0, \varphi'(0) = 1, P \varphi = 0$ and $P (\psi +  \varphi \log) = 0$ (for any determination of the logarithm on a slitted disk). Moreover $\psi(0) \neq 0$.
\end{lemma}

\begin{proof}
Let us search for holomorphic functions $p$ and $q$ near $0$ with $q(0) = 0$ and $P( p(z) + q(z) \log z) = 0$. Let us write
\begin{equation*}
p(z) = \sum_{n \geq 0} p_n z^n \textup{ and } q(z) = \sum_{n \geq 1} q_n z^n \textup{ and } w(z) = \sum_{n \geq 0} w_n z^n.
\end{equation*}
We compute
\begin{equation*}
\begin{split}
& P (p (z) + q(z) \log (z) ) \\ 
& = \sum_{n \geq 1} \left( p_{n+1} n(n+1) + q_{n+1} (2n +1) - \sum_{k + \ell = n} w_k p_\ell\right) z^n  + q_1 - w_0 p_0 \\ 
& \qquad + \left( \sum_{n \geq 1 } \left(q_{n+1}n (n+1) - \sum_{k + \ell = n} w_k q_\ell\right)z^n \right) \log z.
\end{split}
\end{equation*}
Hence, the equation $P( p(z) + q(z) \log z) = 0$ may be rewritten as 
\begin{equation*}
\begin{cases}
q_{n+1}  = \frac{1}{n(n+1)} \sum_{k + \ell = n} w_k q_\ell \textup{ for } n \geq 1 \\
p_0  = \frac{q_1}{w_0} \\
p_{n+1}  =  \frac{1}{n(n+1)}\left( (2n+1) q_{n+1} + \sum_{k + \ell = n} w_k p_\ell \right) \textup{ for } n \geq 1.
\end{cases}
\end{equation*}
This is a triangular systems of equations (the equations may be solved inductively in the order we have written them). Notice that we may impose the values of $q_1$ and $p_1$. 

Let us prove that the resulting power series has a positive radius of convergence. Let $C, R  > 0$ be constants such that $|w_n | \leq C R^n$ for every $n \geq 0$. Let $A = |q_1|$ and $\rho > \max(R,1)$ be such that $C/(\rho - R) \leq 1$. Let us prove by induction that we have $|q_n| \leq A \rho^n$ for every $n \geq 1$. It follows from the choice of $A$ that it is true for $n = 1$. Now, let $n \geq 1$ be such that $|q_k| \leq A \rho^k$ for $k = 0,\dots, n$ and compute
\begin{equation*}
|q_{n+1}| \leq \frac{CA}{n(n+1)} \sum_{k = 0}^{n} R^k \rho^{n-k}  \leq C A \frac{\rho^n}{1 - \frac{R}{\rho}} \leq A \rho^{n+1} \frac{C}{\rho - R} \leq A \rho^{n+1}.
\end{equation*}
This ends the proof by induction: the power series $\sum_{n \geq 1} q_n z^n$ has a positive radius of convergence. Let then $\varrho > \max(\rho,1)$ be such that $C/ (\varrho - R) \leq 1/2$ and $B \geq \max(1, |p_0|,|p_1|)$ be such that $2A/ B + 1/2 \leq 1$. We prove as above that $|p_n| \leq B \varrho^n$ for every $n \geq 0$, which proves that the series $\sum_{n \geq 0} p_n z^n$ has a positive radius of convergence. 

In order to construct the function $\varphi$, we take $q_1 = 0$ and $p_1 = 1$ in the construction above, then it follows from the inductive scheme that the resulting function $q$ is zero and that $p(0) = 0$. The function $p$ in that case is the solution $\varphi$ we are searching for. To construct $\psi$, set $q_1 = 1$ and  $p_1 = 0$. We see from the inductive scheme that the resulting functions $p$ and $q$ satisfy $q = \varphi$, so that we can set $p = \psi$. Notice then that $\psi(0) = w(0)^{-1} \neq 0$.
\end{proof}

In order to deal with resonant states that are not in the kernel of $P_{c,0}$, we will need to sharpen Lemma \ref{lemma:fundamental_solutions}. We start with a useful remark:

\begin{lemma}\label{lemma:primitive_log_polynomial}
Let $n_0$ be an integer. Let $f_0,\dots, f_{n_0}$ be meromorphic functions near $0$. There are meromorphic functions $g_0,\dots,g_{n_0 +1}$ near $0$ such that 
\begin{equation*}
\partial_z \left( \sum_{k = 0}^{n_0+1} g_k(z) (\log z)^k \right) = \sum_{k = 0}^{n_0} f_k(z) (\log z)^k
\end{equation*}
in  the intersection of a neighbourhood of zero with any slitted disk on which the logarithm is well-defined.
\end{lemma}

\begin{proof}
Let us prove the result by induction. We start with the case $n_0 = 0$. Let $a$ denote the residue of $f_0$ at $0$. Developing $f_0(z) - a/z$ in Laurent series at $z = 0$, we find that there is a meromorphic function $g_0$ near $0$ such that $\partial_z g_0(z) = f_0(z) - a/z$ and thus $\partial_z( g_0(z) + a \log z) = f_0(z)$. This ends the proof in the case $n_0 = 0$.

Now take $n_0 > 0$ and assume that the result holds true for smaller values of $n_0$. Let $f_0,\dots, f_{n_0}$ be meromorphic functions near $0$. Let $a$ be the residue of $f_{n_0}$ at $z_0$. As above, we find a meromorphic function $g_{n_0}$ near $0$ such that $\partial_z g_{n_0}(z) = f_{n_0}(z) - a/z$. Notice then that
\begin{equation*}
\partial_z \left( g_{n_0} (\log z)^{n_0} + a \frac{(\log z)^{n_0+1}}{n_0 + 1} \right) = f_{n_0}(z) (\log z)^{n_0} + n_0 g_{n_0}(z) (\log z)^{n_0 - 1}. 
\end{equation*}
The induction hypothesis implies that there are meromorphic functions $g_0,\dots,g_{n_0 - 1}$ near $0$ such that
\begin{equation*}
\partial_z \left( \sum_{k = 0}^{n_0 - 1} g_k(z) (\log z)^k \right) = \sum_{k = 0}^{n_0 - 1} f_k(z) (\log z)^k - n_0 g_{n_0}(z)(\log z)^{n_0 - 1}.
\end{equation*}
Hence, we have
\begin{equation*}
\partial_z \left( \sum_{k = 0}^{n_0+1} g_k(z) (\log z)^k \right) = \sum_{k = 0}^{n_0} f_k(z) (\log z)^k.
\end{equation*}
This completes the proof.
\end{proof}

We can now state a result that will allow us to deal with resonant states that are not in the kernel of $P_{c,0}$.

\begin{lemma}\label{lemma:log_generalized_eigenvectors}
Let $w$ be a holomorphic function near zero such that $w(0) \neq 0$ and let $P$ be the differential operator $P = z \partial_z^2 - w(z)$. Let $n \geq 0$ be an integer and $u_0,\dots, u_{n}$ be holomorphic functions on $\mathbb{D}(0,\delta) \cap \mathbb{H}$ for some small $\delta > 0$. Assume that $P u_0 = 0$ and $P u_k = (\partial_z^2 - \alpha^2) u_{k-1}$ for $k = 1,\dots,n$.

There is an integer $n_0$ and meromorphic functions $g_0,\dots,g_{n_0}$ such that 
\begin{equation}\label{eq:log_singularity}
u_n(z) = \sum_{k = 0}^{n_0} g_k(z) (\log z)^k
\end{equation}
for $z$ near $0$ in $\mathbb{D}(0,\delta) \cap \mathbb{H}$ .
\end{lemma}

\begin{proof}
Let us make a proof by induction. We start with the case $n = 0$. Let $\varphi$ and $\psi$ be the holomorphic functions from Lemma \ref{lemma:fundamental_solutions}. Let us consider the Wronskian
\begin{equation*}
W(z) = \varphi \partial_z( \psi(z) + \varphi(z) \log z) - \partial_z \varphi(z)(\psi(z) + \varphi(z) \log z).
\end{equation*}
Differentiating the definition of $W$, we find that $W$ is a constant function, and letting $z$ goes to $0$, we find that this constant value is $- \psi(0) \neq 0$. Hence, $\varphi$ and $\psi + \varphi. \log$ form a system of fundamental solutions for the equation $Pu = 0$. Hence $u_0$ is a linear combination of $\varphi$ and $\psi + \varphi. \log$, proving the result in the case $n = 0$.

Let us move to the case of $n > 0$, assuming that the results hold for smaller values of $n$. The induction hypothesis implies that there is an integer $n_0 \geq 0$ and meromorphic functions $g_0,\dots ,g_{n_0}$ such that 
\begin{equation*}
(\partial_z^2 - \alpha^2) u_{n - 1}(z) = \sum_{k = 0}^{n_0} g_k(z) (\log z)^k 
\end{equation*} 
for $z$ near $0$ in the domain of definition of $u_{n - 1}$. Since $\varphi$ and $\psi + \varphi \log$ form a system of fundamental solutions for the equation $P u = 0$, we know that there are holomorphic functions $p$ and $q$ on $\mathbb{D}(0,\delta) \cap \mathbb{H}$ (up to making $\delta$ smaller) such that 
\begin{equation*}
u_{n}(z) = p(z) \varphi(z) + q(z) (\psi(z) + \varphi(z) \log z)
\end{equation*}
and 
\begin{equation*}
\partial_z u_{n}(z) = p(z) \partial_z \varphi(z) + q(z) \partial_z (\psi(z) + \varphi(z) \log z).
\end{equation*}
We have then 
\begin{equation*}
\begin{pmatrix}
\varphi(z) & \psi(z) + \varphi(z) \log z \\
z \partial_z \varphi(z) & z \partial_z(\psi(z) + \varphi(z) \log z)
\end{pmatrix} \begin{pmatrix}
\partial_z p(z) \\ \partial_z q(z)
\end{pmatrix} = \begin{pmatrix}
0 \\ P u_{n}(z)
\end{pmatrix}.
\end{equation*}
Inverting this system of equation, we find that
\begin{equation*}
p'(z) = z^{-1} \psi(0)^{-1} (\psi(z) + \varphi(z) \log z) (\partial_z^2 - \alpha^2) u_{n - 1}(z) 
\end{equation*}
and
\begin{equation*}
q'(z) = - z^{-1} \psi(0)^{-1} \varphi(z) (\partial_z^2 - \alpha^2) u_{n - 1}(z).
\end{equation*}
The result follows then from Lemma \ref{lemma:primitive_log_polynomial}.
\end{proof}

\subsection{Description of elements of \texorpdfstring{$\Omega(c)$}{Omega(c)} near \texorpdfstring{$U^{-1}(\set{c})$}{U{-1}(set{c})}}\label{subsection:description_resonant_states}

Let us fix $c \in (c_0 - \delta, c_0 + \delta)$. In order to prove Theorem \ref{theorem:resonances}, we want to compare the elements of $\Omega(c)$ and the resonant states of $P_{c,0}$ defined in \S \ref{subsection:definition_resonances} (that are smooth functions on $\overline{M}_{a,b}$). 

\begin{remark}\label{remark:distribution_theory}
Let us make a reminder of distribution theory that will be useful. Let $I$ be an open interval of $\mathbb{R}$ and $\varrho$ be a positive real number. Assume that $F$ is a holomorphic function on $I + i (0,\varrho)$ and that there are $C,N > 0$ such that for every $z \in I + i (0,\varrho)$ we have $|F(z)| \leq C |\im z|^{-N}$. Then the family of smooth functions $x \mapsto F(x + i \rho)$ on $I$ converges as $\rho$ goes to $0$ to a distribution on $I$ that we denote by $F(\cdot + i 0)$. Moreover, the analytic wave front set of $F (\cdot + i 0)$ is contained in $I \times \mathbb{R}_+^*$. See for instance \cite[pp.41-42]{sjostrand82}. Of course, if $G$ is a holomorphic function defined on $I + i (- \varrho, 0)$ satisfying the same bound as $F$, then $G(\cdot - i \rho)$ converges as $\rho$ goes to $0$ to a distribution on $I$ that we denote by $G(\cdot - i 0)$. The wave front set of $G(\cdot - i 0)$ is contained in $I \times \mathbb{R}_-^*$.

Any distribution on an open subset of $\mathbb{R}$ may be written locally as the sum of distributions of the type $F(\cdot + i 0)$ and $G(\cdot - i 0)$. The possibility (or not) to write $u$ locally as $F(\cdot + i0)$ or $G(\cdot - i 0)$ (instead of a sum of two such terms) may be used to define the analytic wave front set of $u$. This is Sato's definition of the wave front set, one may refer for instance to \cite[pp.41-42]{sjostrand82} for a discussion of this definition and a proof that it is equivalent with the other standard definition of analytic wave front set.
\end{remark}

We will start by giving a detailed description of the elements of $\Omega(c)$ near the points of $U^{-1}(\set{c})$. We begin with the simplest case:

\begin{lemma}\label{lemma:inflexion_point}
Let $\psi \in \Omega(c)$. Let $x_0 \in (a,b)$ be such that $U(x_0) = c$ and $U''(x_0) = 0$. Then $\psi$ is analytic near $x_0$.
\end{lemma}

\begin{proof}
By assumption, there is $n \geq 0$ such that $P_{c,0} ((\partial_x^2 - \alpha^2)^{-1} P_{c,0})^n \psi =0$. For $k = 0,\dots, n$ let $\psi_k = ((\partial_x^2 - \alpha^2)^{-1} P_{c,0})^{n-k} \psi$. Let us prove first that $\psi_0$ is analytic near $x_0$. We have
\begin{equation*}
0 = P_{c,0} \psi_0 = (U - c)\left( \partial_x^2 \psi_0 + \left( \frac{U''}{U - c} - \alpha^2 \right) \psi_0 \right),
\end{equation*}
where the function $U''/ (U - c)$ is analytic near $x_0$. Hence, the function
\begin{equation*}
\partial_x^2 \psi_0 + \left( \frac{U''}{U - c} - \alpha^2 \right) \psi_0
\end{equation*}
is a multiple of the Dirac mass at $x_0$ (since we have $U'(x_0) \neq 0$). Since its wave front set near $x_0$ is contained in a half-line, we find that
\begin{equation*}
\partial_x^2 \psi_0 + \left( \frac{U''}{U - c} - \alpha^2 \right) \psi_0 = 0.
\end{equation*}
Indeed, the wave front set of a Dirac mass is a full line. By elliptic regularity (which in dimension one may just be seen as a consequence of the holomorphic Cauchy--Lipschitz theorem), we find that $\psi_0$ is analytic near $x_0$. Let us now prove by induction that $\psi_k$ is analytic for $k = 1,\dots,n$.

Let $k \in \set{1,\dots,n}$ and assume that $\psi_\ell$ is analytic for $\ell = 0,\dots, k-1$. Notice that $P_{c,0} \psi_k = (\partial_x^2 - \alpha^2) \psi_{k-1}$. As above
\begin{equation*}
(U - c)\left( \partial_x^2 \psi_k + \left( \frac{U''}{U - c} - \alpha^2 \right) \psi_k \right) = (\partial_x^2 - \alpha^2) \psi_{k-1}
\end{equation*}
implies that there is $\lambda \in \mathbb{C}$ such that
\begin{equation*}
\partial_x^2 \psi_k + \left( \frac{U''}{U - c} - \alpha^2 \right) \psi_k = \lambda \delta_{x_0} +  \frac{(\partial_x^2 - \alpha^2) u_{k-1}}{U - c+ i0}.
\end{equation*}
We use again the fact that the wave front set of the left hand side is contained in a half line to find that $\lambda = 0$. Thus, we have
\begin{equation*}
\partial_x^2 \psi_k + \left( \frac{U''}{U - c} - \alpha^2 \right) \psi_k = \frac{(\partial_x^2 - \alpha^2) \psi_{k-1}}{U - c+ i0},
\end{equation*}
which implies that the analytic wave front set of $\partial_x^2 \psi_k + \left( \frac{U''}{U - c} - \alpha^2 \right) \psi_k $ near $x_0$ is contained in $\set{x_0} \times \mathbb{R}_-^*$. The same reasoning replacing $(U - c + i 0)^{-1}$ by $(U - c - i0)^{-1}$ implies that the analytic wave front set of $\partial_x^2 \psi_k + \left( \frac{U''}{U - c} - \alpha^2 \right) \psi_k $ near $x_0$ is contained in $\set{x_0} \times \mathbb{R}_+^*$, and thus that $\partial_x^2 \psi_k + \left( \frac{U''}{U - c} - \alpha^2 \right) \psi_k$ is analytic near $x_0$. By elliptic regularity, we find that $\psi_k$ is analytic near $x_0$. This ends the induction. We proved in particular that $\psi_n = \psi$ is analytic.
\end{proof}

Let us describe now the elements of $\Omega(c)$ near points of $U^{-1}(\set{c})$ that are not inflexion points.

\begin{lemma}\label{lemma:singularity_positive_derivative}
Let $\psi \in \Omega(c)$. Let $x_0 \in (a,b)$ be such that $U(x_0) = c_0$ and $U'(x_0) > 0$. There is an integer $n_0 \geq 0$, meromorphic functions $f_0,\dots,f_{n_0}$ near $x_0$ and $\varrho > 0$ such that
\begin{equation*}
\psi_{|(x_0 - \varrho, x_0 + \varrho)} =  F(\cdot - i 0)
\end{equation*}
where\footnote{The choice of the determination of the logarithm is irrelevant, but the functions $f_0,\dots, f_{n_0}$ depend on this choice.}
\begin{equation}\label{eq:singularity_solution}
F(z) = \sum_{k = 0}^{n_0} f_k(z) (\log (z-x_0))^k.
\end{equation}
\end{lemma}

\begin{proof}
We can assume $U''(x_0) \neq 0$, since otherwise the result follows from Lemma \ref{lemma:inflexion_point}. Let $n \geq 0$ be such that $Q_{c,0}^n P_{c,0} \psi = 0$. For $k = 0,\dots, n$, set $\psi_k =((\partial_x^2 - \alpha^2)^{-1} P_{c_0,0})^{n-k} \psi$. Since $(\partial_x^2 - \alpha^2)^{-1} P_{c,0}$ is a pseudo-differential operator, we find that for $k = 0,\dots, n-1$, the wave front set of $\psi_{k|(x_0 - \varrho, x_0 + \varrho)}$ is contained in $\set{x_0} \times \mathbb{R}_-^*$, if $\varrho$ is small enough. By elliptic regularity, we deduce from $P_{c,0} \psi_0 = 0$ that $\psi_0$ is analytic on $(x_0 - \varrho, x_0 + \varrho) \setminus \set{x_0}$. We deduce then from \cite[Theorem 2]{galkowski_zworski_hypoellipticity} (take $\Lambda = \set{x_0} \times \mathbb{R}_+^*$), that the analytic wave front set of $\psi_{0|(x_0 - \varrho, x_0 + \varrho)}$ is contained in $\set{x_0} \times \mathbb{R}_-^*$. Using the relation $P_{c,0} \psi_k = (\partial_x^2 - \alpha^2) \psi_{k-1}$, valid for $k = 1,\dots, n$, we find by induction, using the same reasoning, that the wave front set of $\psi_{k|(x_0 - \varrho, x_0 + \varrho)}$ is contained in $\set{x_0} \times \mathbb{R}_-^*$ for $k = 0,\dots,n$.

It follows then from \cite[Theorem 6.5]{sjostrand82} that, up to making $\varrho$ smaller, for $k = 0,\dots,n$, there is a holomorphic function $F_k : (x_0 - \varrho, x_0 + \varrho) + i (- \varrho, 0) \to \mathbb{C}$ such that 
\begin{equation*}
\psi_{k|(x_0 - \varrho, x_0 + \varrho)} = F_k(\cdot - i 0).
\end{equation*}
Since $\psi_k$ is analytic on $(x_0 - \varrho, x_0 + \varrho) \setminus \set{x_0}$, we find that $F_k$ has a holomorphic extension to a neighbourhood of $(x_0 - \varrho, x_0 + \varrho) + i (- \varrho, 0] \setminus \set{x_0}$ whose restriction to $(x_0 - \varrho, x_0 + \varrho) \setminus \set{x_0}$ coincides with $\psi_k$. Thus, the analytic continuation principle implies that $P_{c,0} F_0 = 0$ and $P_{c,0} F_k = (\partial_x^2 - \alpha^2)F_{k-1}$ for $k = 1,\dots,n$.

We can consequently apply Lemma \ref{lemma:log_generalized_eigenvectors} (after an affine change of variables and a division by $(U(z) - c)/z$) to find that $F_n$ is of the form \eqref{eq:singularity_solution}. Since $\psi_n = \psi$, the result follows.
\end{proof}

Similarly, we have:

\begin{lemma}\label{lemma:singularity_negative_derivative}
Let $\psi \in \Omega(c_0)$. Let $x_0 \in (a,b)$ be such that $U(x_0) = c_0$ and $U'(x_0) < 0$. There is an integer $n_0 \geq 0$, meromorphic functions $f_0,\dots,f_{n_0}$ near $x_0$ and $\varrho > 0$ such that
\begin{equation*}
\psi_{|(x_0 - \varrho, x_0 + \varrho)} =  F(\cdot + i 0)
\end{equation*}
where 
\begin{equation}\label{eq:singularity_solution_bis}
F(z) = \sum_{k = 0}^{n_0} f_k(z) (\log (z-x_0))^k.
\end{equation}
\end{lemma}

We are now ready to prove Theorem \ref{theorem:resonances}. 

\begin{figure}[b]
   \centering
   \includegraphics[width=0.7\textwidth]{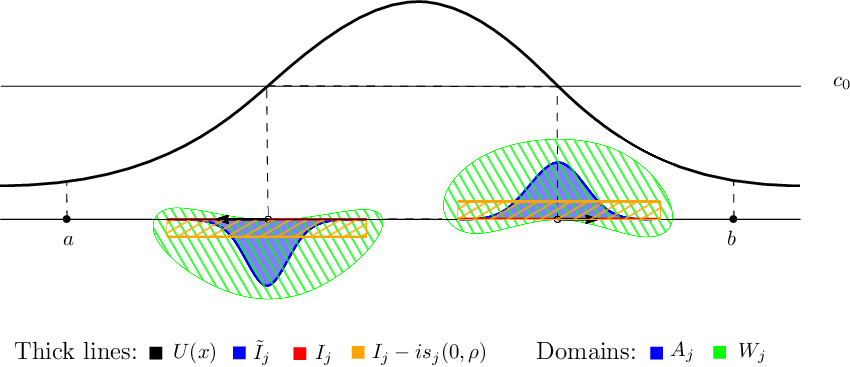}
    \caption{Complex deformations used in the proof of Theorem \ref{theorem:resonances}. Arrows indicate the wave front sets of resonant states. Starting with an element of $\Omega(c)$, we use Lemmas \ref{lemma:singularity_positive_derivative} and \ref{lemma:singularity_negative_derivative} to extend it as a holomorphic function on the interior of the orange rectangles. We use then Lemma \ref{lemma:elliptic_extension} to extend this holomorphic function to the green region. The resulting function may be restricted to the blue lines, and thus define a smooth function on $M_{a,b}$}
   \label{fig:deform}
\end{figure}

\begin{proof}[Proof of Theorem \ref{theorem:resonances}]
Let us start with the case $c \in (c_0 - \delta,c_0 + \delta)$. Our goal is to construct a linear isomorphism between the space of resonant states for $P_{c,0}$ (as defined in Remark \ref{remark:multiplicity_resonances}, these are smooth functions on $\overline{M}_{a,b}$) and $\Omega(c)$ (this is a space of distributions on $(a,b)$). This will be done using Lemma \ref{lemma:elliptic_extension} by ``extending'' the elements of $\Omega(c)$ up to $\overline{M}_{a,b}$, and the resonant states up to $[a,b]$.

Recall the compact set $K$ and the function $m$ from \S \ref{section:escape_function} and let $K_0 \subseteq (a,b)$ denote the support of the function $m$ defined in \S \ref{section:escape_function}. Cover $K_0$ by the interiors of a finite number of disjoint segment $I_1,\dots, I_N$ such that $I_j \subseteq \set{x \in (a,b) : d(x,K) < r}$ and $U'$ does not vanish on $I_j$ for $j= 1,\dots,N$ (see Figure \ref{fig:deform}).

Let $\psi$ be an element of $\Omega(c)$. Assume first that $P_{c,0} \psi = 0$. Let $j \in \set{1,\dots,N}$, and consider the set 
\begin{equation*}
A_j = \set{ x + i t m(x) : x \in I_j, t \in [0,1]} \setminus U^{-1}(\set{c}).
\end{equation*}
It follows from \ref{item:non_zero_ii} in Lemma \ref{lemma:non_zero_function} that there is a neighbourhood $W_j$ of $A_j$ in $\mathbb{C}$ such that $U$ has a holomorphic extension to $W_j$ that never takes the value $c$. Moreover, up to making $W_j$ smaller, we may assume that it is simply connected. Let $s_j$ be the sign of $U'$ on $I_j$. There is a holomorphic function $F_j$ on $I_j - i s_j (0, \rho)$ for some small $\rho > 0$ such that $\psi_{|I_j} = F(\cdot - i s_j 0)$. Indeed, the extension is obtained near points of $U^{-1}(\set{c}) \cap I_j$ by Lemma \ref{lemma:singularity_positive_derivative} or Lemma \ref{lemma:singularity_negative_derivative} (depending on the sign of $s_j$) and near the other points by elliptic regularity. Notice that, up to making $\rho$ smaller, we may assume that $I_j - i s_j (0, \rho) \subseteq W_j$. It follows then from Lemma \ref{lemma:elliptic_extension} that $F_j$ extends to a holomorphic function on $W_j$, since $P_{c,0}$ is elliptic on $W_j$. Let then $\phi_j$ be the restriction of $F_j$ on $\widetilde{I}_j = \set{x + i m(x) : x \in I_j}$, and notice that $\phi_j$ coincides with $\psi$ near the extremities of $I_j$ ($\widetilde{I}_j$ and $I_j$ coincide near their extremities). Notice that we have $P_{c,0} \phi_j = 0$.

Since the smooth functions $\psi_{|[a,b] \setminus K_0}, \phi_1,\dots,\phi_N$ coincide where they are commonly defined, and the interiors of their domains cover $\overline{M}_{a,b}$, they define a smooth function $\phi$ on $\overline{M}_{a,b}$ that satisfies $P_{c,0} \phi = 0$ and coincides with $\psi$ where $\overline{M}_{a,b}$ and $[a,b]$ coincide. We write $\phi = B \psi$. We want now to define $B\psi$ for any $\psi \in \Omega(c)$. To do so, we let $m \geq 0$ be such that $Q_{c,0}^m P_{c,0} \psi = 0$ and write $\psi_k = ((\partial_x^2 - \alpha^2)^{-1} P_{c,0})^{m-k}{\psi}$ for $k = 0,\dots,m$. Since $P_{c,0} \psi_0 = 0$, we may extend $\psi_0$ as a holomorphic function on the $W_j$'s as above. Then, using the equation $P_{c,0} \psi_1 = (\partial_x^2 - \alpha^2) \psi_0$ instead of $P_{c,0} \psi_0$, we may extend $\psi_1$ to the $W_j$'s. Iterating this procedure we end up extending $\psi = \psi_m$, and we can consequently define $B\psi \in \mathfrak{S}_{c,0}(\overline{M}_{a,b})$ as above.

We want now to prove that the linear map $B : \Omega(c) \to \mathfrak{S}_{c,0}(\overline{M}_{a,b})$ is an isomorphism. It follows from the analytic continuation principle and the construction of $B$ above that $B$ is injective. Let us construct a right inverse for $B$. Pick $\phi \in \mathfrak{S}_{c,0}(\overline{M}_{a,b})$. As above, we start by dealing with the case in which $P_{c,0} \phi = 0$. Let $j \in \set{1,\dots,N}$. It follows from Lemma \ref{lemma:elliptic_extension} that there is a holomorphic function $F_j$ on $W_j$ that coincides $\phi$ on $\widetilde{I}_j$ and satisfies $P_{c,0} F_j = 0$. It follows from Lemma \ref{lemma:log_generalized_eigenvectors} that the singularities of $F_j$ near the points of $U^{-1}(\set{c})$ are of the kind \eqref{eq:log_singularity}. In particular, the distribution $F_j(\cdot -i s_j 0)$ is well-defined on the interior of $I_j$, and its wave front set is contained in $\set{ (x,\xi): x \in I_j \cap U^{-1}(\set{c}), s_j \xi > 0}$. Notice that this distribution coincides with $\phi$ near the boundary points of $I_j$. By gluing the distributions $F_j(\cdot - i s_j 0)$ for $j = 1,\dots, N$ and $\phi_{|[a,b] \setminus K_0}$, we define an element $G \phi$ of $\Omega(c)$ that coincides with $\phi$ when $\overline{M}_{a,b}$ and $[a,b]$ coincide. We extend this definition to a general element of $\mathfrak{S}_{c,0}(\overline{M}_{a,b})$ as in the construction of $B$. It follows from the constructions of $B$ and $G$ that $B G \phi = \phi$, and thus $B$ is surjective.

We proved that $B$ is a linear isomorphism between $\Omega(c)$ and $\mathfrak{S}_{c,0}(\overline{M}_{a,b})$, and \ref{item:resonances_real} follow. The proof of \ref{item:resonances_easy} follows a similar strategy but is easier. Indeed, if $\im c > 0$, then the operator $P_{c,0}$ is elliptic on $[a,b]$, and we can consequently do the same proof as above but with the $W_j$'s that contain fully the segments $I_j$'s.
\end{proof}

\section{Alternative characterization of resonances}\label{section:alternative_characterizations}

In this section, we explicit the relation between resonances introduced in Theorem~\ref{theorem:limit} and related notions in the literature. We start by explaining in \S \ref{subsection:jost_functions} how to recover resonances using ODE theory. We prove then that, under the assumptions of Theorem~\ref{theorem:limit}, the real resonances coincide with the ``embedding eigenvalues'' from \cite{generalized_eigenvalue_segment} (an adaptation of this notion to the circle case appears in \cite{generalized_eigenvalue_circle}).

\subsection{Resonances and Wronskian determinant}\label{subsection:jost_functions}

It follows from Cauchy--Lipschitz Theorem and Lemma \ref{lemma:non_zero_function} that for every $c \in (c_0 - \delta, c_0 + \delta) + i (- \delta,+ \infty)$ there are $C^\infty$ functions $f_c^+, f_c^- : \overline{M}_{a,b} \to \mathbb{C}$ such that $P_{c,0} f_c^+ = P_{c,0} f_c^{-} = 0, f_c^+(a) = f_c^{-}(b) = 0$ and $\partial_x f_c^+ (a) = \partial_x f_c^{-}(b) = 1$. Moreover, $f_c^+$ and $f_c^-$ are depend holomorphically on $c$. A standard argument in ODE theory implies that the Wronskian of $f_c^+$ and $f_c^-$ is a constant function. We denote by $W(c)$ its constant value, this is a holomorphic function of $c$ on $(c_0 - \delta, c_0 + \delta) + i (- \delta, + \infty)$. The holomorphic continuation of the Wronskian determinant appears for instance in \cite{rosencrans_sattinger_spectrum, stepin_rayleigh}. The function $W(c)$ may also be used to locate resonances:

\begin{proposition}
Let $c_1 \in (c_0 - \delta,c_0 + \delta) + i (- \delta, + \infty)$. Then, $W(c_1) = 0$ if and only if $c_1 \in \mathcal{R}$. When it happens, the order of vanishing of $W$ at $c_1$ coincides with the multiplicity of $c_1$ as an element of $\mathcal{R}$.
\end{proposition}

\begin{proof}
Notice that $W(c_1) = f_{c_1}^+(b)$. Hence, if $W(c_1) = 0$, we find that $f_{c_1}^+$ satisfies Dirichlet boundary condition and thus $c_1 \in \mathcal{R}$. Reciprocally, if $c_1 \in \mathcal{R}$, there is a non-zero $g \in C^\infty(\overline{M}_{a,b})$ such that $g(a) = g(b) = 0$ and $P_{c_1,0} g= 0$. Since $g$ vanishes at $a$, it is a non-zero multiple of $f_{c_1}^+$, and thus $f_{c_1}^+(b) = 0$, i.e. $W(c_1) = 0$.

Let us now assume that $W(c_1) = 0$ and let $n$ be the order of vanishing of $W$ at $c_1$. Let $n'$ be the multiplicity of $c_1$ as an element of $\mathcal{R}$, and recall that $n'$ is the dimension of 
\begin{equation*}
\mathfrak{S}_{c_1,0}(\overline{M}_{a,b}) = \set{ u \in C^\infty(\overline{M}_{a,b}) : u(a) = u(b) = 0 \textup{ and } \exists p \geq 0, \ Q_{c,0}^p P_{c,0} u = 0}.
\end{equation*}
For every integer $k \geq 1$, differentiating $P_{c,0} f_c^+ = 0$, we find that
\begin{equation}\label{eq:generalized_eigenvector}
P_{c,0} \frac{\mathrm{d}^k}{\mathrm{d}c^k} f_c^+ = k (\partial_x^2 - \alpha^2)\frac{\mathrm{d}^{k-1}}{\mathrm{d}c^{k-1}} f_c^+.
\end{equation}
Since $W$ vanishes at order $n$ at $c_1$, we have
\begin{equation*}
f_{c_1}^+ (b) = \frac{\mathrm{d}}{\mathrm{d}c}f_{c_1}^+(b) = \dots = \frac{\mathrm{d}^{n-1}}{\mathrm{d}c^{n-1}} f_{c_1}^+(b) = 0.
\end{equation*}
Hence, we find that
\begin{equation*}
Q_{c_1,0}^{k} P_{c_1,0} \left( \frac{\mathrm{d}^k}{\mathrm{d}c^k}f_{c_1}^+ \right) = 0
\end{equation*}
for $k = 0,\dots,n-1$. Thus, $f_{c_1}^+,\frac{\mathrm{d}}{\mathrm{d}c}f_{c_1}^+, \dots , \frac{\mathrm{d}^{n-1}}{\mathrm{d}c^{n-1}}f_{c_1}^+$ belong to $\mathfrak{S}_{c_1,\epsilon}(\overline{M}_{a,b})$. Moreover, it follows from \eqref{eq:generalized_eigenvector} that $f_{c_1}^+,\frac{\mathrm{d}}{\mathrm{d}c}f_{c_1}^+, \dots , \frac{\mathrm{d}^{n-1}}{\mathrm{d}c^{n-1}}f_{c_1}^+$ are linearly independent (starting with a non-trivial linear combination of them, apply $Q_{c_1,0}^k P_{c_1,0}$ with a well-chosen $k$ to get a non-zero multiple of $f_{c_1}^+$). Hence, $n' \geq n$.

Let us prove the reverse inequality. Let $\psi \in \mathfrak{S}_{c_1,0}(\overline{M}_{a,b})$ and let $p$ be the smallest integer such that $Q_{c_1,0}^p P_{c_1,0} \psi = 0$. Notice that $P_{c_1,0} ((\partial_x^2 - \alpha^2)^{-1} P_{c_1,0})^p \psi = 0$. Hence, there is $\lambda \in \mathbb{C}$ such that $((\partial_x^2 - \alpha^2)^{-1} P_{c_1,0})^p \psi = \lambda f_c^+$. If $p > 0$, then we have
\begin{equation*}
P_{c_1,0} ((\partial_x^2 - \alpha^2)^{-1} P_{c_1,0})^{p-1} \psi = \lambda (\partial_x^2 - \alpha^2) f_{c_1}^+.
\end{equation*}
It follows that $((\partial_x^2 - \alpha^2)^{-1} P_{c_1,0})^{p-1} \psi$ is a linear combination of $f_{c_1}^+$ and $\frac{\mathrm{d}}{\mathrm{d}c} f_{c_1}^+$. Using the same argument, we find by induction that for $k = 0,\dots, p$, the function $((\partial_x^2 - \alpha^2)^{-1} P_{c_1,0})^{p-k} \psi$ is a linear combination of $f_{c_1}^+,\frac{\mathrm{d}}{\mathrm{d}c}f_{c_1}^+, \dots , \frac{\mathrm{d}^{k}}{\mathrm{d}c^{k}}f_{c_1}^+$. Notice that the coefficient of $\frac{\mathrm{d}^{k}}{\mathrm{d}c^{k}}f_{c_1}^+$ in this linear combination must be non-zero (otherwise we would contradict the minimality of $p$). Hence, we must have $p \leq n-1$ (otherwise, by taking $k = p-n$, we would get $\frac{\mathrm{d}^m}{\mathrm{d}c^m} f_{c_1,+}(b) = 0$, which contradicts the definition of $n$). Finally, we found that $\mathfrak{S}_{c_1,0}(\overline{M}_{a,b})$ is contained in the span of $f_{c_1}^+,\frac{\mathrm{d}}{\mathrm{d}c}f_{c_1}^+, \dots , \frac{\mathrm{d}^{m-1}}{\mathrm{d}c^{m-1}}f_{c_1}^+$. By the previous point, this is actually an equality, and thus $n = n'$.
\end{proof}

\subsection{Generalized embedded eigenvalues following 
\texorpdfstring{\cite{generalized_eigenvalue_segment}}{Wei et al. (202X)}}\label{subsection:generalized_eigenvalue}

In \cite[Definition 5.1]{generalized_eigenvalue_segment}, Wei, Zhang and Zhao introduces a generalization of the notion of embedded eigenvalues in order to discuss linear inviscid damping. This notion has also been used in the circle case in \cite{generalized_eigenvalue_circle}. The following result proves that, in the context of Theorems \ref{theorem:limit} and \ref{theorem:resonances}, the generalized embedded eigenvalues from \cite{generalized_eigenvalue_segment} coincide with the real resonances.

\begin{proposition}\label{proposition:generalized_embedded_eigenvalues}
Let $c \in \mathbb{R}$. Assume that $c$ satisfies the hypothesis \ref{hypothesis}. Then, the following are equivalent :
\begin{enumerate}[label = (\roman*)]
\item the space $\Omega(c)$ is non-trivial;\label{item:our_definition}
\item there is a non-zero $\psi \in H_0^1(a,b)$ such that \label{item:their_definition}
\begin{equation}\label{eq:generalized_embedded_eigenvalues}
(\partial_x^2 - \alpha^2)\psi - \pv \left(\frac{U'' \psi}{U - c}\right) - i \pi \sum_{x \in U^{-1}(\set{c})} \frac{U''(x) \psi(x)}{|U'(x)|} \delta_x = 0.
\end{equation}
\end{enumerate}
\end{proposition}

\begin{remark}
Let us explain why the principal value (denoted by $\pv$) in \eqref{eq:generalized_embedded_eigenvalues} makes sense. From hypothesis \ref{hypothesis}, we know that $c$ is a regular value of $U$, and thus the principal value of $1/(U - c)$ makes sense. It is a distribution of any positive order on $[a,b]$. If $\psi \in H_0^1(a,b)$, then $U'' \psi$ is $\frac{1}{2}$-H\"older-continuous, and thus the distribution $\pv \left(\frac{U'' \psi}{U - c}\right) = U'' \psi \pv \left(\frac{1}{U - c}\right)$ makes sense.
\end{remark}

\begin{proof}[Proof of Proposition \ref{proposition:generalized_embedded_eigenvalues}]
Assume that \ref{item:our_definition} holds, and take a non-trivial $\psi \in \Omega(c)$ such that $P_{c,0} \psi = 0$. It follows from the Lemmas \ref{lemma:singularity_positive_derivative} and \ref{lemma:singularity_negative_derivative} that the singularities of $u$ are at most of the type $x \log(x \pm i 0)$, and thus we find that $u$ belongs to $H_0^1(a,b)$. Now, let us consider the distribution 
\begin{equation*}
\phi \coloneqq (\partial_x^2 - \alpha^2)\psi - \pv \left(\frac{U'' \psi}{U - c}\right) - i \pi \sum_{x \in U^{-1}(\set{c})} \frac{U''(x) \psi(x)}{|U'(x)|} \delta_x.
\end{equation*}
We have $(U - c) \phi = P_{c,0} \psi = 0$. Hence, since $c$ is a regular value of $U$, for every $x \in U^{-1}(\set{c})$ there is a complex number $\lambda_x$ such that $\phi = \sum_{x \in U^{-1}(\set{c})} \lambda_x \delta_x$. Take $x_0 \in U^{-1}(\set{c})$, and assume for instance that $U'(x_0) > 0$. We know then that there is $\varrho > 0$ and a holomorphic function $F$ on $(x_0 + \varrho, x_0 - \varrho) + i (- \varrho, 0)$ such that $\psi_{|(x_0 - \varrho, x_0 + \varrho)} = F(\cdot - i 0)$. Up to taking $\varrho$ smaller, we may assume that $U- c$ does not vanish on $(x_0 - \varrho, x_0 + \varrho) + i (- \varrho, 0)$ and define $G = F/(U- c)$. The distribution $G(\cdot- i 0)$ on $(x_0 - \varrho,x_0 + \varrho)$ has its wave front set contained in $(x_0 - \varrho, x_0 + \varrho) \times \mathbb{R}_-^*$. Moreover, using the formula $(x - i 0)^{-1} = \pv(1/x) + i \pi \delta_0$ (see for instance \cite[(2.2.10)]{friedlander_joshi_1998}), we find that $G(\cdot - i0) = \pv(\psi /(U - c)) + i \pi/U'(x_0) \delta_{x_0}$. Hence, we find that $\phi$ is given near $x_0$ by $(\partial_x^2 - \alpha^2)\psi - U'' G(\cdot - i0)$. Thus, the wave front set of $\phi$ near $x_0$ is contained in $(x_0 - \varrho, x_0 + \varrho) \times \mathbb{R}_-^*$, which imposes $\lambda_{x_0} = 0$ (the wave front set of a Dirac mass is a full line). We can work similarly near each point of $U^{-1}(\set{c})$, and thus we have $\phi = 0$. Recalling the definition of $\phi$, we get \ref{item:their_definition}.

Let us now assume that \ref{item:their_definition} holds, and let $\psi$ be as in \ref{item:their_definition}. We have $P_{c,0} \psi = 0$, and consequently we only need to check that wave front set condition to prove that $\psi \in \Omega(c_0)$. From $P_{c,0} \psi = 0$, we get that $\psi$ is smooth away from $U^{-1}(\set{c})$. Pick a point $x_0 \in U^{-1}(\set{c})$. Notice that the map $x \mapsto U(x) - c$ induces a real-analytic diffeomorphism between a neighbourhood of $x_0$ and a neighbourhood of $0$. Let $g$ be the inverse diffeomorphism. We shall consider the distribution $\phi = \psi \circ g$ near $0$. Near $x_0$, the equation satisfied by $\psi$ writes
\begin{equation*}
(\partial_x^2 - \alpha^2) \psi - \pv \left(\frac{U'' \psi}{U - c}\right) - i \pi \frac{U''(x_0) \psi(x_0)}{|U'(x_0)|} \delta_{x_0} = (\partial_x^2 - \alpha^2) \psi - \frac{U'' \psi}{U - c - i s_{x_0} 0} = 0, 
\end{equation*}
where $s_{x_0}$ is the sign of $U'(x_0)$. It follows that the distribution $\phi$ satisfies near $0$ the equation 
\begin{equation}\label{eq:non_standard_equation}
A \phi  - U'' \circ g \frac{\phi}{x - i 0} = 0,
\end{equation}
where $A$ is an elliptic differential operator of order $2$. Here, the sign $s_{x_0}$ disappeared because if $s_{x_0} = - 1$ then $g$ is orientation reversing. For the same reason, the wave front set condition that we want to prove for $\phi$ is always $\WF(\phi) \subseteq \set{0} \times \mathbb{R}_-^*$. If $U''(x_0) = 0$, then we find that $U'' \circ g/(x- i 0)$ is a smooth function and thus $\phi$ is in the kernel of an elliptic differential operator, proving that $\phi$ is smooth near zero. We will consequently assume that $U''(x_0) \neq 0$, which allows us to divide \eqref{eq:non_standard_equation} by $U'' \circ g$ and find that we have near zero
\begin{equation}\label{eq:better_non_standard_equation}
B \phi - \frac{\phi}{x - i 0} = 0,
\end{equation}
where $B$ is also an elliptic differential operator of order $2$. Let $\chi$ be a $C^\infty$ function supported near $0$ and such that $\chi \equiv 1$ on a small neighbourhood of zero. Let $E$ denote the Fourier multiplier associated to the characteristic function of $\mathbb{R}_+$:
\begin{equation*}
Ef(x) = \frac{1}{2 \pi} \int_{0}^{+ \infty} e^{i x \xi} \hat{f}(\xi) \mathrm{d}\xi.
\end{equation*}
Let $h = \chi \phi$ and notice that the Fourier transform of $E(h/(x - i 0))$ is
\begin{equation*}
\xi \mapsto i  \mathds{1}_{\mathbb{R}_+}(\xi) \int_{\xi}^{+ \infty} \hat{h}(\eta) \mathrm{d}\eta.
\end{equation*}
Here, we used the formula for the Fourier transform of $(x - i 0)^{-1}$ given in \cite[Example 7.1.17]{hormander_book1}. Hence, we find that if $Eh \in H^\ell(\mathbb{R})$ for some $\ell > 1/2$ and $k < \ell - 1$ then $E(h/(x- i 0))$ belongs to $H^k(\mathbb{R})$. By assumption $h$ belongs to $H^1(\mathbb{R})$ and thus $Eh$ does too. It follows then from \eqref{eq:better_non_standard_equation} that
\begin{equation}\label{eq:bootstrap}
E B(h) = E [\chi, B] \phi + E\left( \frac{h}{x - i 0}\right) \in H^{-1/2}.
\end{equation}
Indeed, $[\chi, B]$ is a differential operator with coefficients supported away from zero and thus $[\chi, B] \phi$ is smooth. Notice that $EB$ is a pseudodifferential operator of order $2$ elliptic on $\set{0} \times \mathbb{R}_+^*$. By elliptic regularity, it follows then from \eqref{eq:bootstrap} that $E h \in H^{3/2}(\mathbb{R})$. But now, formula \eqref{eq:bootstrap} implies that $EB(h) \in L^2(\mathbb{R})$, and thus that $Eh \in H^2(\mathbb{R})$. Iterating this argument, we end up finding that $Eh$ is smooth, and it follows that the wave front set of $h$ is contained in $\set{0} \times \mathbb{R}_-^*$. It implies that $\psi \in \Omega(c)$, and thus $\Omega(c)$ is non-trivial.
\end{proof}

\section{Further results}\label{section:further_results}

In this section, we explain how the proofs of Theorems \ref{theorem:limit} and \ref{theorem:resonances} may be adapted to prove related results. In \S \ref{subsection:circle_case}, we describe how to adapt the analysis above to study Orr--Sommerfeld equation on the circle (instead of the segment). In \S \ref{subsection:other_pertubrations}, we discuss the dependence of the set $\mathcal{R}$ from Theorem \ref{theorem:limit} on $U$ and $\alpha$.

\subsection{Orr--Sommerfeld equation on the circle}\label{subsection:circle_case}

If one starts from Navier--Stokes equation on a cylinder $\mathbb{S}^1 \times \mathbb{R}$ (instead of $[a,b] \times \mathbb{R}$), then the process described in the introduction produces Orr--Sommerfeld equation on a circle. Let $U$ be a $C^\infty$ function on a circle $\mathbb{S}^1 = \mathbb{R}/ L \mathbb{Z}$ for some $L > 0$. Let $\alpha > 0$. For $c \in \mathbb{C}$ and $\epsilon \geq 0$, we introduce the operator
\begin{equation*}
    P_{c,\epsilon} = i\alpha^{-1}\epsilon^2(\partial_x^2 - \alpha^2)^2 + (U-c)(\partial_x^2 - \alpha^2) - U''
\end{equation*}
as in the segment case. For $\epsilon > 0$, we let $\Sigma_\epsilon$ be the set of $c \in \mathbb{C}$ such that the operator $P_{c,\epsilon} : H^4(\mathbb{S}^1) \to L^2(\mathbb{S}^1)$ is not invertible. We define the notion of multiplicity for elements of $\Sigma_\epsilon$ as in the segment case. We will work hypothesis \ref{hypothesis} with the second assumption removed (it does not make sense in the absence of boundary points). We also remove the boundary condition in the definition of $\Omega(c)$ for $c \in \mathbb{R}$. With these modification, Theorem \ref{theorem:limit} and \ref{theorem:resonances} also hold in the circle case.

Let us explain briefly how to deduce the proofs of Theorems \ref{theorem:limit} and \ref{theorem:resonances} in the circle case from the proofs in the segment case. The ideas exposed in \S \ref{section:complex_scaling} are used in the same way, except that in the circle case there is no $M_{a,b}$ and we work directly on $M$. Concerning the definition of the function $m$ (and thus of the deformation $M$ of the circle), we require the same property as in \S \ref{section:escape_function}, except that we want them to hold on the full circle. 

Concerning the technical estimates from \S \ref{subsection:modified_operator_circle}, there is not much to be changed, the only thing is that we take $q = 0$, this function was just here in the segment case to control $P_{c,\epsilon}$ away from $[a,b]$. The subsection \S \ref{subsection:modified_operator_segment} has no analogue in the circle case, as this subsection mostly deals with boundary conditions issues (that do not appear in the circle case). The main estimate Proposition \ref{proposition:invertible_compact_perturbation} is already given in the circle case by Lemma \ref{lemma:invertibility_P_full_circle}. Notice in particular, that we do not need to assume $k \in ( - 3/2, - 1/2)$ to get uniform invertibility estimates as $\epsilon$ goes to $0$ in the circle case. 

There are very few changes in \S \ref{subsection:definition_resonances}, the main difference is that we can take any value of $k \in \mathbb{R}$ for the definition of resonances and in Lemma \ref{lemma:invertibility_a_priori}. However, the asymptotic expansion for $P_{c,\epsilon}^{-1}$ as $\epsilon$ goes to $0$ from \S \ref{subsection:LVmethod} is now easier to establish. Indeed, in the absence of boundary conditions to fit, we do not need the factor $I - N_{c,\epsilon}$ in the inductive scheme \eqref{eq:inductive_approximation} that produces the approximation for $P_{c,\epsilon}^{-1}$. It yields the explicit asymptotic expansion for the inverse
\begin{equation}\label{eq:approximate_inverse_circle}
    P_{c,\epsilon}^{-1} \underset{\epsilon \to 0}{\simeq} P_{c,0}^{-1} \sum_{m \geq 0} \epsilon^{2m} \left(-i\alpha^{-1}(\partial_x^2 - \alpha^2)^2 P_{c,0}^{-1}\right)^m.
\end{equation}
Notice the absence of odd power of $\epsilon$ in this expansion. This approximation is valid as described in Lemma \ref{lemma:description_approximation} in the segment case (except that there is no restriction on $k$ in the circle case): we need to see $P_{c,\epsilon}^{-1}$ as a deregularizing operator if we want to get a good approximation in terms of powers of $\epsilon$. The discussion from \S \ref{subsection:perturbation_resonances} is also simpler in the circle case, since we do not need to get rid of boundary layers in the projector $\Pi_{c_1,\epsilon}$.

Finally, the analysis from \S \ref{section:description_resonant_states} is mostly local, and thus there is no crucial difference between the circle and the segment case.

\begin{remark}
As mentioned above, in the circle case, there is no boundary layer in the approximation formula \eqref{eq:approximate_inverse_circle}. Consequently, the analogue of the projector $\Pi_{c_1,\epsilon}$ from \S \ref{subsection:perturbation_resonances} has a Taylor expansion in power of $\epsilon^2$ at $\epsilon = 0$ (if one accepts to lose some derivatives with each term in the expansion). Hence, the Taylor expansions from Theorem \ref{theorem:limit} in the circle case only contains even powers of $\epsilon$.

If $c_1 \in \mathcal{R}$ with multiplicity $1$, we can use this fact to derive a formula for the first term in the Taylor expansion for the perturbation of $c_1$ given by Theorem \ref{theorem:limit}. Let $\psi_0 \in \mathfrak{S}_{c_1,0}(\overline{M}_{a,b}) \setminus \set{0}$, then we can define for $\epsilon > 0$ small $\psi_\epsilon = (\partial_x^2 - \alpha^2)^{-1}\Pi_{c_1,\epsilon}(\partial_x^2 - \alpha^2) \psi_0$. Notice that there is a Taylor expansion in powers of $\epsilon^2$ for $\psi_\epsilon$, valid in $C^\infty(M)$. Moreover, if $c(\epsilon)$ is the element of $\Sigma_\epsilon$ close to $c_1$, then we have $P_{c(\epsilon),\epsilon} \psi_\epsilon = 0$. Writing $c(\epsilon) = c_1 + \epsilon^2 \tilde{c} + \mathcal{O}(\epsilon^4)$ and $\psi_\epsilon = \psi_0 + \epsilon^2 \phi + \mathcal{O}(\epsilon^4)$, and expanding $P_{c(\epsilon),\epsilon} \psi_\epsilon = 0$, we find that
\begin{equation*}
    i\alpha^{-1}(\partial_x^2 - \alpha^2)^2 \psi_0 - \tilde{c}(\partial_x^2 - \alpha^2) \psi_0 + P_{c_1,0} \phi = 0.
\end{equation*}
Integrating this formula on $M$ against the density $(U-c_1)^{-1} \psi_0 \mathrm{d}z$, we find
\begin{equation*}
    i\alpha^{-1} \int_M \frac{\psi_0}{U-c_1}(\partial_x^2 - \alpha^2)^2 \psi_0 \mathrm{d}z = \tilde{c} \int_{M} \frac{\psi_0}{U - c_1} (\partial_x^2 - \alpha^2) \psi_0 \mathrm{d}z = \tilde{c} \int_{M} U'' \left(\frac{\psi_0}{U-c_1}\right)^2 \mathrm{d}z.
\end{equation*}
Of course, we can shift contours in this formula (but a priori, we cannot replace the contour by $\mathbb{S}^1$ as integrability may fail). This formula allows to compute $\tilde{c}$ since the simplicity of $c_1$ implies that
\begin{equation*}
    \int_{M} U'' \left(\frac{\psi_0}{U-c_1}\right)^2 \mathrm{d}z \neq 0.
\end{equation*}
Let us prove this fact (the following argument is inspired by the proof of \cite[Lemma 1]{stepin_rayleigh}). Let us consider the lift $\widetilde{M}$ of $M$ to $\mathbb{C}$ by the canonical covering $\mathbb{C} \to \mathbb{C}/ L \mathbb{Z}$. Let $\tilde{\psi}_0,\widetilde{U}$ be the lifts respectively of $\psi_0$ and $U$ to $\widetilde{M}$. Let $\widetilde{P}_{c_1,0}$ be the lift of $P_{c_1,0}$ to $\widetilde{M}$ and consider a non-trivial solution $\tilde{\phi}$ to the equation  $\widetilde{P}_{c_1,0} \tilde{\phi} = (\partial_x^2 - \alpha^2) \tilde{\psi}_0$. Such a solution exists because $\widetilde{U} - c_1$ does not vanish on $\widetilde{M}$. Consider then a point $a \in \widetilde{M}$ such that $\tilde{\psi}_0(a) \neq 0$ and let $\gamma$ be the path from $a$ to $a + L$ along $\widetilde{M}$. We compute then
\begin{equation*}
    I \coloneqq \int_\gamma (\widetilde{P}_{c_1,0} \tilde{\phi}) \frac{\tilde{\psi}_0}{\widetilde{U} - c_1} \mathrm{d}z.
\end{equation*}
We use first the equation $\widetilde{P}_{c_1,0} \tilde{\phi} = (\partial_x^2 - \alpha^2) \tilde{\psi}_0 = \frac{\widetilde{U}''}{\widetilde{U}-c_1} \tilde{\psi}_0$ to find
\begin{equation*}
    I = \int_M U'' \left( \frac{\psi_0}{U - c_1} \right)^2 \mathrm{d}z.
\end{equation*}
On the other hand, integrations by parts yield
\begin{equation*}
    I = \psi_0(a)(\partial_x \tilde{\phi}(a+L) - \partial_x \tilde{\phi}(a)) - \partial_x \psi_0(a)(\tilde{\phi}(a+L) - \tilde{\phi}(a)).
\end{equation*}
Let now $\tilde{u}$ be a solution to $\widetilde{P}_{c_1,0} \tilde{u} = 0$ linearly independent of $\tilde{\psi}_0$. Notice that $\tilde{u}$ is not $L$ periodic (otherwise the corresponding function on $\widetilde{M}$ would be in the kernel of $P_{c_1,0}$, which would contradict the simplicity of $c_1$). Consequently, we must have $\tilde{u}(a) \neq \tilde{u}(a+L)$. Otherwise, since the Wronskian determinant of $\tilde{\psi}_0$ and $\tilde{u}$ is constant and $\tilde{\psi}_0$ is $L$ periodic, we would get that $\partial_x \tilde{u}(a) = \partial_x \tilde{u}(a+L)$, and thus $\tilde{u}$ would be periodic. Consequently, we may add a multiple of $\tilde{u}$ to $\tilde{\phi}$ to impose $\tilde{\phi}(a+L) = \tilde{\phi}(a)$. But then, $\tilde{\phi}(a)$ must be distinct from $\tilde{\phi}(a+L)$ (the simplicity of $c_1$ forbids $\tilde{\phi}$ to be $L$-periodic), and thus $I = \psi_0(a) (\tilde{\phi}(a+L) - \tilde{\phi}(a)) \neq 0$, which is what we wanted to prove.

\end{remark}

\subsection{Other kinds of perturbations}\label{subsection:other_pertubrations}

\begin{figure}[t]
   \centering
   \begin{subfigure}{0.45\textwidth}
       \centering
       \includegraphics[width=\textwidth]{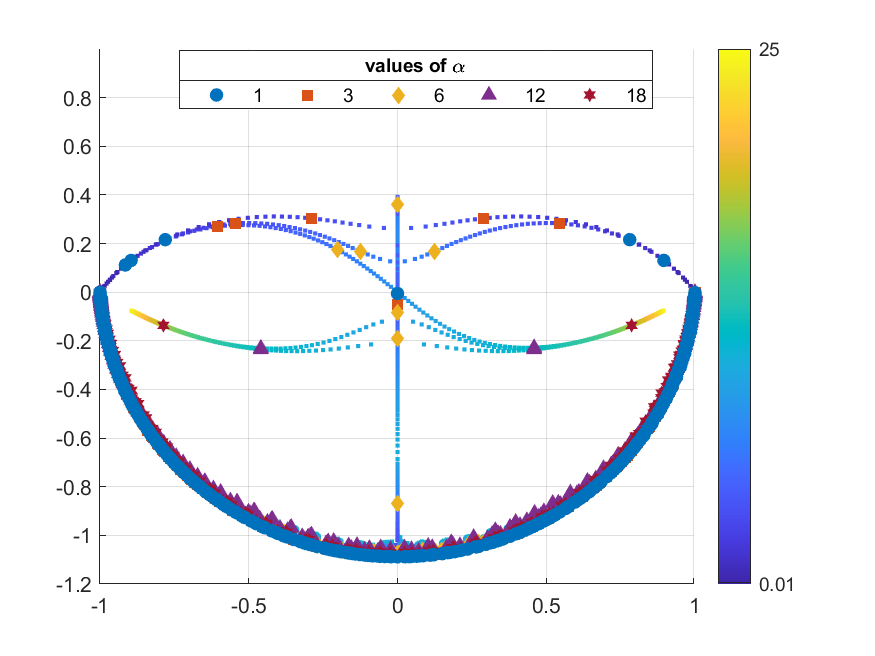}
       \caption{Resonances for a segment}
       \label{fig:sinusoid_segment_alpha}
   \end{subfigure}
   \begin{subfigure}{0.45\textwidth}
       \centering
       \includegraphics[width=\textwidth]{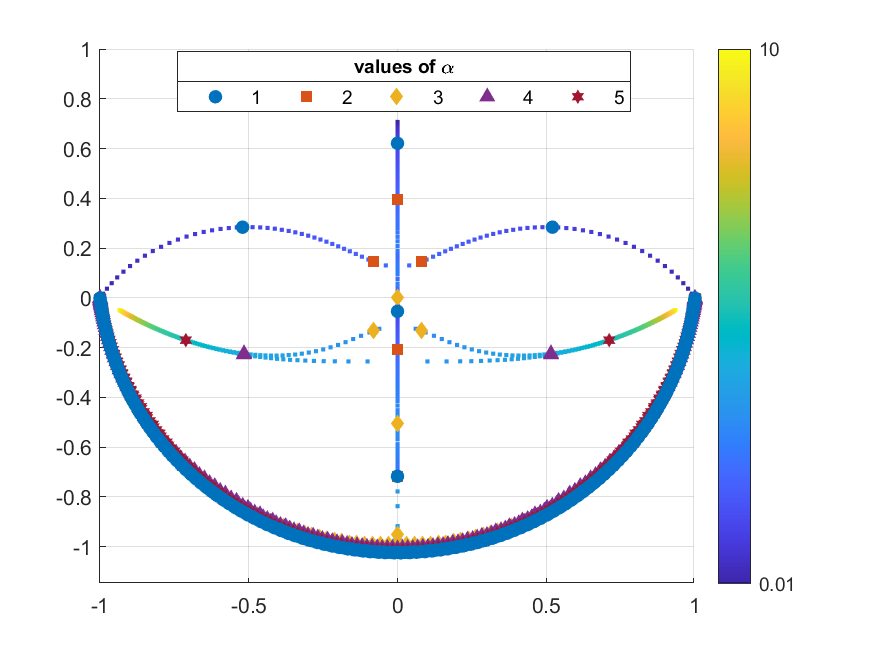}
       \caption{Resonances for a circle}
       \label{fig:sinusoid_circle_alpha}
   \end{subfigure}
   \caption{ Numerical computation of $\mathcal{R}_{\alpha}$ (there is no parameter $t$ in these examples). Markers and dots are colored according to values of $\alpha$, but using different colormaps. The lower curves represent the values of the parameter $c$ for which Rayleigh equation is not elliptic on the spaces defined by complex deformation (they do not belong to $\mathcal{R}_\alpha$).
   (A) $U(x)=\cos(3\pi x)$, $x\in [-1,1]$, with $\alpha$ varies from $0.01$ to 25.
   (B) $U(x)=\sin(3x)$, $x\in \mathbb R/2\pi\mathbb Z$, with $\alpha$ varies from $0.01$ to 10.
   }
   \label{fig:alpha}
\end{figure}

Theorem \ref{theorem:limit} associates to the family of operator $c \mapsto P_{c,0}$ a set of resonances $\mathcal{R}$. It is then natural to wonder how $\mathcal{R}$ depends on the parameters $U$ and $\alpha$ that appear in the definition of the operator $P_{c,0}$. Let us discuss this topic in the segment case, it is understood that the same thing could be done on the circle, with the approach from \S \ref{subsection:circle_case}.

Let $a < b$ be real number. We embed $[a,b]$ in a circle $\mathbb{S}^1 = \mathbb{R} / L \mathbb{Z}$ as in \S \ref{section:complex_scaling}. Let $N$ be a $C^\infty$ manifold. For each $t \in N$, let $U_t$ be a $C^\infty$ function from $[a,b]$ to $\mathbb{R}$. We assume that there is a compact subset $K$ of $(a,b)$, a real number $r > 0$ and for each $t$ an extension\footnote{The existence of these extensions is an artificial requirement that we make in order to spare the definition of an analogue of the space $\mathcal{B}_{K,r}$ on the segment. The main point here is that we control $U_t$ in $C^\infty$, and as an analytic functions near the points of $U_t^{-1}(\set{c_0})$.} of $U_t$ to $\mathbb{S}^1$ such that $t \mapsto U_t$ defines a $C^\infty$ map from $N$ to $\mathcal{B}_{K,r}$. We have then the following regularity results for resonances (see also Figure \ref{fig:alpha}).

\begin{theorem}\label{theorem:variation_resonances}
Let $\alpha_0 > 0, t_0 \in N$ and $c_0 \in \mathbb{R}$. Assume that $c_0$ is a regular value for $U_{t_0}$ and that $U_{t_0}^{-1}(\set{c_0})$ is a subset of $K$. Then, there is a neighbourhood $I$ of $\alpha_0$ in $\mathbb{R}$, a neighbourhood $V$ of $t_0$ in $N$ and $\delta_ 0 > 0$ such that:
\begin{itemize}
    \item for every $t \in V$ and $\alpha_1 \in I$, the $\delta$ in Theorem \ref{theorem:limit} applied with $U = U_{t}$ and $\alpha =\alpha_1$ may be chosen to be $\delta_0$, we let $\mathcal{R}_{t,\alpha_1} \subseteq (c_0 - \delta_0, c_0 + \delta_0) + i(- \delta_0, + \infty)$ denote the resulting set of resonances;
    \item the set $\mathcal{R}_{t,\alpha_1}$ depends smoothly\footnote{As $\Sigma_\epsilon$ does at $\epsilon = 0$ in Theorem \ref{theorem:limit}.} on $t$ and $\alpha_1$.
\end{itemize}
\end{theorem}

\begin{proof}[Sketch of proof of Theorem \ref{theorem:variation_resonances}.]
One just needs to notice that the constructions from the proof of Theorem \ref{theorem:limit} may be performed uniformly for $t$ close to $t_0$ and $\alpha_1$ close to $\alpha_0$. Then, the resonances are just the eigenvalues of the operator
\begin{equation*}
    A_{t,\alpha_1} = U_t - U''_t (\partial_x^2 - \alpha_1^2)^{-1}
\end{equation*}
on $\overline{M}_{a,b}$. Since $(c,t,\alpha_1) \mapsto A_{t,\alpha_1}-c$ is a smooth family, holomorphic in $c \in (c_0 - \delta_0,c_0 + \delta_0)+ i(- \delta,+ \infty)$, of elliptic operators (hence with compact resolvent), the result then follows from standard perturbation theory: we can use a spectral projection to reduce locally (in the spectral parameter) to the finite dimensional case \cite[Theorem 6.29 p.187 and IV.3.5]{kato_perturbation_theory} and end the proof by computing the determinant of a matrix, as in the proof of Theorem \ref{theorem:limit}.
\end{proof}

\begin{remark}
Theorem \ref{theorem:variation_resonances} only deals with variation of $\alpha$ within a compact subset of $(0, + \infty)$. However, it is tempting to consider the limits $\alpha \to 0$ and $\alpha \to + \infty$. Concerning the latter, notice that the complex deformation of $[a,b]$ defined in \S \ref{section:escape_function} does not depend on $\alpha$. Hence, in the setting of Theorem \ref{theorem:limit}, that fact that $c_0$ is a resonance or not can be checked by looking at the spectrum of the operator $U - U''(\partial_x^2 - \alpha^2)^{-1}$ on $L^2(M_{a,b})$ for instance. Since $U-c_0$ does not vanish on $M_{a,b}$, we find that $c_0$ is never a resonance when $\alpha$ is large. Indeed, using $h = \alpha^{-1}$ as a semiclassical parameter, we may prove that the operator norm of $(\partial_x^2 - \alpha^2)^{-1}$ on $L^2(M)$ is a $\mathcal{O}(\alpha^{-2})$ as $\alpha$ goes to $+ \infty$. Working then with the boundary layers $e^{\alpha(a-x)}$ and $e^{\alpha(x-b)}$, one find that the same estimate holds on $L^2(M_{a,b})$, and it follows that $U-c_0 - U''(\partial_x^2 - \alpha^2)^{-1}$ is invertible when $\alpha$ is large.
\end{remark}

\section{Examples}\label{section:examples}

In this section, we give some examples of profile $U$ to which Theorems \ref{theorem:limit} and \ref{theorem:resonances} apply, both on the segment (\S\S \ref{subsection:example_couette}--\ref{subsection:example_segment}) and the circle (\S \ref{subsection:example_circle}).

\subsection{Example with no resonance: Couette flow}\label{subsection:example_couette}

\begin{figure}[t]
   \centering
   \includegraphics[width=.8\textwidth]{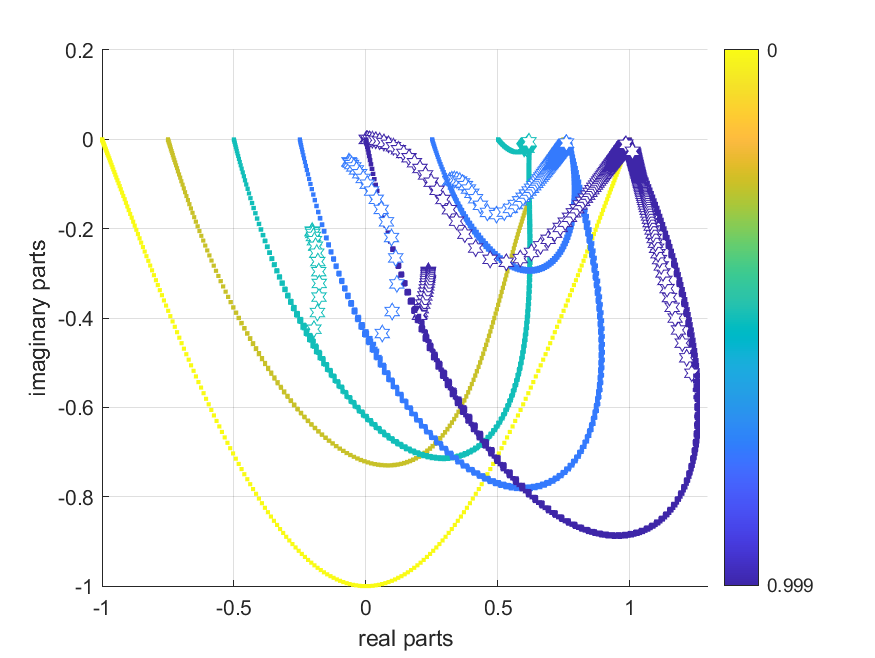}
   \caption{Couette--Poiseuille flows $U_{\theta}(x)=(1-\theta) x + \theta (1-x^2)$, $x\in [-1,1]$, $\theta\in [0,1]$.  Colored according to $\theta$. For each choice of $\theta$, snowflakes represent numerically computed resonances for $U_{\theta}$ with $\alpha$ varies from $10^{-5}$ to  $10$. Other points close to smooth curves represent the range of $U_{\theta}$ on the complex deformation (parameters for which the operator is not elliptic).
   }
   \label{fig:cp}
\end{figure}

The simplest example that we can think of is the Couette flow, given by the profile $U(x) = x$ on the segment $[-1,1]$. Let us fix $\alpha > 0$. We can apply Theorem \ref{theorem:limit} for any $c_0 \in \mathbb{R}\setminus \set{-1,1}$. Notice that there are no resonances in this case. Indeed, there was an element $c$ in the set $\mathcal{R}$ from Theorem \ref{theorem:limit}, then one could find a non-trivial solution to
\begin{equation*}
    (x - c) (\partial_x^2 \psi(x) - \alpha^2 \psi(x)) = 0, \ x \in M, \ \psi(-1) = \psi(1) = 0,
\end{equation*}
where $M$ is the complex deformation from the proof of Theorem \ref{theorem:limit}. Since $x-c$ do not vanish on $M$ (as follows from the proof of Theorem \ref{theorem:limit}, since $c \in \mathcal{R}$), we find that $\psi$ is a linear combination of $x \mapsto e^{\alpha x}$ and $x \mapsto e^{- \alpha x}$, and the boundary condition forces $\psi$ to be identically zero.

Hence, the set $\mathcal{R}$ from Theorem \ref{theorem:limit} is empty in this case. Notice that for Couette flow, global informations are known on the set $\Sigma_\epsilon$. Indeed, it is proven in \cite{romanov_couette} that there is a constant $C > 0$ that does not depend on $\epsilon$ and $\alpha$ such that for every $\epsilon > 0$ the set $\Sigma_\epsilon$ is contained in the half-plane $\set{ c \in  \mathbb{C}: \im c \leq - C \alpha^{-1} \epsilon^2}$. This estimates may be improved by working on a channel $[-1,1] \times \mathbb{S}^1$ instead of $[-1,1] \times \mathbb{R}$, see \cite{better_couette}. With Theorem \ref{theorem:limit}, we know in addition that as $\epsilon$ goes to $0$, the only places on the real axis where elements of $\Sigma_\epsilon$ can accumulate are at $1$ and $-1$.

Notice that Theorem \ref{theorem:variation_resonances} allow to consider small perturbation of Couette flow, such as the Couette--Poiseuille flow defined by $U(x) = p x + q(1- x^2)$ for some real numbers $p$ and $q$. The analysis above prove that if $p > 0$ and $K$ is a compact subset of $\mathbb{R} \setminus \set{ - p,p}$, then for $q$ small enough, the Couette--Poiseuille flow has no resonance in $K$.

\subsection{Examples on a segment}\label{subsection:example_segment}

\begin{figure}[t]
   \centering
   \begin{subfigure}{0.45\textwidth}
       \centering
       \includegraphics[width=\textwidth]{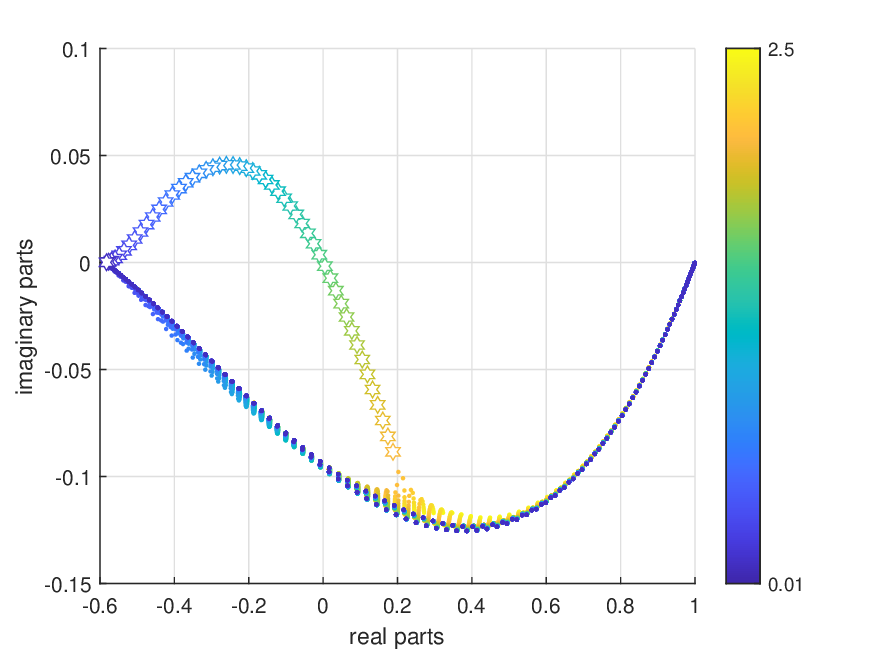}
       \caption{Resonances for $\alpha\in [0.01, 2.5]$}
       \label{fig:cos_07pi_alpha}
   \end{subfigure}
   \begin{subfigure}{0.45\textwidth}
       \centering
       \includegraphics[width=\textwidth]{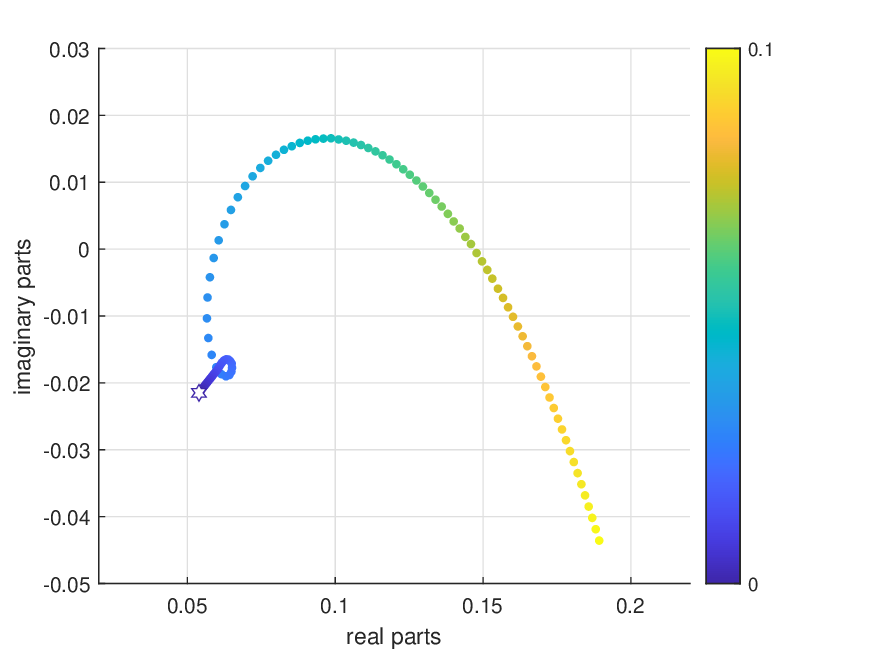}
       \caption{Viscous perturbation of a resonance}
       \label{fig:cos_07pi_045}
   \end{subfigure}
   \caption{ Shear profile $U(x)=\cos(0.7\pi x)$, $x\in [-1,1]$. (A) Numerical computation of $\mathcal{R}_{\alpha}$ (snowflakes, there is no parameter $t$ in this example). Colored according to values of $\alpha$. The lower curve represents the values of the parameter $c$ for which Rayleigh equation is not elliptic on the spaces defined by complex deformation.
   (B) Numerically computed viscous perturbations of the resonance (snowflake near $0.05-0.02i$) for $\alpha=\sqrt{0.7^2-0.45^2}\pi\approx 1.68$ and $\epsilon\in [0, 0.1]$ predicted by Theorem \ref{theorem:limit}. For some small $\epsilon$ (numerically, $0.035<\epsilon<0.07$), the resonance seems to become an unstable eigenvalue. Colored according to $\epsilon$.
   }
   \label{fig:cos07_alpha}
\end{figure}

Consider shear profiles
\[ U_{\omega,\theta}(x) := \sin(\omega x + \theta), \ x\in [-1,1], \text{ with } \omega, \theta\in \mathbb R, \ \pm \omega+\theta\notin \pi\mathbb Z. \]
The Rayleigh equation for $c = 0$ then reads
\[ U_{\omega,\theta}(\partial_x^2-\alpha^2)\psi-U_{\omega,\theta}''\psi = U_{\omega,\theta}(\partial_x^2+\omega^2-\alpha^2)\psi=0, \ \psi(-1)=\psi(1)=0. \]
Notice that $U_{\omega,\theta}$ is analytic and that $0$ is a regular value of $U_{\omega,\theta}$ (so that Theorem \ref{theorem:limit} applies). Moreover, if $x \in [-1,1]$ is such that $U_{\omega,\theta}(x) = 0$ then $U_{\omega,\theta}''(x) = \omega^2 U_{\omega,\theta}(x) = 0$. It follows then from Lemma \ref{lemma:inflexion_point} that the elements of $\Omega(0)$ are analytic. Consequently, $0$ is a resonance for Rayleigh equation in this case if and only if the equation 
\begin{equation}\label{eq:simple_rayleigh} 
(\partial_x^2+\omega^2-\alpha^2)\psi=0, \ \psi(-1)=\psi(1)=0 
\end{equation}
has an analytic non-trivial solution on $[-1,1]$. It happens if and only if $\omega^2 - \alpha^2 = \frac{\pi^2 k^2}{4}$ for some non-zero integer $k$, in which case the solutions to \eqref{eq:simple_rayleigh} are spanned by $x \mapsto \sin(\frac{k \pi}{2} x)$ (if $k$ is even) or $x \mapsto \cos(\frac{k \pi}{2} x)$ (if $k$ is odd).

It follows then from Theorem \ref{theorem:resonances} that $0$ is a resonance for the profile $U_{\omega,\theta}$ if and only if $\omega^2 - \alpha^2 = \frac{\pi^2 k^2}{4}$ for some non-zero integer $k$. Moreover, when this condition is met, the resonance zero is simple. By contradiction, if $0$ was not a simple resonance, then, according to Theorem \ref{theorem:limit} there would be an analytic function $\phi$ such that
\begin{equation*} 
U_{\omega,\theta}(\partial_x^2+\omega^2-\alpha^2)\phi=(\partial_x^2 - \alpha^2)\psi, \ \phi(-1)=\phi(1)=0 ,
\end{equation*}
where $\psi$ is a non-trivial solution to \eqref{eq:simple_rayleigh}. The existence of $\phi$ implies in particular that $\psi$ vanishes whenever $U_{\omega,\theta}$ vanishes (because $(\partial_x^2 - \alpha^2) \psi = -\omega^2 \psi$). From the condition $\omega^2 - \alpha^2 = \frac{\pi^2 k^2}{4}$, we deduce that $\omega$ is strictly larger than $\pi k/2$, and thus $U_{\omega,\theta}$ vanishes at least $k$ times in $(-1,1)$. However, the function $\psi$ vanishes only $k-1$ times in $(-1,1)$, a contradiction. 

\begin{figure}[t]
   \centering
       \includegraphics[width=.8\textwidth]{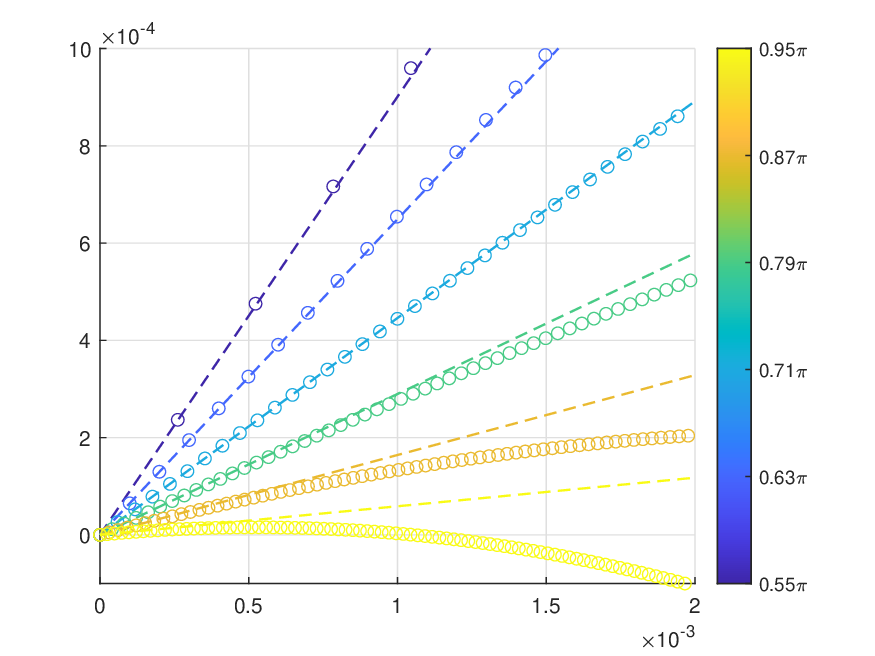}
   \caption{ Shear profile $U_{\omega,\frac{\pi}{2}}(x)=\cos(\omega x)$, $x\in [-1,1]$, $\alpha=\sqrt{\omega^2-\frac{\pi^2}{4}}$, $\omega\in (0.5\pi, \pi)$. Colored according to $\omega$. For each $\omega$ and the corresponding $\alpha$, $c=0$ is a resonance of the flow. Circles are numerically computed viscous perturbations of the resonance $c=0$. Dashed lines are the predicted first order perturbations.
   }
   \label{fig:first_order}
\end{figure}

A particularly interesting family of profiles are $U_{\omega,\frac{\pi}{2}}(x)=\cos(\omega x)$ for $\omega\in (\frac{\pi}{2}, \pi)$ and $\alpha = \sqrt{\omega^2-\frac{\pi^2}{4}}$. In these cases, $c_1=0$ is a simple resonance of the shear flow with a resonant state $\psi_0(x)=\cos(\frac{\pi x}{2})$. Let $c(\epsilon)$ be the viscous perturbation of $c_1$, then by Proposition \ref{proposition:first_order_perturbation}, the first order approximation of $c(\epsilon)$ is given by $c(\epsilon) = \dot{c}(0)\epsilon + \mathcal O(\epsilon^2)$ where 
\[\begin{split} 
\dot{c}(0) = & \frac{\pi^2 e^{\frac{\pi i}{4}}}{2\omega^2 \sqrt{\alpha |\cos(\omega)|}} \left( \int_{M_{-1,1}} \frac{ (\cos(\frac{\pi x}{2}))^2}{ \cos(\omega x)} \mathrm{d}x \right)^{-1} \\
= & \frac{ \pi^2 e^{\frac{\pi i}{4}} }{2\omega^2 \sqrt{\alpha|\cos(\omega)|} } \left( \mathrm{p.v.} \int_{-1}^1 \frac{(\cos(\frac{\pi x}{2}))^2}{\cos(\omega x)} \mathrm{d}x + \frac{2\pi i}{\omega} \left( \cos\left( \frac{\pi^2}{4\omega}\right) \right)^2 \right)^{-1}.
\end{split}\]
The comparison between the numerically computed viscous perturbations and our first order approximation prediction is exhibited in Figure \ref{fig:first_order}. In particular, we see that the imaginary parts of the viscous perturbations are positive for $\omega\in (\frac{\pi}{2},\pi)$.

\subsection{Examples on the circle}\label{subsection:example_circle}

Let us now discuss examples on the circle $\mathbb{R}/ 2 \pi \mathbb{Z}$. An important family of periodic flows are Kolmogorov flows given by 
\[ U_k(x)=\sin(kx), \ k>0, \ k\in \mathbb Z.\]
Putting $c=0$ in the Rayleigh equation for $U_k$ and we have 
\[ \sin(kx)\left( \psi''+(k^2-\alpha^2)\psi\right)=0. \]
This shows that 
\begin{enumerate}
    \item[1.] If $\alpha=k$, then $c=0$ is a simple resonance for the Rayleigh equation and the associated solutions to Rayleigh equation are constant functions;
    \item[2.] If $\alpha=\sqrt{k^2-\ell^2}$ for $1\leq \ell\leq k-1$, $\ell\in \mathbb Z$, then $c=0$ is a resonance for the Rayleigh equation with multiplicity $2$ and the associated solutions are the linear combinations of $x \mapsto e^{\pm i\ell x}$.
\end{enumerate}
As in the segment case, we can prove that there is no extra multiplicity due to ``generalized resonant states'' by comparing the number of zeros of $U_k$ and of the solutions to Rayleigh equation. Inviscid Kolmogorov flows are known to be unstable (meaning that for all $k$, there exists $\alpha$ such that the Rayleigh equation has eigenvalues with positive imaginary parts). A detailed study on the unstable spectrum can be found in \cite{belenkaya_friedlander_yudovich_1999}. 

\appendix

\section{FBI transform approach to complex deformation}\label{section:FBI_transform}

In a previous version of this work, instead of the complex deformation method, we were using an FBI transform based approach, in the spirit of \cite{FBI_bonthonneau_jezequel,galkowski_zworski_viscosity,galkowski_zworski_hypoellipticity,bonthonneau_guillarmou_jezequel, asymptotically_hyperbolic_jezequel}, which are themselves inspired by the work of Helffer and Sj\"ostrand \cite{helffer_sjostrand,sjostrand96}. It was already noticed in \cite[Appendix B]{galkowski_zworski_viscosity} that the simpler complex deformation method can sometimes be used to replace the FBI transform approach from the references above. In loose terms, this replacement is possible when ``the escape function is linear in $\xi$''. In our context, this condition is met, but we found it more convenient to keep working with FBI transform, until we noticed that the spaces defined by the two methods are actually the same. In the current version of this work, all references to FBI transform have been removed from the core of the paper. However, we thought that it would be interesting to sketch the proof of the equivalence of the two approaches. Indeed, we think that in some situation it could be useful to think of the complex deformation method in terms of FBI transform.

We will work here on the circle, as in \S \ref{subsection:generalities}, and rely on the exposition from \cite{galkowski_zworski_viscosity} (working on more general closed manifolds than torus may be done using \cite{FBI_bonthonneau_jezequel}, one may also work on a segment using some results from \cite{asymptotically_hyperbolic_jezequel}). For the sake of simplicity, we will work on a circle $\mathbb{S}^1 = \mathbb{R} / 2 \pi \mathbb{Z}$ of length $2 \pi$. The FBI transform that we will use is a map $T : \mathcal{D}'(\mathbb{S}^1) \to C^\infty(T^* \mathbb{S}^1)$. It is defined using the kernel
\begin{equation*}
    K_T(x,\xi,y) = h^{-\frac{3}{4}} \sum_{k \in  \mathbb{Z}} e^{\frac{i \Phi(x,\xi, y - 2 \pi k)}{h}} \langle \xi \rangle^{\frac{1}{4}}, \ (x,\xi,y) \in \mathbb{S}^1 \times \mathbb{R} \times \mathbb{S}^1.
\end{equation*}
In this formula, $h$ is a small implicit parameter and the phase $\Phi$ is defined by the formula
\begin{equation*}
    \Phi(x,\xi,y) = (x -y)\xi + \frac{i}{2}\langle \xi \rangle (x-y)^2
\end{equation*}
Moreover, the Japanese bracket is defined by $\langle \xi \rangle = \sqrt{1 + \xi^2}$. For $u \in \mathcal{D}'(\mathbb{S}^1)$, the map $Tu$ is defined by
\begin{equation*}
    Tu(x,\xi) = \int_{\mathbb{S}^1} K_T(x,\xi,y) u(y) \mathrm{d}y \textup{ for } (x,\xi) \in T^* \mathbb{S}^1 \simeq \mathbb{S}^1 \times \mathbb{R}.
\end{equation*}
Since the kernel $K_T$ is analytic in $y$, we see that the definition of $Tu$ above actually makes sense if $u \in \mathcal{B}_{\mathbb{S}^1,r}'$ for $r > 0$ small enough (the space $\mathcal{B}_{\mathbb{S}^1,r}'$ is defined in \S \ref{subsection:generalities}). Moreover, $K_T$ is also analytic in $(x,\xi)$, and $Tu$ extends consequently as a holomorphic function on a complex neighbourhood of $T^* \mathbb{S}^1$ in $(\mathbb{C}/ 2 \pi \mathbb{Z}) \times \mathbb{C}$. Let then $G : T^* \mathbb{S}^1 \to \mathbb{R}$ be a small enough symbol of order $1$, that we will call an escape function, and define
\begin{equation*}
    \Lambda = \set{ (x + i \partial_{\xi}G(x,\xi) , \xi - i \partial_x G(x,\xi)) : x \in \mathbb{S}^1, \xi \in \mathbb{R}} \subseteq T^* (\mathbb{C}/ 2 \pi \mathbb{Z}) \simeq (\mathbb{C}/ 2 \pi \mathbb{Z}) \times \mathbb{C}.
\end{equation*}
We say that $\Lambda$ is an I-Lagrangian: it is Lagrangian for the imaginary part of the canonical complex symplectic form on $T^* (\mathbb{C}/ 2 \pi \mathbb{Z})$.  On $\Lambda$, we will use the function $H$ defined by $H(x,\xi) = G(\re x, \re \xi) - \xi \partial_{\xi}G(\re x,\re \xi)$. Now, if $r$ and $G$ are small enough, for $u \in \mathcal{B}_{\mathbb{S}^1,r}'$, the holomorphic function $Tu$ is defined on a neighbourhood of $\Lambda$, and we denote by $T_\Lambda u$ the restriction of $Tu$ to $\Lambda$. To $\Lambda$, we associate then the following scale of spaces
\begin{equation*}
    \mathcal{H}_\Lambda^k = \set{ u \in \mathcal{B}_{\mathbb{S}^1,r}' : T_\Lambda u \in L^2\left(\Lambda,e^{- \frac{2H}{h}} \langle |\xi| \rangle^{2k} \re(\mathrm{d}\xi\mathrm{d}x)\right)} \textup{ for } k \in \mathbb{R}.
\end{equation*}
Here, $\re(\mathrm{d}\xi \mathrm{d}x)$ is our notation for the real part of the canonical complex symplectic form on $T^* (\mathbb{C}/ 2 \pi \mathbb{Z})$. We refer to \cite{galkowski_zworski_viscosity,FBI_bonthonneau_jezequel} for the basic properties of these spaces.

Now, let us consider a smooth function $m_0$ from $\mathbb{S}^1$ to $\mathbb{R}$. For some small $\tau > 0$, we define the function $m = \tau m_0$ on $\mathbb{S}^1$ and the symbol $G(x,\xi) = \tau m_0(x) \xi$ on $T^* \mathbb{S}^1$. Using the symbol $G$, we define the scale of spaces $(\mathcal{H}_\Lambda^k)_{k \in \mathbb{R}}$ as above, while we use the function $m$ to define a deformation $M$ of the circle as in \S \ref{section:complex_scaling}. 

Let us fix a small $r > 0$, and recall from \S \ref{subsection:generalities} that for every $k \in \mathbb{R}$ we may see the Sobolev space $H^k(M)$ as a subspace of $\mathcal{B}_{\mathbb{S}^1,r}'$, provided $\tau$ is small enough. Since the space $\mathcal{H}_\Lambda^k$ is also a subspace of $\mathcal{B}_{\mathbb{S}^1,r}'$, it makes sense to compare these spaces. Our claim is that, for $\tau$ and $h$ small enough, the spaces $H^k(M)$ and $\mathcal{H}_\Lambda^k$ coincide for every $k \in \mathbb{R}$, with equivalent norm, uniformly as $h$ goes to $0$ if we work with the semiclassical Sobolev spaces on $M$.

In order to justify this claim, let us start by pointing out that 
\begin{equation}\label{eq:identification_cotangent}
\begin{split}
    \Lambda & = \set{ (x + i m(x), (1 - i m'(x)) \xi) : (x,\xi) \in T^* \mathbb{S}^1} \\
            & = \set{ (x,\xi) \in T^* (\mathbb{C}/ 2 \pi \mathbb{Z}) : x \in M, \forall v \in T_x M, \xi v \in \mathbb{R}}.
\end{split}
\end{equation}
Hence, $\Lambda$ identifies with the cotangent bundle of $M$. Consequently, we may use the transform $T_\Lambda$ to define a FBI transform on $M$ in the following way. For $u \in \mathcal{D}'(M)$, we see it as an element of $\mathcal{B}_{\mathbb{S}^1,r}'$ as in \S \ref{subsection:generalities} and compute $T_\Lambda u$, which is a function on $\Lambda$. Using the identification $\Lambda \simeq T^* M$, the transform $T_\Lambda u$ defines a function $\mathcal{T}u$ on $T^* M$. Using the explicit formula for $T$ given above and the injection of $\mathcal{D}'(M)$ in $\mathcal{B}_{\mathbb{S}^1,r}'$ given by \eqref{eq:identification_hyperfunctions} we find that $\mathcal{T}$ is a ($C^\infty$) FBI transform on $M$, in the sense\footnote{There is a small difference: in \cite{wunsch_zworski_fbi}, there is a cut-off near the diagonal $\set{x = y}$ in the definition of an FBI transform. We do not need it here since the kernel is exponentially small as a smooth function away from the diagonal. We could introduce this cut-off artificially to fit exactly in \cite[Definition 2.4]{wunsch_zworski_fbi}, which would only change the transform $\mathcal{T}$ by a negligible operator.} of \cite[Definition 2.4]{wunsch_zworski_fbi}. The main point to be checked is that the phase $\Phi$ is admissible (\cite[Definition 2.1]{wunsch_zworski_fbi}). Most of the definition is easily verified, but the crucial point is that, because of \eqref{eq:identification_cotangent}, for every $x \in M$, the function $y \mapsto \im \Phi(x,\xi,y)$, well-defined on a neighbourhood of $x$, vanishes at order $2$ at $y = x$. Moreover, if $\tau$ is small enough, then the Hessian of this map is positive definite, which proves the second and the fourth point in \cite[Definition 2.1]{wunsch_zworski_fbi}.

Hence, it follows from \cite[Proposition 3.1]{wunsch_zworski_fbi} that $T_\Lambda$ is bounded from $H^k(M)$ to $L^2(\Lambda,\langle |\xi| \rangle^{2k} \re(\mathrm{d}\xi\mathrm{d}x))$. Notice that the fact that $G$ is linear in $\xi$ also implies that $H = 0$, and thus we find that $H^k(M)$ is contained in $\mathcal{H}_\Lambda^k(M)$, with continuous inclusion.

To prove the reciprocal inclusion, let us introduce the inverse transform $S_\Lambda$ defined by
\begin{equation*}
    S_\Lambda v(x) = \int_{\Lambda} K_S(x,y,\xi) v(y,\xi) \mathrm{d}y \mathrm{d}\xi
\end{equation*}
where the kernel $K_S$ is defined by
\begin{equation*}
   K_S (x,y,\xi) = \frac{h^{- \frac{3}{4}}}{(2 \pi)^{\frac{3}{2}}}  \sum_{k \in \mathbb{Z}} e^{ - \frac{i}{h} \Phi^*(x - 2 \pi k,y,\xi)} \left( 1 + \frac{i}{2} \left\langle x - y - 2 \pi k, \frac{\xi}{\langle \xi \rangle} \right\rangle \right)\langle \xi \rangle^{\frac{1}{4}},
\end{equation*}
where
\begin{equation*}
    \Phi^*(x,y,\xi) = (x-y) \xi - \frac{i}{2}\langle \xi \rangle (x-y)^2.
\end{equation*}
The transform $S_\Lambda v$ is a priori defined when $v$ is a rapidly decaying function on $\Lambda$. However, $S_\Lambda$ extends as a bounded map from $L^2(\Lambda,\langle |\xi| \rangle^{2k} \re(\mathrm{d}\xi\mathrm{d}x))$ to $\mathcal{H}_\Lambda^k$, and we have the relation $S_\Lambda T_\Lambda = I$.

Using the identification $\Lambda \simeq T^* M$, the transform $S_\Lambda$ defines a map $\mathcal{S}$ from $C_c^\infty(T^* M)$ to $C^\infty(M)$. As above, we find that the adjoint of $S$ is a $C^\infty$ FBI transform, and it follows from \cite[Proposition 3.1]{wunsch_zworski_fbi} that $S_\Lambda$ maps $L^2(\Lambda,\langle |\xi| \rangle^{2k} \re(\mathrm{d}\xi\mathrm{d}x))$ inside $H^k(M)$. Hence, if $u \in \mathcal{H}_\Lambda^k$, then $u = S_\Lambda T_\Lambda u \in H^k(M)$. Thus, we proved that $\mathcal{H}_\Lambda^k(M) = H^k(M)$.

One reason for which we think that it is interesting to know that, under the condition above, the spaces defined using FBI transform and complex deformation are the same is that, since the FBI transform construction is more general, there are some manipulations that are possible in this setting that could also be useful when working with complex deformations. For instance, with the notations above, if one wants to ``deform'' $L^2(M)$ into $L^2(\mathbb{S}^1)$, a natural way to do it is to deform $M$ continuously into $\mathbb{S}^1$. But there is another way to do it, which is to see $L^2(M)$ as $\mathcal{H}_\Lambda^0$ and then to deform the escape function $G$ into $0$ (maybe losing the linearity in $\xi$ along the way). This can be done for instance in such a way that $G$ is only changed within a compact set at any finite time of the deformation (except the final one), so that the space $\mathcal{H}_\Lambda^k$ does not change along the deformation, only its norm does. A similar deformation argument is used in \cite{galkowski_zworski_hypoellipticity}.

\section{Matlab programs}\label{section:matlab}

\subsection{Orr--Sommerfeld equation on a segment}\label{subsection:channel_matlab}
Let us recall the numerical discretization for fourth order boundary value problems designed in \cite[\S 14]{trefethen_spectral}.
For $N>0$, let 
$$x_j = \cos(j\pi/N), \ 0\leq j\leq N,$$ 
be Chebyshev points.
We define maps 
\[\begin{split}
& [\bullet]: H^1([-1,1])\to \mathbb C^{N+1}, \ [u]:=(u(x_0), u(x_1), \cdots, u(x_N))^T, \\
& \mathcal I: \mathbb C^{N+1}\to \mathbb P^{N}, \ \mathcal I \mathbf (v_0,v_2,\cdots, v_N)^T = \sum_{0\leq j\leq N}\left(\prod_{0\leq k\leq N, k\neq j}\frac{x-x_k}{x_j-x_k}\right) v_j.
\end{split}\]
Here $\mathbb P^{N}$ is the space of polynomials of order less than or equal to $N$.
For a differential operator $P$, the discretization of $P$, which we denote by $[P]$, is the matrix defined by 
\begin{equation*}
    [P][u]=[P \mathcal I[u]].
\end{equation*}
For two differential operators $P_1$, $P_2$, we also use approximation
\[ [P_1 P_2] \approx [P_1][P_2]. \] 

\begin{proposition}[{\cite[Theorem 7]{trefethen_spectral}}]
Define the discretization of a differential operator as above. Then
\begin{enumerate}
    \item[1.] Entries of the $(N+1)\times (N+1)$ matrix $[\partial_x]$ are
    \[\begin{split} 
        & [\partial_x]_{00}=\frac{2N^2+1}{6}, \ [\partial_x]_{NN} = -\frac{2N^2+1}{6}, \\ & [\partial_x]_{jj} = -\frac{x_j}{2(1-x_j^2)}, \ 1\leq j\leq N-1, \\
        & [\partial_x]_{jk} = \frac{(-1)^{j+k}c_j}{c_k(x_j-x_k)}, \ 0\leq j,k\leq N, \ j\neq k.
    \end{split}\]
    with $c_0=c_N=2$, $c_j=1$, $1\leq j\leq N-1$.
    
    \item[2.] Let $M_f$ be the multiplication operator by a function $f$, then 
    \[ [M_f] = \mathrm{diag}([f]).\]
\end{enumerate}
\end{proposition}

To implement the Dirichlet boundary condition, we introduce $\Pi_{\mathrm D}$ and its adjoint $\Pi_{\mathrm D}^*$ as follows
\[\begin{split}
    & \Pi_{\mathrm D}: \mathbb C^{N+1}\to \mathbb C^{N-1}, \ \Pi_{\mathrm D}(v_0, v_1, \cdots, v_N)^T = (v_1, \cdots, v_{N-1})^T, \\
    & \Pi_{\mathrm D}^*: \mathbb C^{N-1}\to \mathbb C^{N+1}, \ \Pi_{\mathrm D}^*(v_1,\cdots, v_{N})^T = (0,v_1, \cdots, v_{N-1},0)^T.
\end{split}\]
We represent a differential operator acting on functions satisfying Dirichlet conditions by
\[ [P]_{\mathrm D}:=\Pi_{\mathrm D}[P]\Pi_{\mathrm D}^*. \]
The definition of $[P]_{\mathrm D}$ is justified by the following equality 
\[ [P]_{\mathrm D}[u] = \Pi_{\mathrm D}[P \mathcal I \Pi_{\mathrm D}^* [u] ]. \]

Recall that we would like to solve the eigenvalue problem for the complex deformed Orr--Sommerfeld equation:
\begin{equation*}\begin{split}
    & \left(i\alpha^{-1}\epsilon^2 \left( D(\tau)^4-2\alpha^2 D(\tau)^2 +\alpha^4 \right) +U(\gamma_{\tau}(x))(D(\tau)^2-\alpha^2)-U''(\gamma_{\tau}(x)) \right) \phi \\
    & = c(D(\tau)^2-\alpha^2)\phi
\end{split}\end{equation*}
where 
\[ \gamma_{\tau}(x) = x+i\tau m_0(x), \ D(\tau) = (1+i\tau m_0'(x))^{-1}\partial_x.\]
Thus it suffices to compute
\begin{enumerate}
    \item[1.]  $[D(\tau)^2]_{\mathrm D}$ for $D(\tau)^2$ acting on functions satisfying the Dirichlet boundary condition;
    \item[2.] $[D(\tau)^4]_{\mathrm{DN}}$ for $D(\tau)^4$ acting on functions satisfying both Dirichlet and Neumann boundary conditions.
\end{enumerate}

The first matrix $[D(\tau)^2]_{\mathrm D}$ is straightforward to compute:
\[ [D(\tau)^2]_{\mathrm D} = \Pi_{\mathrm D}[D(\tau)]^2\Pi_{\mathrm D}^* = \Pi_{\mathrm D}\left( \mathrm{diag}( [(1+i\tau m_0(x))^{-1}]) [D] \right)^2 \Pi_{\mathrm D}^*. \]

\begin{figure}[t]
   \centering
   \begin{subfigure}{0.45\textwidth}
       \centering
       \includegraphics[width=\textwidth]{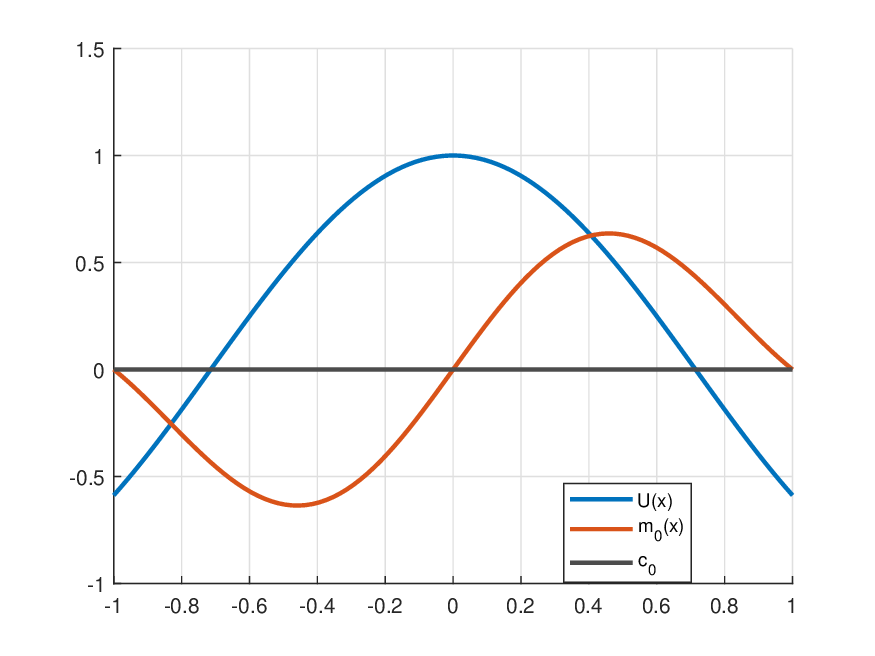}
       \caption{}
       \label{fig:cos_07pi_05_function}
   \end{subfigure}
   \begin{subfigure}{0.45\textwidth}
       \centering
       \includegraphics[width=\textwidth]{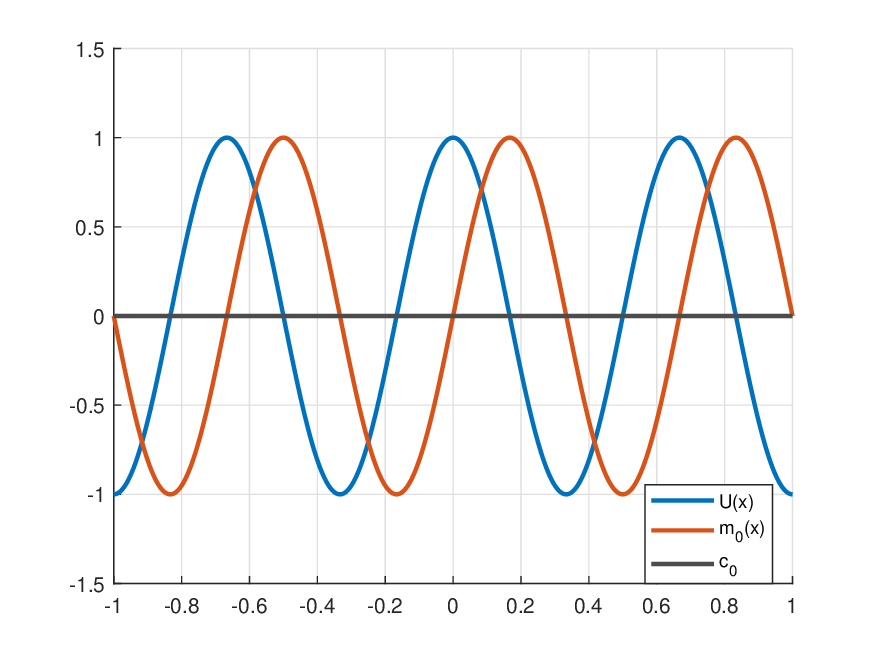}
       \caption{}
       \label{fig:sinusoid_segment_functions}
   \end{subfigure}\\
   \begin{subfigure}{0.45\textwidth}
       \centering
       \includegraphics[width=\textwidth]{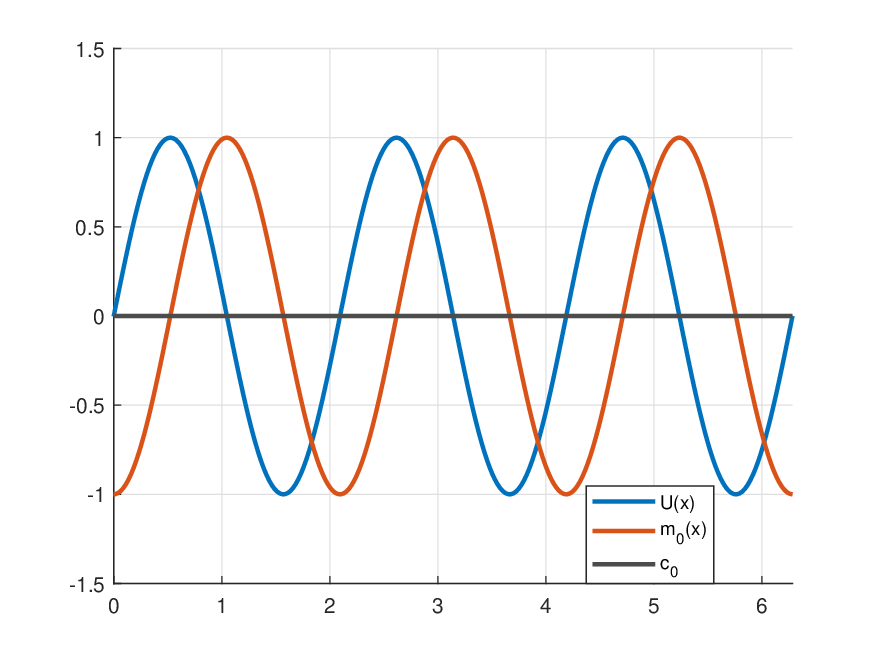}
       \caption{}
       \label{fig:sinusoid_circle_functions}
   \end{subfigure}
   \begin{subfigure}{0.45\textwidth}
       \centering
       \includegraphics[width=\textwidth]{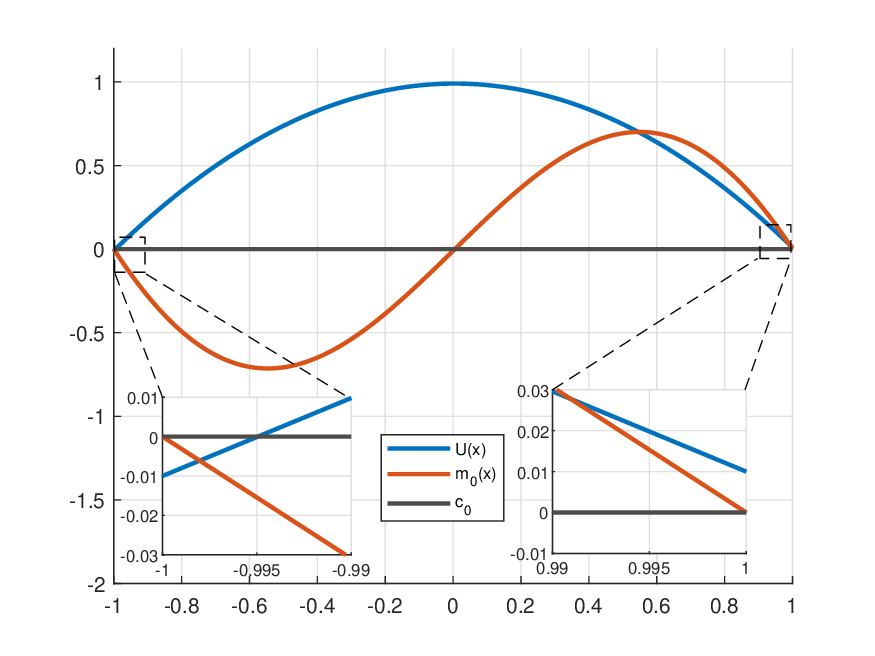}
       \caption{}
       \label{fig:cp_functions_wide}
   \end{subfigure}
   \caption{Shear profiles $U$ and the corresponding escape functions $m_0$. (A) $U(x)=\cos(0.7\pi x)$, $x\in [-1,1]$.  (B) $U(x)=\cos(3\pi x)$, $x\in [-1,1]$. (C) $U(x)=\sin(3x)$, $x\in \mathbb R/2\pi\mathbb Z$. (D) $U_{\theta}(x)=(1-\theta) x+\theta (1-x^2)$, $x\in [-1,1]$, $\theta=0.99$.}
   \label{fig:functions}
\end{figure}

Let us now consider $[D(\tau)^4]_{\mathrm{DN}}$. For a given function $u$ satisfying Dirichlet and Neumann boundary conditions, we use interpolation polynomials of the form $p(x)=(1-\gamma_{\tau}(x)^2)q(x)$, $q(\pm 1)=0$ to approximate $u$. Then 
\[ D(\tau)^4 p = (1-\gamma_{\tau}(x)^2) D(\tau)^4 q - 8\gamma_{\tau}(x)D(\tau)^3 q -12 D(\tau)^2q. \]
Notice that 
\[ q = \mathcal I \Pi_{\mathrm D} [(1-\gamma_{\tau}(x))^{-1}]\Pi_{\mathrm D}^* [u].\]
Thus the discretization of $D(\tau)^4$ with both Dirichlet and Neumann boundary conditions is
\[\begin{split} 
[D(\tau)^4]_{\mathrm{DN}} = & \Pi_{\mathrm D}\left( \mathrm{diag}([1-\gamma_{\tau}(x)^2])[D(\tau)]^4 - 8\mathrm{diag}([\gamma_{\tau}(x)])[D(\tau)]^3-12[D(\tau)]^2 \right)\Pi_{\mathrm D}^* \\
& \times \Pi_{\mathrm D} \mathrm{diag}( [(1-\gamma_{\tau}(x))^{-1}] )\Pi_{\mathrm D}^*.
\end{split}\]

In Figures \ref{fig:cos}, \ref{fig:cos07_alpha} and \ref{fig:first_order}, we considered $U(x)=\cos(0.7\pi x)$, $x\in [-1,1]$. We used the escape function\footnote{This function does not satisfy all the assumptions from \S \ref{section:escape_function}. However, it follows from Lemma~\ref{lemma:elliptic_extension} that it does not affect the spectrum of Rayleigh equation.} $m_0(x) = \sin(0.7\pi x)\cos(0.5\pi x)$. This choice of $m_0$ vanishes at $\pm 1$ and has correct signs at zeros of $U$. See Figure \ref{fig:cos_07pi_05_function}. In Figure \ref{fig:cos}, we set $\tau=0.1$.

In Figures \ref{fig:sinusoid_segment_tau} and \ref{fig:sinusoid_segment_alpha}, we consider $U(x)=\cos(3\pi x)$, $x\in [-1,1]$. We used the escape function $m_0(x) = \sin(3\pi x)$. See Figure \ref{fig:sinusoid_segment_functions}. In Figure \ref{fig:sinusoid_segment_alpha}, we set $\tau=0.1$.

In Figures \ref{fig:cp}, we considered the Couette--Poiseuille flow
\[ U_{\theta}(x):=(1-\theta)x+\theta(1-x^2), \ x\in [-1,1], \ \theta\in [0,1].\]
We used the following deformation
\begin{equation*}
 m_0(\theta,x) = \left( 2\theta x+\theta-1 \right)\cos(0.5\pi x). 
\end{equation*}
 See Figure \ref{fig:cp_functions_wide}. The value of $\tau$ in Figure \ref{fig:cp} is set to $1$. 

The following Matlab program, modified from \cite[Program 40]{trefethen_spectral}, computes the eigenvalue with the smallest modulus of the Orr--Sommerfeld equation for the shear profile $U(x)=\cos(\omega x)$ on $[-1,1]$ based on the discretization introduced above.

\begin{lstlisting}[frame=single,numbers=left,style=Matlab-Pyglike]
% Spectrum of Orr-Sommerfeld operator on the deformed segment with Dirichlet and Neumann boundary conditions

function OS_channel = OS_channel(eps,alp,omega,N,tau)

% Couette-Poiseuille flow and its second order derivative
U = @(x) cos(omega*x);
D2U = @(x) -omega^2*cos(omega*x);

% choice of escape function and its derivative
m_0 = @(x) sin(omega*x).*cos(.5*pi*x);
Dm_0 = @(x) omega*cos(omega*x)*cos(.5*pi*x)...
            -.5*pi*sin(omega*x).*sin(.5*pi*x);

% complex deformation and its derivative
gamma = @(x) x + 1i*tau*m_0(x); 
Dgamma = @(x) 1 + 1i*tau*Dm_0(x); 

% deformed the second D2_theta and the fourth order D4_theta differentiation matrix, with Dirichlet and Neumann boundary conditions
[D,x] = cheb(N); 
A = diag(1./Dgamma(x));
D_tau = A*D;
D2_tau = D_tau^2; D2_tau = D2_tau(2:N,2:N);
S_tau = diag([0; 1 ./(1-gamma(x(2:N)).^2); 0]);
D4_tau = (diag(1-gamma(x).^2)*D_tau^4 - 8*diag(gamma(x))*D_tau^3 - 12*D_tau^2)*S_tau;
D4_tau = D4_tau(2:N,2:N);

% deformed profile and its derivative
x = x(2:N);
U_tau = U(gamma(x));
D2U_tau = D2U(gamma(x));

% spectrum of the deformed operator
I = eye(N-1);
A = 1i*eps^2/alp*( D4_tau-2*alp^2*D2_tau+alp^4*I )...
    +diag(U_tau)*(D2_tau-alp^2*I)-diag(D2U_tau);
B = D2_tau-alp^2*I;
OS_channel = eigs(A,B,1,'smallestabs');

end

% Chebyshev differentiation matrix by Trefethen (2000)
function [D,x] = cheb(N)

if N==0, D = 0; x = 1; return, end
x = cos(pi*(0:N)/N)';
c = [2; ones(N-1,1); 2].*(-1).^(0:N)';
X = repmat(x,1,N+1);
dX = X-X';
D = (c*(1./c)')./(dX+(eye(N+1)));
D = D - diag(sum(D'));

end
\end{lstlisting}

\subsection{Orr--Sommerfeld equation on the circle}

In the circle case, we use Discrete Fourier transform (DFT) to discretize the operator. For an even integer $N>0$, let $\mathcal F_N$ be the DFT matrix 
\[ (\mathcal F_N)_{k\ell} = -\frac{e^{-\frac{2\pi i k\ell}{N}}}{\sqrt N}, \ -\frac{N}{2}\leq k\leq \frac{N}{2}-1, \ 0\leq \ell \leq N-1.\]
We also denote
\[\begin{split} 
& \{ \bullet \}: C^{\infty}(\mathbb S^1)\to \mathbb C^N, \ \{u\} = \left(u(0), u(\tfrac{2\pi}{N}), \cdots, u(\tfrac{2\pi(N-1)}{N}) \right)^T, \\
& \mathcal F^{-1}: \mathbb C^N\to C^{\infty}(\mathbb S^1), \ \mathcal F^{-1}\left(u_{-\frac{N}{2}},\cdots, u_{\frac{N}{2}-1} \right)^T = \sum_{-\frac{N}{2}\leq k\leq \frac{N}{2}-1} u_{k} e^{ikx}. 
\end{split}\]
We define the discretization of a differential $P$ through DFT by 
\[ \left([P]_{\mathcal F}\right) v:= \mathcal F_N \{ P\mathcal F^{-1} v \} \text{ for } v\in \mathbb C^N.  \]
For two differential operators $P_1$, $P_2$, as before, we use the approximation
\[ [P_1 P_2]_{\mathcal F} \approx [P_1]_{\mathcal F} [P_2]_{\mathcal F}.\]

Direct computations give
\begin{proposition}
Let $[P]_{\mathcal F}$ be defined as above. Then 
\begin{enumerate}
    \item[1.] $[\partial_x]_{\mathcal F} =  \mathrm{diag}( ik )_{-\frac{N}{2}\leq k\leq \frac{N}{2}-1}$;
    
    \item[2.] For a function $f$, let $M_f$ be the multiplication operator by $f$, then 
    \[ [M_f]_{\mathcal F} = \mathcal F_N \left(\mathrm{diag}( \{f\} )\right) \mathcal F_N^*. \]
    Here $\mathcal F_N^*$ is the conjugate transpose of $\mathcal F_N$.
\end{enumerate}
\end{proposition}
Using these facts, we can compute the matrix for 
\[ Q_{0,\epsilon}(\tau) = i\alpha^{-1}\epsilon^2(D(\tau)^2-\alpha^2) + U(\gamma_{\tau}(x)) - U''(\gamma_{\tau}(x))(D(\tau)^2-\alpha^2)^{-1} \]
where $\gamma_{\tau}$ and $D(\tau)$ are as in \S \ref{subsection:channel_matlab},
and compute its eigenvalues.

In Figure \ref{fig:sin3x}, \ref{fig:sinusoid_circle_tau}, and \ref{fig:sinusoid_circle_alpha}, we consider $U(x)=\sin(3x)$, $x\in \mathbb R/2\pi\mathbb Z$. In the numerical experiments, we used the escape function $m_0(x) = -\cos(3x)$. See Figure~\ref{fig:sinusoid_circle_functions}. In Figure~\ref{fig:sin3x}, the parameter $\tau$ is set to be $0.15$. 

The following Matlab program, modified from \cite[Appendix B]{galkowski_zworski_viscosity}, computes eigenvalues of the Orr--Sommerfeld equation on $\mathbb R/2\pi\mathbb Z$ for the profile $U(x)=\sin(3x)$.

\begin{lstlisting}[frame=single,numbers=left,style=Matlab-Pyglike]
% Spectrum of the complex deformed Orr-Sommerfeld on the circle

function OS_circle = OS_circle(eps,alp,N,tau)

% periodic shear profile and its second derivative
a = 3; U = @(x) sin(a*x); 
D2U = @(x) -a^2*sin(a*x);

% choice of escape function and its derivative
m_0 = @(x) -cos(a*x); 
Dm_0 = @(x) a*sin(a*x); 

% complex deformation and its derivative
gamma = @(x) x + 1i*tau*m_0(x);
Dgamma = @(x) 1 + 1i*tau*Dm_0(x);

% deformed Fourier differentiation matrix
x = (2*pi/N*(0:N-1))';
F = fft(eye(N))/sqrt(N);
D = diag([0:N/2-1,-N/2:-1]);
A = diag(1./Dgamma(x));
D = F*A*F'*D;

% inverse of \partial_x^2-\alpha^2
Lap = -D^2-alp^2*eye(N); 
Lap_inv = Lap^(-1);

% complex deformed profile and multiplication operator
U_tau = diag(U(gamma(x))); U_tau = F*U_tau*F';
D2U_tau = diag(D2U(gamma(x))); D2U_tau = F*D2U_tau*F';
    
% spectrum of the deformed operator
Q_tau = 1i*eps^2/alp*Lap + U_tau  - D2U_tau*Lap_inv;
OS_circle = eig(Q_tau);

end
\end{lstlisting}

\bibliographystyle{alpha}
\bibliography{biblio.bib}

\end{document}